\documentclass[11pt]{amsart}
\usepackage{amsmath,amsthm,amsfonts,amssymb,mathrsfs}
\usepackage[dvipsnames]{xcolor}
\usepackage[colorlinks=true, citecolor=blue, anchorcolor=red]{hyperref}
\usepackage{enumerate}
\usepackage{tikz}
\usepackage{leftidx}

\newcommand{\supp}{\operatorname{supp}}
\newcommand{\R}{\mathbb R}
\newcommand{\N}{\mathbb N}
\newcommand{\C}{\mathbb C}
\newcommand{\Z}{\mathbb Z}
\newcommand{\E}{\mathbb E}
\newcommand{\F}{\mathbb F}

\newcommand{\X}{\mathbb X}
\renewcommand{\Re}{\mathop{\text{\upshape{Re}}}}
\renewcommand{\Im}{\mathop{\text{\upshape{Im}}}}
\newcommand{\tr}{\mathop{\text{\upshape{tr}}}}

\newcommand{\RR}{\mathcal{R}}

\newcommand{\sign}{\mathop{\textrm{sign}}}
\newcommand{\rad}{\mathop{\textrm{\upshape{Rad}}}\nolimits}
\newcommand{\spur}{\mathop{\textrm{\upshape{tr}}}}
\newcommand{\sym}{\R^{n\times n}_{\textrm{sym}}}

\newcommand{\mreg}{\mathop{\text{\upshape{MR}}}}
\newcommand{\buc}{\mathop{\text{\upshape{BUC}}}}
\newcommand{\Isom}{\mathop{\text{\upshape{Isom}}}}
\newcommand{\HT}{\mathop{\text{\upshape{HT}}}}
\newcommand{\op}{\mathop{\text{\upshape{op}}}}

\newcommand{\cal}[1]{{\mathcal #1}}
\renewcommand{\Im}{\operatorname{Im}}

\renewcommand{\epsilon}{\varepsilon}
\renewcommand{\bar}[1]{\overline{#1}}
\newcommand{\id}{\mathop{\textrm{id}}\nolimits}
\newcommand{\norm}{|\!|\!|}

\renewcommand{\tilde}{\widetilde}
\renewcommand{\phi}{\varphi}
\renewcommand{\div}{\mathop{\text{\upshape{div}}}}
\DeclareMathOperator{\sgn}{sgn}
\newcommand{\co}{\mathop{\text{\upshape{conv}}}}
\newcommand{\aco}{\mathop{\text{\upshape{aconv}}}}
\newcommand{\symb}{\mathop{\textrm{symb}}}

\newcommand{\MRO}{\leftidx{_0}{\!\mreg}{_p}}

\newtheorem{theorem}{Theorem}[section]
\newtheorem{lemma}[theorem]{Lemma}
\newtheorem{corollary}[theorem]{Corollary}

\theoremstyle{definition}
\newtheorem{definition}[theorem]{Definition}
\newtheorem{remark}[theorem]{Remark}
\newtheorem{example}[theorem]{Example}

\numberwithin{equation}{section}

\allowdisplaybreaks

\begin{document}

\title[Maximal regularity for parabolic evolution equations]
{An introduction to maximal regularity\\ for parabolic evolution equations}
\author{Robert Denk}
\address{University of  Konstanz, Department of Mathematics and Statistics,
         78457 Konstanz, Germany}
\email{robert.denk@uni-konstanz.de}

\begin{abstract}
In this note, we give an introduction to the concept  of maximal $L^p$-regularity as a method to solve nonlinear
partial differential equations. We first define maximal regularity for autonomous and non-autonomous problems and
describe the connection to Fourier multipliers and $\RR$-boundedness.  The abstract results are applied to a large
class of parabolic systems in the whole space and to general parabolic boundary value problems. For this, both the
construction of solution operators for boundary value problems and a characterization of trace spaces of Sobolev
spaces are discussed. For the nonlinear equation, we obtain local in time well-posedness in appropriately chosen
Sobolev spaces. This manuscript is based on known results and consists of an extended version of lecture notes on
this topic.
\end{abstract}



\thanks{The author wants to express his gratitude to Hideo Kozono and Takayoshi Ogawa for the possibility to give a
lecture series with this topic at the Tohoku Forum for Creativity, Sendai, and for their great hospitality during his
stay at Tohoku University.}

\date{March 3, 2020}

 \maketitle

\tableofcontents

\section{Introduction}

In this survey, we give an introduction to the method of maximal $L^p$-regularity which has turned out to be useful for
the analysis of nonlinear (in particular, quasilinear) partial differential equations. The aim of this note is to
present an overview on the main ideas and tools for this approach. Therefore, we are not trying to present the state of
the art but restrict ourselves to relatively simple situations. At the same time, we focus on the mathematical
presentation and not on the historical development of this successful branch of analysis. So we do not give detailed
bibliographical remarks but refer to some nowadays standard literature, where more details on the history and on the
bibliography can be found. This survey could serve as a basis for an advanced lecture course in partial differential
equations, for instance for Ph.D. students. In fact, the present paper is based on a series of lectures given in July
2017  at the Tohoku University in Sendai, Japan, and on an advanced course for master students at the University of
Konstanz, Germany, in the summer term 2019.

Although the concept of maximal regularity is classical, some main achievements for the abstract theory were obtained
in the 1990's and in the first decade of the present century by, e.g., Amann (see \cite{Amann04}, \cite{Amann95}) and
Pr\"uss (see \cite{Pruess02}). The basic idea of maximal regularity is to solve nonlinear partial differential
equations by a linearization approach. Let us consider an abstract quasilinear equation of the form
\begin{equation}
  \label{1-1}
  \begin{aligned}
    \partial_t u(t) - A(u(t))u(t) & = F(u(t)),\\
    u(0) & = u_0.
  \end{aligned}
\end{equation}
The linearization of \eqref{1-1} at some fixed function $u$ is given by
\begin{equation}
  \label{1-2}
  \begin{aligned}
    \partial_t v(t) - A(u(t))v(t) & = F(u(t)),\\
    v(0) & = u_0.
\end{aligned}
\end{equation}
In the maximal regularity approach, one tries to solve the linear equation in appropriate function spaces and to show
that the solution has the optimal regularity one could expect. In this case, let $v =: S_u(F(u), u_0)$ denote the
($u$-dependent) solution operator  of the linear equation \eqref{1-2}. If $S_u$  induces an isomorphism between
appropriately chosen pairs of Banach spaces, then the solvability of the nonlinear equation \eqref{1-1} can be reduced
to a fixed-point equation of the form $u = S_u(F(u), u_0)$. In many situations, the contraction mapping principle can
be applied to obtain a unique solution of the fixed point equation and, consequently, of the nonlinear equation
\eqref{1-1}. In this way, typically short-time existence or existence for small data can be shown. For the long-time
asymptotics and the stability of the solution, different methods have to be used. Here, we mention the monograph by
Pr\"uss and Simonett \cite{Pruess-Simonett16}, which covers  the abstract theory of maximal regularity, stability
results, and many examples in fluid mechanics and geometry.

As mentioned above, one key ingredient in the maximal regularity approach is the choice of appropriate function spaces
for the right-hand sides and the solution of the nonlinear equation. In the present note, we restrict ourselves to the
$L^p$-setting, where the basic spaces are $L^p$-Sobolev spaces. (For maximal regularity in H\"older spaces, we mention
the monograph by Lunardi \cite{Lunardi95}.) Maximal $L^p$-regularity is closely related to the question of Fourier
multipliers, as we will see in Section~3 below. Therefore, it was a breakthrough for the application of this concept,
when an equivalent description for maximal regularity in terms of vector-valued Fourier multipliers and
$\RR$-sectoriality was found by Weis \cite{Weis01} in the year 2001.

The description of maximal $L^p$-regularity by $\RR$-boundedness made it possible to show that a large class of
parabolic boundary value problems have this property. As standard references for $\RR$-boundedness and applications to
partial differential operators, we mention \cite{Denk-Hieber-Pruess03} and \cite{Kunstmann-Weis04}. For boundary value
problems, also the question of appropriate function spaces on the boundary appears, which leads to the characterization
of trace spaces. Here the trace can be taken with respect to time (for the initial value at time 0) or with respect to
the space variable (for inhomogeneous boundary data). It turns out that the theory of trace spaces is highly nontrivial
and connected with interpolation properties of intersections of Sobolev spaces. In this way, modern theory of
vector-valued Sobolev spaces with non-integer order of  differentiability enters. Results on trace spaces can be found,
e.g., in \cite{Denk-Hieber-Pruess07}, for a survey on vector-valued Sobolev spaces we refer to \cite{Amann19} and
\cite{Hytonen-vanNeerven-Veraar-Weis16}.

The plan of the present survey follows the topics just mentioned. In Section~2, we state the idea and the formal
definition of maximal regularity, mentioning the graphical mean curvature flow as a prototype example. The connection
to vector-valued Fourier multipliers and $\RR$-boundedness is given in Section~3. In Section~4, we briefly summarize
the main definitions of the different types of (non-integer) Sobolev spaces and give some key references. The
application of the abstract concept to parabolic partial differential equations in the whole space is given in
Section~5, the application to parabolic boundary value problems in Section~6. Finally, we return to nonlinear evolution
equations in Section~7, where local well-posedness and higher regularity for the solution are discussed.

There are, of course, many topics in the context of maximal $L^p$-regularity which are not covered here. First, we want
to mention the application of maximal regularity to stochastic partial differential equations, which leads to the
notion of stochastic maximal regularity. Here, the class of radonifying operators plays an important role. A survey on
stochastic maximal regularity can be found, e.g., in \cite{VanNeerven-Veraar-Weis15}, for random sums and radonifying
operators see also \cite{Hytonen-vanNeerven-Veraar-Weis17}. Another development that could be mentioned is the maximal
$L^p$-regularity approach for boundary value problems which are not parabolic in a classical sense (as defined in
Sections~5 and 6 below). Some main applications are free boundary value problems from fluid mechanics or problems
describing phase transitions like the Stefan problem. Here, the related symbols are not quasi-homogeneous, and the
theory described below cannot be applied. One concept to show maximal $L^p$-regularity for such problems uses the
Newton polygon, and we refer to \cite{Denk-Kaip13} for more details.

\section{Maximal regularity and $L^p$-Sobolev spaces}

\subsection{Linearization and maximal regularity}

We start with an example of a quasilinear parabolic equation.

\begin{example}[Graphical mean curvature flow]
  \label{2.1}
  Let $T_0\in(0,\infty]$, let $M$ denote an  $n$-dimensional parameter space, and let  $X(t, \cdot) \colon
M  \to\R^{n+1} $, $t\in [0,T_0)$,  be a   family of regular maps. Here, regular means that the Jacobian $D_x X(t,x)$
with  respect to $x\in M$ is injective for all $x\in M$ and $t\in [0,T_0)$. We set $M_t := X(t,M )$.  Then the vectors
$\partial_{x_1} X(t,x), \dots, \partial_{x_n} X(t,x)$ form a basis for the tangent space $T_x M_t$ at the point
$X(t,x)$. In particular, we are interested in the graphical situation where $M  = \R^n$ (or some domain in $\R^n$) and
where $X$ is given as the graph of some function  $u\colon[0,T_0)\times\R^n\to\R$, so we have $X(t,x) = (x,u(t,x))$ for
$x\in\R^n$ and $t\in [0,T_0)$.

Let $\nu\colon [0,T_0)\times M \to\R^{n+1}$ be one choice of the normal vector  to $M_t$, so $\nu(t,x)$  is a unit
vector which  is orthogonal to the tangent space $T_xM$. For each $i=1,\dots,n$, the vector $\partial_{x_j}\nu(t,x)$ is
an element of $T_xM_t$, and therefore we can write
\[ \partial_{x_j}\nu(t,x) = \sum_{i=1}^n S_{ij}(t,x) \partial_{x_i}X(t,x).\]
The matrix $S(t,x) := (S_{ij}(t,x))_{i,j=1,\dots,n}$ is called the shape operator at the point $X(t,x)$,  its
eigenvalues are called  the principal curvatures, and its trace $H(t,x) := \tr S(t,x)$ is called the mean curvature.

The family of hypersurfaces $(M_t)_{t\in [0,T_0)}$ is said to move according to the mean curvature flow (see, e.g.,
\cite{Colding-Minicozzi-Pedersen15} for a survey) if
\[ \partial_t X (t,x)\cdot \nu(t,x) = - H(t,x)\nu(t,x)\quad ((t,x)\in [0,T_0)\times M^n).\]
In the graphical situation, one choice of the normal vector is given by
\[ \nu(t,x) = \frac{1}{\sqrt{1+|\nabla u|^2}}\,\binom{-\nabla u(t,x)}{1}.\]
From this, we obtain for the mean curvature
\[ H(t,x) = -\div \Big( \frac{ \nabla u(t,x)}{\sqrt{1+|\nabla u|^2}}\Big),\]
and  the equation for the graphical mean curvature flow  is given by
  \begin{equation}
    \label{2-1}
    \begin{aligned}
    \partial_t u - \Big(\Delta u -\sum_{i,j=1}^n\frac{\partial_i u \partial_j u}{1+|\nabla u|^2}\,\partial_i \partial_j
     u\Big)  & = 0 \quad \text{in } (0,T_0),\\
    u(0) & = u_0.
    \end{aligned}
  \end{equation}
Here, $u_0$ is the initial value at time $t=0$, so $M_0$ is given as the $X(0,\R^n)$ with $X(0,x)=(x,u_0(x))$. As the
coefficients of the second derivatives of $u$ depend on $u$ itself, this is an example of a quasilinear parabolic
equation.
\end{example}

The above example can be written in the abstract form
\begin{equation}\label{2-2}
  \begin{aligned}
    \partial_t u + F(u) u & = G(u),\\
    u(0) & = u_0,
  \end{aligned}
\end{equation}
where $F(u)$ is a linear operator depending on $u$ and $G(u)$ (which equals zero in the example) is, in general, some
nonlinear function  depending on $u$. For the  linearization of \eqref{2-2}, we fix some function $u$ and are looking
for a solution of the Cauchy problem
\begin{equation}
  \label{2-3}
  \begin{aligned}
    \partial_tv + F(u) v & = G(u),\\
    v(0) & = u_0.
  \end{aligned}
\end{equation}
Note that \eqref{2-3} is a linear equation with respect to $v$, and therefore it can be treated with methods from
linear operator  theory and semigroup theory. In general,  \eqref{2-3} is  a non-autonomous problem, as $u$ and
therefore also $F(u)$ still depend on time. Setting $A(t) := F(u(t))$ and $f(t):= G(u(t))$, we obtain
\begin{equation}
  \label{2-3a}
  \begin{aligned}
  \partial_t v(t) - A(t) v & = f(t) \quad (t>0),\\
  v(0) & = u_0.
\end{aligned}
\end{equation}

The idea of maximal regularity consists in showing ``optimal'' regularity for the linearized equation. Roughly
speaking, one should not  loose any regularity when solving the linear equation, as the solution will be inserted into
the equation in the next step of some iteration process. Considering \eqref{2-3a} in an operator theoretic sense, we
want to have good mapping properties of the solution operator who maps the right-hand side data $f$ and $u_0$ to the
solution $v$. For this, we have to fix function spaces for the right-hand side and the solution. So we have to choose
the basic space $\F$ for the right-hand side $f$ and a solution space $\E$ for $v$. The choice of the space $\gamma_t
\E$ for the initial value $u_0$ will then be canonical, see below.

In case of maximal regularity, we expect a unique solution of \eqref{2-3a} and a continuous solution operator $S_u$
(depending on $A(t)$ and  therefore on $u$)
\[ S_u\colon \F\times \gamma_t \E \to \E,\; (f,u_0)\mapsto v\]
of the linear equation \eqref{2-3a}. Then the nonlinear Cauchy problem is uniquely solvable if and only if the fixed
point  equation
\[ u = S_u(G(u),u_0)\]
has a unique solution $u\in\E$.

In many cases, one can show that the right-hand side of this fixed point equation defines a contraction, and therefore
Banach's fixed point  theorem (contraction mapping principle) gives a unique solution. To obtain the contraction
property, one usually has to choose a small time interval or small initial data $u_0$. Typical applications for this
method are
\begin{itemize}
  \item the graphical mean curvature flow or more general geometric equations,
  \item Stefan problems describing phase transitions with a free boundary,
  \item Cahn-Hilliard equations,
  \item variants of the Navier-Stokes equation.
\end{itemize}
For a survey on the idea of maximal regularity and on the above applications, we mention the monographs \cite{Amann95},
\cite{Pruess02}, and \cite{Pruess-Simonett16}.

The notion of maximal regularity depends on the function spaces in which the equation is considered. Typical function
spaces for partial  differential equations are H\"older spaces and $L^p$-Sobolev spaces. In the present survey, we
restrict ourselves to $L^p$-Sobolev spaces, i.e., we are considering maximal $L^p$-regularity. Here, the basic function
space for the right-hand side of \eqref{2-3a} will be $f  \in L^p((0,T); X)$, where  $X$ is some Banach space. In the
$L^p$-setting, one will typically choose $X=L^p(G)$ for some domain $G\subset\R^n$. The aim is to show that the
operator $A(t):= F(u(t))$ has, for every fixed $u$, maximal regularity in the sense specified below.

\subsection{Definition of maximal $L^p$-regularity}

We start with the notion of maximal $L^p$-regularity in the autonomous setting, i.e. for an operator $A$ independent of
$t$. Let $X$ be a Banach space, and let $A\colon X\supset D(A)\to X$ be a closed and densely defined linear operator.
Let $J=(0,T)$ with  $T\in (0,\infty]$. We consider the initial value problem
  \begin{align}
  \partial_t u(t) - A u(t) & = f(t)\quad (t\in J),\label{2-4}\\
  u(0) & = u_0.\label{2-5}
  \end{align}
Here, the right-hand side of \eqref{2-4} belongs to $\F:= L^p(J;X)$. For optimal regularity, we will expect $\partial_t
u\in L^p(J;X)$ and (consequently)  $Au\in L^p(J;X)$. An even stronger assumption would include $u\in L^p(J;X)$, too, so
that the  ``optimal'' space for the solution $u$ is given by
\begin{equation}\label{2-5a}
 \E := W_p^1(J;X)\cap L^p(J;D(A)).
\end{equation}
Here, for $k\in\N_0$ the vector-valued Sobolev space  $W_p^k(J;X)$ is defined as the space of all $X$-valued
distributions $u$ for which $\partial^\alpha u \in L^p(J;X)$ for all $|\alpha|\le k$, see Section~4 (cf. also
\cite{Hytonen-vanNeerven-Veraar-Weis16}, Section~2.5).

For the initial value $u_0$, we define the trace space:

\begin{definition}
  \label{2.2}
  a) The trace space $\gamma_t\E$ is defined by $\gamma_t\E:= \{ \gamma_t u: u\in\E\}$, where $\gamma_t u := u|_{t=0}$ stands
  for
  the time trace of the function $u$ at time $t=0$. We endow $\gamma_t \E$ with its canonical norm
  \[ \|x\|_{\gamma_t \E} := \inf\{ \|u\|_{\E}: u\in\E,\, \gamma_t u = x\}.\]

  b) We set $\leftidx{_0}{\E}:= \{u\in\E: \gamma_t u =0\}$ for the space of all functions in $\E$ with vanishing time
  trace at
  $t=0$.
\end{definition}

\begin{remark}
\label{2.3}
a) Note in the above definition that, by Sobolev's embedding theorem, one has the continuous embedding
\[ W_p^1((0,T);X) \subset C([0,T], X)\]
for every finite $T$, where the right-hand side stands for the space of continuous $X$-valued functions. Therefore, the
value  $\gamma_t u = u(0)$ is well defined as an element of $X$ for every $u\in\E$.

b) Let $T\in (0,\infty)$ again. By a), we obtain for $x\in \gamma_t\E$ and for every $u\in\E$ with $\gamma_t u=x$,
\[ \|x\|_X = \|\gamma_t u\|_X \le \max_{t\in [0,T]}\|u(t)\|_X \le C \|u\|_{W_p^1(J;X)} \le C \|u\|_{\E}.\]
Therefore, $\gamma_t\E\subset X$ with continuous embedding. On the other hand, if $x\in D(A)$, then the function $u(t)
:= e^{-t} x$ belongs  to $\E$ with $\|u\|_{\E}\le C\|x\|_X$ and satisfies $\gamma_t u = x$. Therefore, also the
continuous embedding $D(A)\subset \gamma_t\E$ holds.
\end{remark}

The following theorem is a deep result in the theory of interpolation of Banach spaces. Here, the real interpolation
functor  $(\cdot,\cdot)_{\theta,p}$ appears. We refer to \cite{Lunardi18} and \cite{Triebel95} for an introduction and
survey on interpolation spaces.

\begin{lemma}
  \label{2.4a}
  Let $A$ be a closed and densely defined operator, and let $\E$ be defined by \eqref{2-5a}.

  a) The  trace space $\gamma_t\E$ coincides with the real interpolation space with parameters  $1-\frac 1p$ and $p$, i.e.,
   we
  have
   \[ \gamma_t\E = (X, D(A))_{1-1/p,p}\]
   in the sense of equivalent norms.

   b) We have the continuous embedding  $\E\subset  C([0,T];\gamma_t\E)$. In particular, the time trace $\gamma_t\colon
    \E\to\gamma_t\E,\,  u\mapsto u(0)$ is well defined, and  $\gamma_t\E$ is independent of $T$.

  c) The norm of the continuous embedding  $\E\subset C([0,T];\gamma_t\E)$ depends, in general, on  $T$ and grows for
  decreasing $T$. On the  subspace  $\leftidx{_0}{\E}$, however, this norm can be chosen independently of  $T>0$, i.e.,
  there exists a constant $C_1$ independent of $T$ such that
  \[ \|u\|_{C([0,T];\gamma_t \E)} \le C_1 \|u\|_{\E} \quad ( u\in \leftidx{_0}{\E}{} ).\]
\end{lemma}

\begin{definition}
  \label{2.4}
  Let $T\in (0,\infty],\, J:=(0,T)$, and $p\in [1,\infty]$.

    a) We say that $A$ has maximal $L^p$-regularity ($A\in \mreg_p(J;X)$) if for each $f\in\F$ and $u_0\in\gamma_t\E$
    there exists a unique  solution $u\in\E$ of  \eqref{2-4}. Here, a function $u\in\E$ is called a solution of
    \eqref{2-4}--\eqref{2-5} if equality in \eqref{2-4} holds in the space $L^p(J;X)$ (i.e., for almost all $t\in (0,T)$),
    and equality \eqref{2-5} holds in $X$.

  b) We write $A\in \MRO(J;X)$ if for each $f\in\F$ and $u_0\in\gamma_t\E$ there exists a function  $u\colon [0,T]\to X$
   satisfying $\partial_t u\in L^p(J;X)$ and $Au \in L^p(J;X)$ such that  \eqref{2-4} holds for almost all $t\in (0,T)$ and
   \eqref{2-5} holds as equality in $X$, and if for all $f\in \F$ and $u_0\in\gamma_t\E$ the inequality
  \begin{equation}\label{2-7} \|\partial_t u\|_{L^p(J;X)} + \|Au\|_{L^p(J;X)}\le C\big( \|f\|_{L^p(J;X)} +
  \|u_0\|_{\gamma_t\E}\big)
  \end{equation}
  holds with a constant $C=C(J)$ independent of $f$ and $u_0$.

  c) We set $\mreg_p(X) := \mreg_p((0,\infty);X)$ and $\MRO(X) := \MRO((0,\infty);X)$.
\end{definition}

\begin{remark}
  \label{2.5}
  a) By the  definition of the spaces, the map
  \[  \binom {\partial_t - A}{\gamma_t}\colon \E\to \F\times \gamma_t\E, u\mapsto \binom{\partial_t u - Au}{\gamma_t u}\]
  is continuous. If $A\in \mreg_p(J;X)$, then, due to the definition of maximal regularity, this map is a bijection and
  therefore, by
  the open mapping theorem, an isomorphism. In particular, we obtain the a priori estimate
  \begin{equation}
   \label{2-6}
   \|u\|_{L^p(J;X)} + \|\partial_t u\|_{L^p(J;X)} + \| Au\|_{L^p(J;X)} \le C \big( \|f\|_{L^p(J;X)} +
   \|u_0\|_{\gamma_t\E}\big),
   \end{equation}
   which is stronger than \eqref{2-7}.

    b) If $A\in\mreg_p(J;X)$, then \eqref{2-4}--\eqref{2-5} with $u_0:=0$ is uniquely solvable for all $f\in\F$. On the
    other hand, for a given $u_0\in\gamma_t\E$, there exists an extension $u_1\in\E$ with $\gamma_t u_1 = u_0$
    by the definition of the trace
    space. Setting $u=u_1+u_2$, then we see that $u_2$ has to satisfy
    \begin{equation}
      \label{2-4a}
      \begin{aligned}
        \partial_t u_2(t) - A u_2(t) &= \tilde f(t)\quad (t>0),\\
        u_2(0) & = 0,
      \end{aligned}
    \end{equation}
    where $\tilde f := f - Au_1\in \F$. Therefore, the operator $A$ has maximal regularity if and only if the Cauchy
    problem \eqref{2-4a} is uniquely solvable for all $\tilde f \in \F$.

    c) Let the time interval  $J$ be finite, and assume $A\in \MRO (J;X)$. Then the Cauchy problem \eqref{2-4a} has a
     unique solution
    $u$ for all $\tilde f\in\F$ with $\partial_t u\in L^p(J;X)$. As $u(0)=0$, we can apply  Poincar\'e's inequality in the
    vector-valued Sobolev space $W_p^1((0,T);X)$ and obtain  $u\in L^p(J;X)$. This yields $u\in \E$, and by part b) of
    this remark, we see that $A\in \mreg_p(J;X)$. Therefore, $\MRO(J;X)=\mreg_p(J;X)$ for finite time intervals.
    Similarly, if $A\in \MRO((0,\infty);X)$ and if $A$ is invertible, we can estimate
    $\|u\|_{L^p(J;X)} \le C \|Au\|_{L^p(J;X)}$ and obtain
    $u\in\E$ again, which implies  $A\in \mreg_p((0,\infty);X)$.
\end{remark}

It turns out that the property of maximal $L^p$-regularity is independent of $p$. For a proof of the following result,
we refer to \cite{Dore93}, Theorem~4.2.

\begin{lemma}
If $A\in \mreg_p(X)$ holds for some $p\in (1,\infty)$, then $A\in \mreg_p(X)$ holds for every $p\in (1,\infty)$.
\end{lemma}

Based on this, we write $\mreg(X)$ instead of $\mreg_p(X)$. Note that the constant $C$ in \eqref{2-7} still depends on
$p$.

By Definition~\ref{2.4} and Remark~\ref{2.5} b), the operator $A$ has maximal $L^p$-regularity in $J=(0,\infty)$ if and
 only if the Cauchy problem
\begin{equation}
  \label{2-10}
  \begin{aligned}
  \partial_t u(t) - A u(t) & = f(t)\quad (t\in (0,\infty)),\\
  u(0) & = 0
  \end{aligned}
\end{equation}
has a unique solution $u\in W_p^1(J;X)$. We can extend $f$ and $u$ by zero to the whole line $t\in\R$ and obtain
functions $f\in L^p(\R;X)$ and $u\in W_p^1(\R;X)$ (for this, we need $u(0)=0$). After this, we apply the Fourier
transform in $t$, which is defined for smooth functions by
\[ (\mathscr F_t u)(\tau) := (2\pi)^{-1/2} \int_\R u(t) e^{-it\tau} dt.\]
For tempered distributions, we define $\mathscr F_t $ by duality. Note that $[\mathscr F_t (\partial_t u)](\tau) =
i\tau (\mathscr F_tu)(\tau)$. Therefore, \eqref{2-10} is equivalent to
\begin{equation}
  \label{2-11}
  (i\tau-A)(\mathscr F_tu)(\tau ) = (\mathscr F_t f)(\tau)\quad (\tau\in\R).
\end{equation}

\begin{theorem}
  \label{2.7}
  Let $J=(0,\infty)$ and $A$ be a closed densely defined operator. Then $A\in\MRO(J;X)$  if and only if the operator
  \[ \mathscr F_t^{-1} i\tau (i\tau-A)^{-1} \mathscr F_t \]
  defines a continuous operator in $L^p(\R;X)$.
\end{theorem}

\begin{proof}
By definition, $A\in \MRO(J;X)$ if and only if \eqref{2-10} has a unique solution $u$ with $\partial_t u\in L^p(\R;X)$
(again extending  the functions by zero to the whole line), and if we have an estimate of $\partial_t u$. This is
equivalent to unique solvability of the Fourier transformed problem \eqref{2-11}, i.e., the existence of
$(i\tau-A)^{-1}$ for almost all $\tau\in\R$ such that the solution $u$ satisfies
\[ \partial_t u = \mathscr F_t^{-1} i\tau (i\tau-A)^{-1} \mathscr F_t f \in L^p(\R;X),\]
and the estimate of $\partial_t u$ is equivalent to the condition $\mathscr F_t^{-1} i\tau (i\tau-A)^{-1} \mathscr
F_t\in L(L^p(\R;X))$.
\end{proof}

\subsection{Maximal regularity for non-autonomous problems}

With respect to the nonlinear equation \eqref{2-3} and its linearization \eqref{2-3a}, it makes sense to define maximal
regularity  also for non-autonomous problems. So we consider
\begin{align}
  \partial_t u(t) - A(t) u (t)& = f(t)\quad (t\in(0,T)),\label{2-12}\\
  u(0) & = u_0.\label{2-12a}
\end{align}
Here we assume that all operators $A(t)$ are closed and densely defined operators in some Banach space $X$ and have the
same  domain $D_A$. We also assume that we have a norm $\|\cdot\|_A$ on $D(A)$ which is, for every $t\in (0,T)$,
equivalent to the graph norm of $A(t)$, which is given by $\|\cdot \|_X + \|A(t)\,\cdot\,\|_X$. In this way, we can
identify the unbounded operator $A(t)\colon X\supset D_A\to X$ with the bounded operator $A(t)\in L(D_A,X)$. Moreover,
we assume that $A\in L^\infty((0,T); L(D_A,X))$.

Analogously to the autonomous case, we consider the basic space for the right-hand side $\F := L^p(J;X)$ with
$J:=(0,T)$ and  the solution space
\begin{equation}
  \label{2-13}
  \E := W_p^1(J;X) \cap L^p(J;D_A).
\end{equation}
We identify  $A\colon (0,T)\to L(D_A,X)$ with a function on  $\E$ by setting
\[ (Au)(t) := A(t)u(t)\quad (t\in (0,T),\; u\in \E).\]
The trace space $\gamma_t\E$ is defined as in Definition~\ref{2.2} a).

\begin{definition}
  \label{2.9}
  a) Let $f\in\F$ and $u_0\in\gamma_t\E$. Then a function  $u\colon (0,T)\to X $ is called a strong ($L^p$)-solution
   of \eqref{2-12}--\eqref{2-12a} if  $u\in\E$ and if \eqref{2-12} holds for almost all  $t\in (0,T)$ and \eqref{2-12a}
   holds in $X$.

  b) We say that  $A\in L^\infty((0,T);L(D_A,X))$ has maximal $L^p$-regularity on $(0,T)$ if for all $f\in\F$ and
  $u_0\in \gamma_t\E$   there exists a unique strong solution  $u\in\E$ of \eqref{2-12}--\eqref{2-12a}.
\end{definition}

\begin{remark}
  \label{2.10}
  Similarly to the autonomous case, the operator $A\in L^\infty((0,T);L(D_A,X))$ has maximal regularity if and only if
  \[  (\partial_t - A,\gamma_t)\colon\E\to \F\times\gamma_t\E\]
  is an isomorphism of Banach spaces. By trace results, this is equivalent to the condition that
  \eqref{2-12}--\eqref{2-12a} with $u_0=0$ has a unique solution $u\in\E$ for every $f\in \F$.
\end{remark}

The following result shows that maximal regularity for the non-autonomous operator family $(A(t))_{t\in (0,T)}$ can be
reduced to maximal  regularity for each $A(t)$ if the operator depends continuously on time.

\begin{theorem}
  \label{2.11}
  Let $T\in (0,\infty)$ and  $A\in C([0,T], L(D_A,X))$. Then $A$ has maximal $L^p$-regularity in the sense of
  Definition~\ref{2.9}  if and only if for every $t\in [0,T]$ we have $A(t)\in\mreg((0,T);X)$.
\end{theorem}

This is shown, using perturbation arguments, in  \cite{Amann04}, Theorem~7.1.

\section{The concept of $\mathcal R$-boundedness and the theorem of Mikhlin}

In Theorem~\ref{2.7}, we have seen that maximal regularity of the operator $A$ is equivalent to the boundedness of the
 operator
\[\mathscr F_t^{-1} m \mathscr F_t\colon L^p(\R;X) \to L^p(\R;X),\] where the operator-valued symbol $m\colon \R\to L(X)$
is given
 by $m(\tau) := i\tau (i\tau-A)^{-1}$. The classical theorem of Mikhlin gives sufficient conditions for a scalar-valued
 symbol to induce a  bounded operator in $L^p(\R^n)$. For the operator-valued analogue, the concept of
 $\mathcal R$-boundedness can be used. Therefore, we discuss in this section the notion of an $\RR$-bounded family and
 vector-valued variants of Mikhlin's theorem. As references for this section, we mention \cite{Denk-Hieber-Pruess03},
 Section~3, and \cite{Kunstmann-Weis04}, Section~2.

\subsection{$\RR$-bounded operator families}

Let $X$ and $Y$ be Banach spaces.

\begin{definition}\label{3.1}
A family  $\mathcal{T}\subset L(X,Y)$ is called  $\RR$-bounded if there exists a constant $C>0$ and some
$p\in[1,\infty)$ such that for all  $N\in\N$, $T_j\in\mathcal T$, $x_j\in X$ ($j=1,\dots,N$) and all sequences
$(\epsilon_j)_{j\in\N}$ of independent and identically distributed $\{-1,1\}$-valued and symmetric random variables on
a probability space $(\Omega,\mathscr A,\mathbb P)$ we have
\begin{equation}\label{3-1}
\Big\|\sum_{j=1}^N\varepsilon_jT_jx_j\Big\|_{L^p(\Omega,Y)}\leq
C\Big\|\sum_{j=1}^N\varepsilon_jx_j\Big\|_{L^p(\Omega,X)}.
\end{equation}
In this case,  $\RR_p(\mathcal{T}):=\inf\{C>0: \mbox{(\ref{3-1})} \textnormal{ holds}\}$ is called the $\RR$-bound of
$\mathcal{T}$.
\end{definition}

\begin{remark}
  \label{3.2}
a) For the sequence of random variables as above, we have $\mathbb P(\{\epsilon_j=1\}) = \mathbb P(\{\epsilon_j=-1\}) =
\frac12$. As the measure $\mathbb P\circ(\epsilon_1,\dots,\epsilon_N)^{-1}$ is discrete, the independence of the
sequence is equivalent to the condition
\[ \mathbb P(\{\epsilon_1=z_1,\dots,\epsilon_N=z_N\}) = 2^{-N} \quad
\Big((z_1,\dots,z_N)\in\{-1,1\}^N,\; N\in\N\Big).\] Therefore, $\RR$-boundedness is equivalent to the condition
\begin{equation}
\label{3-2}
\begin{aligned}
\exists\; & C>0\;\forall\; N\in\N\; \forall\; T_1,\dots,T_N\in\cal
T\;\forall\;x_1,\dots,x_N\in X\\
& \quad \Big(\sum_{z_1,\dots,z_N=\pm 1} \Big\|\sum_{j=1}^N z_j
T_jx_j\Big\|_Y^p\Big)^{1/p}
 \le C \Big(\sum_{z_1,\dots,z_N=\pm 1} \Big\|\sum_{j=1}^N z_j
 x_j\Big\|_X^p \Big)^{1/p}.
 \end{aligned}
\end{equation}
However, the stochastic description is advantageous, in particular, one can choose the probability space
$(\Omega,\mathscr A,\mathbb P) = ([0,1],\mathscr B([0,1]),\lambda)$, where $\mathscr B([0,1])$ stands for the Borel
$\sigma$-algebra, $\lambda$  for the Lebesgue measure, and the random variables  $\epsilon_j$ are given by the
Rademacher functions (see below). It seems to be unclear if the notation ``$\RR$'' stands for ``randomized'' or for
``Rademacher''.
\end{remark}

\begin{definition}
  \label{3.3}
  The Rademacher functions $r_n\colon[0,1]\to \{-1,1\}$ are defined by
  \[ r_n(t) := \sign \sin(2^n\pi t)\quad (t\in [0,1]).\]
\end{definition}

By definition, we have
\[ r_1 (t) = \begin{cases}
  1, & t\in (0,\frac12),\\ -1, & t\in (\frac 12, 1).
\end{cases}\]
The function $r_2$ has value $1$ on the intervals
$(0,\frac 14)$ and $(\frac 12, \frac 34)$. An immediate calculation yields
\[ \int_0^1 r_n(t) r_m(t) dt = \delta_{nm}\quad (n,m\in\N).\]
Moreover,
\[ \lambda(\{t\in [0,1]: r_{n_1}(t) = z_1,\dots, r_{n_M}(t)=z_M\}) =
\frac{1}{2^M} = \prod_{j=1}^M \lambda(\{t\in [0,1]: r_{n_j}(t) = z_j\}).\] Therefore, the sequence $(r_n)_{n\in\N}$ is
independent and identically distributed on the probability space $([0,1], \mathscr B([0,1]),\lambda)$ as in
Definition~\ref{3.1}. As all properties of  $(\epsilon_j)_j$ which are needed in this definition only depend on the
joint probability distribution, we can always choose $\epsilon_n = r_n$.

\begin{definition}
  \label{3.4} Let $X$ be a Banach space and  $1\le p<\infty$. Then $\rad_p(X)$
  is defined as the Banach space of all sequences $(x_n)_{n\in\N}\subset
  X$ for which the limit $\lim_{N\to\infty} \sum_{n=1}^N r_n(t)x_n =:f(t)$ exists for almost all $t\in [0,1]$ and
  defines a function $f\in L^p([0,1];X)$. For $(x_n)_{n\in\N}\in \rad_p(X)$, we define
  \[ \big\| (x_n)_{n\in\N} \big\|_{\rad_p(X)} := \Big\| \sum_{n=1}^\infty r_n x_n
  \Big\|_{L^p([0,1];X)}. \]
\end{definition}

\begin{remark}
  \label{3.5}
  a) It can be shown that for any sequence $(x_n)_{n\in\N}\subset X$,
  the sequence $\big( \big\|\sum_{n=1}^N r_n x_n\big\|_{L^p([0,1];X)}\big)_{N\in\N}$ is increasing, and therefore
 $\rad_p(X)$ is the space of all sequences $ (x_n)_{n\in\N}$ such that
 \[ \Big\|\sum_{n=1}^\infty r_n x_n
  \Big\|_{L^p([0,1];X)}<\infty.\]

  b) By definition, the map $J\colon\rad_p(X)\to L^p([0,1];X),\; (x_n)_n\mapsto
  \sum_{n=1}^\infty r_n x_n$ is well-defined. Assume that $J((x_n)_n)=0$, i.e., $\sum_n r_n x_n =0$ holds in
  $L^p([0,1];X)$. Then $\sum_n r_n f(x_n) =0$ for all
  $f\in X'$. Taking the inner product in $L^2$ with  $r_{n_0}$ for some fixed $n_0$, we get, using the orthogonality,
  $f(x_{n_0})=0$ for all  $f\in X'$ and therefore $x_{n_0}=0$. As $n_0$ was arbitrary, we obtain $x_n=0$ for all
  $n\in\N$, which shows that $J$ is injective. Therefore,  $\rad_p(X)$
 can be considered as a subspace of  $L^p([0,1];X)$, and the norm in $\rad_p(X)$ is the restriction of the norm in
 $L^p([0,1];X)$.
\end{remark}

\begin{theorem}[Kahane-Khintchine inequality]
  \label{3.6}
  The spaces  $\rad_p(X)$ are isomorphic for all  $1\le p<\infty$, i.e., there
  exist constants $C_p>0$ with
  \[ \frac1{C_p} \Big\| \sum_{n=1}^\infty r_n x_n
  \Big\|_{L^2([0,1];X)} \le \Big\| \sum_{n=1}^\infty r_n x_n
  \Big\|_{L^p([0,1];X)} \le C_p  \Big\| \sum_{n=1}^\infty r_n x_n
  \Big\|_{L^2([0,1];X)}.\]
\end{theorem}

In the scalar case $X=\C$, the proof of this inequality is elementary, for arbitrary Banach spaces, however, rather
complicated. In the  scalar case Theorem~{3.6} is known as Khintchine's inequality, in the Banach space valued case as
Kahane's inequality. We omit the proof which can be found, e.g., in  \cite{Hytonen-vanNeerven-Veraar-Weis16}, Theorem
3.2.23.

\begin{lemma}
  \label{3.8}
  a) If condition  \eqref{3-1} in Definition~\ref{3.1}
  holds for some $p\in[1,\infty)$, then it holds for all $p\in [1,\infty)$. For the corresponding $\RR$-bounds $\RR_p(\cal T)$ the
  inequality
  \[ \frac1{C_p^2} \RR_2(\cal T)\le \RR_p(\cal T)\le C_p^2 \RR_2(\cal
  T)\]
holds, where the constants  $C_p$ are from Theorem~\ref{3.6}.

b) A family $\cal T\subset L(X,Y)$ is  $\RR$-bounded with
 $\RR_2(\cal T)\le C$ if and only if for all $N\in\N$ and all
$T_1,\dots,T_N\in \cal T$, the map
\[ \mathbf T((x_n)_{n\in\N}) :=(y_n)_{n\in\N},\; y_n:=\begin{cases} T_n
x_n,&n\le N,\\ 0,&n>N\end{cases}\]
defines a bounded linear operator
$\mathbf T \in L(\rad_2(X))$ with norm $\|\mathbf T\|\le C$.
\end{lemma}

\begin{proof}
Part a) follows directly from Kahane's inequality, and part b) is a reformulation of the definition of
$\RR$-boundedness and an application  of the $p$-independence from a).
\end{proof}

\begin{remark}
  \label{3.9}
  a) If $\cal T\subset L(X,Y)$ is $\RR$-bounded, then $\cal
  T$ is uniformly bounded with  $\sup_{T\in\cal T}\|T\|\le \RR(\cal
  T)$. This follows immediately if we set $N=1$ in  the definition of  $\RR$-boundedness.

  b) If $X$ and $Y$ are Hilbert spaces, then
  $\RR$-boundedness is equivalent to uniform boundedness. In fact, in this situation also the spaces
 $L^2([0,1];X)$ and $L^2([0,1];Y)$
  are Hilbert spaces, and  $(r_n x_n)_{n\in\N}\subset L^2([0,1];X)$ and
  $(r_n T_n x_n)_{n\in\N}\subset L^2([0,1];Y)$ are orthogonal sequences. If $\|T\|\le C_{\cal T}$ for all $T\in\cal T\subset
  L(X,Y)$, then
  \begin{align*}
   & \Big\| \sum_{n=1}^N r_n  T_n x_n\Big\|_{L^2[0,1];Y)}^2 =
    \sum_{n=1}^N \| r_n T_n x_n\|^2_{L^2([0,1];Y)}= \sum_{n=1}^N \|T_n x_n\|_Y^2 \le C_{\cal
    T}^2\sum_{n=1}^N\|x_n\|_X^2\\
    & \qquad = C_{\cal T}^2\Big\|\sum_{n=1}^N r_n
    x_n\Big\|^2_{L^2([0,1];X)}.
  \end{align*}
\end{remark}

\begin{remark}
  \label{3.10}
  Let $X,Y,Z$ be Banach spaces, and  $\cal T,\cal S\subset L(X,Y)$ and
  $\cal U\subset L(Y,Z)$ be $\RR$-bounded. Then the families
  \[ \cal T + \cal S := \{ T+S: T\in\cal T, \, S\in\cal S\}\]
  and
  \[ \cal U\cal T := \{ UT: U\in\cal U,\, T\in\cal T\}\]
  are $\RR$-bounded, too, with
  \[ \RR( \cal T+\cal S) \le \RR(\cal T)+\RR(\cal S),\; \RR(\cal
  U\cal T)\le \RR(\cal U)\RR(\cal T).\]
  To see this, let $S_n\in\cal S$, $T_n\in\cal T$ and $U_n\in\cal U$ for
  $n=1,\dots,N$. Then the statement follows from
  \[ \Big\|\sum_{n=1}^N r_n(T_n+S_n)x_n\Big\|_{L^1([0,1];Y)}\le
   \Big\|\sum_{n=1}^N r_n T_nx_n\Big\|_{L^1([0,1];Y)} +
   \Big\|\sum_{n=1}^N r_nS_nx_n\Big\|_{L^1([0,1];Y)}\]
   and
   \[ \Big\|\sum_{n=1}^N r_nU_nT_nx_n\Big\|_{L^1([0,1];Z)}\le
   \RR(\cal U)\Big\|\sum_{n=1}^N r_n T_n x_n\Big\|_{L^1([0,1];Y)}.\]
\end{remark}

The following result turns out to be useful for showing $\RR$-boundedness.

\begin{lemma}[Kahane's contraction principle]\label{3.7}
Let $1\leq p<\infty$. Then for all  $ N\in\N$, for all $x_j\in X$ and all $ a_j, b_j\in\C$ with $|a_j|\leq|b_j|$,
$j=1,\dots,N$ we have
\begin{equation}\label{3-3}
\Big\|\sum_{j=1}^Na_jr_jx_j\Big\|_{L^p([0,1];X)}\leq
2\Big\|\sum_{j=1}^Nb_jr_jx_j\Big\|_{L^p([0,1];X)}.
\end{equation}
\end{lemma}

 \begin{proof}
Considering $\tilde x_j:=b_j x_j$, we may assume without loss of generality that  $b_j = 1$ and $|a_j|\leq 1$ for all
$j=1,\dots,N$. Treating  $\Re a_j$ and $ \Im a_j$ separately, we only have to show that for real $a_j $ with $|a_j|
\leq 1$ the inequality
\begin{equation}\label{3-4}
\Big\|\sum\limits_{j=1}^N a_jr_j x_j\Big\|_{L^p([0,1];X)} \leq
 \Big\|\sum\limits_{j=1}^N  r_j x_j\Big\|_{L^p([0,1];X)}
\end{equation}
holds. For this, let $\{e^{(k)}\}_{k=1,...,2^N}$ be a numbering of all vertices of the  cube  $[-1,1]^N$.  Because of
$a:= (a_1,\dots,a_N)^T\in [-1,1]^N$, the vector $a$ can be written as a convex combination of all $e^{(k)}$, i.e.,
there exist $\lambda_k \in [0,1]$ with
\[
\sum\limits_{k=1}^{2^N}\lambda_k = 1 \quad\mbox{and}\quad a =
\sum\limits_{k=1}^{2^N}\lambda_k e^{(k)}.
\]
Therefore, for  $e^{(k)} = (e^{(k)}_1,\dots,e^{(k)}_N)^T$ we see that
\begin{align*}
& \Big\|\sum\limits_{j=1}^N a_jr_j x_j\Big\|_{L^p([0,1];X)} \le
\sum\limits_{k=1}^{2^N} \lambda_k \Big\|\sum\limits_{j=1}^N r_j e^{(k)}_j x_j\Big\|_{L^p([0,1];X)} \\
&\qquad \le  \max\limits_{1\leq k\leq 2^N}\Big\|\sum\limits_{j=1}^N r_j e^{(k)}_j x_j\Big\|_{L^p([0,1];X)}
= \Big\|\sum\limits_{j=1}^N  r_j x_j\Big\|_{L^p([0,1];X)}.
\end{align*}
In the last equality we used the fact that $\{r_j: j=1,...,N\}$ and $\{r_j e^{(k)}_j : j=1,...,N\}$ have the same joint
probability distribution.
\end{proof}

\begin{theorem}
  \label{3.14} Let $\cal T\subset L(X,Y)$ be $\RR$-bounded. Then also the convex hull
  \[ \co\cal T := \Big\{ \sum_{k=1}^n \lambda_k T_k : n\in\N, \,
  T_k\in\cal T,\, \lambda_k\in [0,1],\, \sum_{k=1}^n \lambda_k
  =1\Big\}\]
  and the absolute convex hull
  \[ \aco\cal T := \Big\{ \sum_{k=1}^n\lambda_k T_k: n\in\N,\, T_k\in\cal
  T,\,\lambda_k\in\C,\,\sum_{k=1}^n |\lambda_k|=1\Big\}\]
  are $\RR$-bounded. The same holds for the closures $ \bar{\co \cal T}^s$ of $\co\cal T$ and $\bar{\aco\cal T}^s$ of
  $\aco \cal T$ with respect to the strong
  operator topology.  We have $\RR(\bar{\co \cal T}^s)
  \le \RR(\cal T)$ and $\RR(\bar{\aco\cal T}^s)\le 2\RR(\cal T)$.
\end{theorem}

\begin{proof}
  a) Let $T_1,\dots,T_N\in\co(\cal T)$. Then there exist
  $\lambda_{k,j}\in[0,1]$  and $T_{k,j}\in\cal T$
  with $\sum_{j=1}^{m_k} \lambda_{k,j}=1$ and
  $T_k=\sum_{j=1}^{m_k} \lambda_{k,j} T_{k,j}$.

  Define $\lambda_{k,j} := 0$ and $T_{k,j}:=0$ for $j\in\N$ with $j>m_k$ and $k=1,\dots,N$. For
  $\ell\in\N^N$ we define $\lambda_\ell := \prod_{k=1}^N \lambda_{k,\ell_k}$ and $T_{k,\ell}:= T_{k,\ell_k}$
  for $k=1,\dots,N$. Then $\lambda_\ell\in [0,1]$ as well as
  \[\sum_{\ell\in\N^n}\lambda_\ell = \sum_{\ell_1\in\N} \cdots\sum_{\ell_N\in\N} \lambda_{1,\ell_1}
  \cdot\ldots\cdot \lambda_{N,\ell_N} = 1.\]
  For all  $k=1,\dots,N$ we obtain
\begin{align*}
  &\sum_{\ell\in\N^N} \lambda_\ell T_{k,\ell} = \sum_{\ell\in\N^N} \lambda_\ell T_{k,\ell_k} = \Big( \sum_{\ell_k\in\N}
   \lambda_{k,\ell_k} T_{k,\ell_k}\Big) \prod_{j\not=\ell}\Big( \sum_{\ell_j\in\N} \lambda_{j,\ell_j}\Big) \\
  & \qquad= \sum_{\ell_k\in\N} \lambda_{k,\ell_k} T_{k,\ell_k} = T_k.
\end{align*}
 Note that these sums are finite. We get
  \begin{align*}
    &\Big\| \sum_{k=1}^N r_k T_k x_k\Big\|_{L^p([0,1];Y)}  = \Big\|\sum_{k=1}^N \sum_{\ell\in\N^N} r_k
    \lambda_\ell T_{k,\ell} x_k\Big\|_{L^p([0,1];Y)} \\
    & \qquad \le
    \sum_{\ell\in\N^N} \lambda_\ell \Big\| \sum_{k=1}^N r_k T_{k,\ell} x_k
    \Big\|_{L^p([0,1];Y)}  \le \RR(\cal T)\sum_{\ell \in\N^n} \lambda_\ell \Big\| \sum_{k=1}^N
    r_k x_k\Big\|_{L^p([0,1];X)} \\
    & \qquad = \RR(\cal T) \Big\| \sum_{k=1}^N r_k
    x_k\Big\|_{L^p([0,1];X)}.
  \end{align*}
  Consequently,  $\RR(\co \cal T) \le \RR(\cal T)$.

  b) By Kahane's contraction principle, $\RR(\cal T_0)\le 2\RR(\cal
  T)$, where we define
  \[ \cal T_0 := \{ \lambda T: T\in\cal T,\, \lambda\in\C,\,
  |\lambda|\le 1\}.\]
  Because of $\co\cal T_0=\aco \cal T$, we get $\RR(\aco\cal T) \le
  2\RR(\cal T)$ due to a).

  c) The closedness with respect to the strong operator topology follows directly from the definition of $\RR$-boundedness.
\end{proof}

The above results are useful to prove $\RR$-boundedness in general Banach spaces. In the special situation that $X$ is
some $L^q$-space, there is a helpful description of $\RR$-boundedness:

\begin{lemma}[Square function estimate]
  \label{3.11} Let $(G,\mathscr A, \mu)$ be a $\sigma$-finite measure space, $X=L^q(G)$,
  and let $1\le q<\infty$. Then $\cal T\subset L(X)$
  is $\RR$-bounded if and only if there exists an $M>0$ with
  \[ \Big\| \Big(\sum_{j=1}^N |T_n
  f_n|^2\Big)^{1/2}\Big\|_{L^q(G)} \le M \Big\| \Big(\sum_{j=1}^N
  |f_n|^2\Big)^{1/2}\Big\|_{L^q(G)}\]
  for all $N\in\N$, $T_n\in\cal T$ and $f_n\in L^q(G)$.
\end{lemma}

\begin{proof}
  We write $f\approx g$ if there are constants $C_1,C_2>0$ with
   $C_1|f|\le |g|\le C_2|f|$. To show  $\RR$-boundedness, by Kahane's inequality, we can consider the
    $\RR_q$-bound. For this, we can calculate
  \begin{align*}
    & \Big\|\sum_{n=1}^N r_n f_n  \Big\|_{L^q([0,1];L^q(G))}^q
    = \int_0^1 \Big\| \sum_{n=1}^N r_n(t)
    f_n(\cdot)\Big\|_{L^q(G)}^q dt = \int_0^1 \int_G \Big|\sum_{n=1}^N r_n(t)
    f_n(\omega)\Big|^q d\mu(\omega) \; dt\\
    & \qquad = \int_G \int_0^1 \Big| \sum_{n=1}^N
    r_n(t)f_n(\omega)\Big|^q dt d\mu(\omega)\approx \int_G \Big( \int_0^1 \Big| \sum_{n=1}^N r_n(t)
    f_n(\omega)\Big|^2 dt\Big)^{q/2} d\mu(\omega) \\
    & \qquad = \int_G \Big( \sum_{n=1}^N |f_n(\omega)|^2\Big)^{q/2}
    d\mu(\omega)= \Big\| \Big( \sum_{n=1}^N
    |f_n|^2\Big)^{1/2}\Big\|_{L^q(G)}^q.
  \end{align*}
  Here, Fubini's theorem and the inequality of  Khintchine were used.
  Now the statement follows by considering the above calculation for both sides of the definition of $\RR$-boundedness.
\end{proof}

\begin{example}
  \label{3.12}
  Using the square function estimate, it is easy to construct an example of a uniformly bounded operator family which
  is not $\RR$-bounded.   Let $p\in[1,\infty)\setminus\{2\}$. Then the family $\{T_n:n\in\N_0\}
  \subset L(L^p(\R))$, $T_nf(\cdot) :=
  f(\cdot-n)$ of translations is not $\RR$-bounded, as for
  $f_n=\chi_{[0,1]}$ we have
  \begin{align*}
  \Big\| \Big( \sum_{n=0}^{N-1}
  |T_nf_n|^2\Big)^{1/2}\Big\|_{L^p(\R)} & = \|\chi_{[0,N]}\|_{L^p(\R)}
  = N^{1/p},\\
  \Big\| \Big( \sum_{n=0}^{N-1}
  | f_n|^2\Big)^{1/2}\Big\|_{L^p(\R)} & = N^{1/2}\|\chi_{[0,1]}\|_{L^p(\R)}
  = N^{1/2}.
  \end{align*}
 For  $1\le p<2$, we use the fact that  $\frac{N^{1/p}}{N^{1/2}}\to\infty$ for
  $N\to\infty$. The proof for $p>2$ is similar.
\end{example}

\begin{lemma}
  \label{3.13}
a) Let $G\subset\R^n$ be open and  $1\le p<\infty$. For $\phi\in L^\infty(G)$, define $m_\phi\in L(L^p(G;X))$ by
$(m_\phi f)(x) := \phi(x)f(x)$. Then for $r>0$ one obtains
\[ \RR_p\Big(\{ m_\phi: \phi\in L^\infty(G),
\,\|\phi\|_\infty\le r\}\Big) \le 2r.\]

b) Let $1\le p<\infty$, $G\subset\R^n$ be open, and $\cal T\subset L(L^p(G;X),L^p(G;Y))$
be $\RR$-bounded. Then
\[ \RR_p\Big(\{ m_\phi Tm_\psi: T\in\cal T, \,\phi,\psi\in
L^\infty(G), \|\phi\|_\infty \le r ,\, \|\psi\|_\infty\le
s\}\Big) \le 4 rs \mathcal R_p(\mathcal T).\]
\end{lemma}

\begin{proof}
  a) By the theorem of Fubini and Kahane's contraction principle,
  \begin{align*}
    & \Big\| \sum_{k=1}^N r_k m_{\phi_k}
    f_k\Big\|_{L^p([0,1];L^p(G;X)} =
     \Big\| \sum_{k=1}^N r_k \phi_k
    f_k\Big\|_{L^p(G; L^p([0,1];X)} \\
    & \qquad\le 2r     \Big\| \sum_{k=1}^N r_k
    f_k\Big\|_{L^p(G; L^p([0,1];X)}  = 2r  \Big\| \sum_{k=1}^N r_k
    f_k\Big\|_{L^p([0,1];L^p(G;X)}.
  \end{align*}

  b) follows from a) and Remark~\ref{3.10}.
\end{proof}

In the following corollary, we consider strongly measurable function. Note that a function $N\colon G\to L(X,Y)$ is
called strongly measurable if there exists a $\mu$-zero set $A\in\mathscr A$ such that $N|_{G\setminus A}$ is
measurable and $N(G\setminus A)$ is separable.

\begin{corollary}
  \label{3.15} Let $(G,\mathscr A,\mu)$ be a $\sigma$-finite measure space and
  $\cal T\subset L(X,Y)$ be $\RR$-bounded. Let
  \[ \cal N := \{ N\colon G\to L(X,Y)\,|\, N\text{ strongly measurable with }N(G)\subset \cal T\}.\]
  For $h\in L^1(G,\mu)$ and $N\in\cal N$ define
  \[ T_{N,h} x := \int_{G} h(\omega)N(\omega)x
  d\mu(\omega)\quad (x\in X).\]
  Then
  \[ \RR\Big(\{ T_{N,h}: \|h\|_{L^1(G,\mu)}\le 1,\; N\in\cal N\} \Big)\le
  2\RR(\cal T).\]
\end{corollary}

\begin{proof}
  Let $\epsilon>0$. For $x_1,\dots,x_N\in X$, $h\in L^1(G,\mu)$ and $N\in\cal N$
  we consider the measurable map
  \[ M\colon G\to Y^N,\; M(\omega) := \big( N(\omega)x_j
  \big)_{j=1,\dots,N}.\]
  Then $M\in L^\infty(G;Y^N)$ is strongly measurable, and therefore there exist a measurable partition
  $G = \bigcup_{j=1}^\infty
  G_j$, $G_i\cap G_j=\emptyset$ for $i\not=j$, and $\omega_j\in G_j$ with
  \[ \| N(\omega) x_k - N(\omega_j) x_k\|_{Y} <\epsilon \quad\text{
  for almost all $\omega\in G_j$ and all $k=1,\dots,N$}.\]
  Define
  \[ S:= \sum_{j=1}^\infty \Big(\int_{G_j} h(\omega)
  d\mu(\omega)\Big) N(\omega_j).\]
  Then $\| T_{N,h}x_k - S
  x_k\|_Y<\epsilon $ for all $k=1,\dots,N$. Therefore,  $T_{N,h}$
  is a subset of the neighbourhood of $S$ given by  $x_1,\dots,x_N$ and $\epsilon$
  with respect to the strong operator topology. Because of  $S\in
  \bar{\aco\cal T}^s$, we obtain $T_{N,h}\in \bar{\aco\cal T}^s$. Now the statement follows from Theorem~\ref{3.14}.
\end{proof}

\begin{corollary}
  \label{3.16} Let $N\colon\Sigma_{\theta'} \to L(X,Y)$ be holomorphic and bounded, and let
   $N(\partial\Sigma_\theta\setminus\{0\})$ be $\RR$-bounded for some
   $\theta<\theta'$. Then $N(\Sigma_\theta)$ is
  $\RR$-bounded, and for every $\theta_1<\theta$ the family
   $\{\lambda\frac\partial{\partial\lambda}
  N(\lambda):\lambda\in\Sigma_{\theta_1}\}$ is $\RR$-bounded.
\end{corollary}

\begin{proof}
  Considering $M(\lambda) := N(\lambda^{2\theta/\pi})$,
  we may assume $\theta=\frac\pi 2$. Now we use Poisson's formula
  \[ N(\alpha+i\beta) = \frac 1\pi \int_{-\infty}^\infty
  \frac{\alpha}{\alpha^2+(s-\beta)^2} N(is) ds\quad (\alpha>0).\]
  Because of $\|\frac1\pi
  \frac{\alpha}{\alpha^2+(\cdot-\beta)^2}\|_{L^1(\R)} =1$, the first assertion follows from Corollary~\ref{3.15}.

  By Cauchy's integral formula, we have
  \[ \lambda\frac{\partial}{\partial\lambda} N(\lambda) =
  \int_{\partial\Sigma_\theta} h_\lambda(\mu) N(\mu)d\mu\quad
  (\lambda\in \Sigma_{\theta_1})\]
  for $h(\lambda) := \frac1{2\pi i}
  \frac{\lambda}{(\mu-\lambda)^2}$. Because of
  $\sup_{\lambda\in\Sigma_{\theta_1}}\|h_\lambda\|_{L^1(\partial\Sigma_\theta)}
  <\infty$, the second assertion follows from Corollary~\ref{3.15}, too.
\end{proof}

\begin{lemma}
  \label{3.17} Let $G\subset\C$ be open, $K\subset G$ be compact, and
  $H\colon G\to L(X,Y)$ be holomorphic. Then $H(K)$ is $\RR$-bounded.
\end{lemma}

\begin{proof}
  Let $z_0\in K$. Then there exists an $r>0$ with
  \[ H(z) = \sum_{k=0}^\infty H^{(k)}(z_0) \frac{(z-z_0)^k}{k!}\quad
  (|z-z_0|\le r).\]
  Here the series converges in $L(X,Y)$  and
  \[\rho_0 := \sum_{k=0}^\infty \| H^{(k)}(z_0)\|_{L(X,Y)}
  \frac{r^k}{k!}<\infty.\]
  As a set with one element,  $\{ H^{(k)}(z_0)\}$ is $\mathcal R$-bounded with $\mathcal R$-bound
  $\|H^{(k)}(z_0)\|_{L(X,Y)}$. By Kahane's contraction principle, the family  $\{ H^{(k)}(z_0)
  \frac{(z-z_0)^k}{k!}: z\in B(z_0,r)\}$ is $\mathcal R$-bounded, too, with $\RR$-bound not greater than  $2\frac{r^k}{k!}
  \|H^{(k)}(z_0)\|_{L(X,Y)}$. Therefore, we obtain for all finite partial sums the $\RR$-bound  $2\rho_0$.
  Taking the closure with respect to the strong operator topology, the same holds for the infinite sum.
  By a finite covering of $K$, we obtain the statement of the lemma.
\end{proof}

\begin{theorem}
  \label{3.18} Let $G\subset\R^n$ be open and $1<p<\infty$.
Let $\Lambda$ be a set and $\{k_\lambda:\lambda\in\Lambda\}$ be a family of measurable kernels
$k_\lambda\colon G\times G\to L(X,Y)$
with
\[ \RR_p\Big(\big\{ k_\lambda(z,z'):\lambda\in\Lambda\big\}\Big)\le
k_0(z,z')\quad(z,z'\in G).\] Assume that for the corresponding scalar integral operator
\[ (K_0f)(z) = \int_G k_0(z,z')f(z')dz'\quad (f\in L^p(G))\]
one has $K_0\in L(L^p(G))$. Define
\[ (K_\lambda f)(z) = \int_G k_\lambda(z,z')f(z')dz'\quad
 (f\in L^p(G;X)).\]
Then $K_\lambda\in L(L^p(G;X),L^p(G;Y))$ with
\[\RR_p \big(\{ K_\lambda:\lambda\in\Lambda\}\big) \le
\|K_0\|_{L(L^p(G))}.\]
\end{theorem}

\begin{proof}
  We use the definition of  $\RR$-boundedness and  get
  \begin{align*}
   & \Big\| \sum_{j=1}^Nr_j K_{\lambda_j}
    f_j\Big\|_{L^p([0,1];L^p(G;Y))}  = \Big( \int_0^1 \Big\| \sum_{j=1}^N r_j(t) \int_G
    k_{\lambda_j} (\cdot,z')f_j(z')dz'\Big\|_{L^p(G;Y)}^p
    dt\Big)^{1/p}\\
    & \qquad = \Big( \int_0^1 \Big\| \int_G \sum_{j=1}^N r_j(t)
    k_{\lambda_j}(\cdot,z')f_j(z')
    dz'\Big\|_{L^p(G;Y)}^pdt\Big)^{1/p}\\
    &\qquad= \Big( \int_0^1 \int_G \Big\| \int_G \sum_{j=1}^N r_j(t)
    k_{\lambda_j}(z,z')f_j(z')dz'\Big\|_Y^p dz\;dt\Big)^{1/p}\\
   & \qquad= \Big( \int_G \int_0^1 \Big\| \int_G \sum_{j=1}^N r_j(t)
    k_{\lambda_j}(z,z')f_j(z')dz'\Big\|_Y^p dt\;dz\Big)^{1/p}.
  \end{align*}
  Setting $\phi(t,z,z'):= \sum_{j=1}^N r_j(t)
    k_{\lambda_j}(z,z')f_j(z')$, the integral with respect to $t$ in the last term equals
     $\| \int_G
    \phi(\cdot,z,z')dz'\|_{L^p([0,1])}^p$. Now we apply the inequality
    \[ \Big\| \int_G \phi(\cdot,z,z') dz'\Big\|_{L^p([0,1])} \le
    \int_G \| \phi(\cdot,z,z')\|_{L^p([0,1])} dz'\]
   for Bochner integrals and obtain, using the assumption of $\RR$-boundedness,
  \begin{align*}
   & \Big\| \sum_{j=1}^N r_j K_{\lambda_j}
    f_j\Big\|_{L^p([0,1];L^p(G;Y))} \\
    & \qquad\le \Big( \int_G \Big[ \int_G \Big\|
    \sum_{j=1}^N r_j(\cdot)
    k_{\lambda_j}(z,z')f_j(z')\Big\|_{L^p([0,1];Y)}
    dz'\Big]^p dz \Big)^{1/p}\\
    &\qquad\le \Big( \int_G \Big[ \int_G k_0(z,z') \Big\|
    \sum_{j=1}^N r_j(\cdot)
   f_j(z')\Big\|_{L^p([0,1];X)}
    dz'\Big]^p dz \Big)^{1/p}\\
    & \qquad= \Big\| K_0\Big(  \Big\|
    \sum_{j=1}^N r_j
   f_j(\cdot)\Big\|_{L^p([0,1];X)}\Big) \Big\|_{L^p(G)}\\
   & \qquad\le \|K_0\|_{L(L^p(G))} \Big\|\Big( \Big\| \sum_{j=1}^N r_j
   f_j(\cdot)\Big\|_{L^p([0,1];X)}\Big)\Big\|_{L^p(G)}\\
   & \qquad= \|K_0\|_{L(L^p(G))} \Big\|  \sum_{j=1}^N r_j
   f_j\Big\|_{L^p([0,1];L^p(G;X))}.
 \end{align*}

\vspace*{-2em}
\end{proof}

\subsection{Fourier multipliers and Mikhlin's theorem}

We have already seen in Theorem~\ref{2.7} that maximal regularity is equivalent to the $L^p(\R;X)$-boundedness of the
operator  $\mathscr F_t^{-1} i\tau (i\tau-A)^{-1} \mathscr F_t$. This is a typical example of a (vector-valued) Fourier
multiplier. In the analysis of partial differential equations and boundary value problems in $L^p$-spaces, the question
of Fourier multipliers play a central role. The answer is given by the classical theorem of Mikhlin and by its Banach
space valued variants.

In the following, we use the standard notation $D := -i (\partial_{x_1},\dots,\partial_{x_n})$ as well as the standard
multi-index  notation $D^\alpha = (-i)^{|\alpha|} \partial_{x_1}^{\alpha_1} \ldots \partial_{x_n}^{\alpha_n}$. We start
with a simple example.

\begin{example}
  \label{3.19}
Consider the Laplacian $\Delta$ in  $L^p(\R^n)$ with maximal domain
$D(\Delta) := \{ u\in L^p(\R^n): \Delta u\in L^p(\R^n)\}$. Obviously we have
$D(\Delta)\supset W^2_p(\R^n)$. To show that we even have equality, we consider
 $u\in D(\Delta)$ and $f:=u-\Delta u\in L^p(\R^n)$. Let $|\alpha|\le 2$. Then
\[ D^\alpha u = \mathscr F^{-1} \xi^\alpha\mathscr Fu = -\mathscr
F^{-1} \frac{\xi^\alpha}{1+|\xi|^2} \mathscr F f \] holds as equality in  $\mathscr S'(\R^n)$, where $\mathscr F$
stands for the $n$-dimensional Fourier transform (see below).  To obtain $D^\alpha u\in L^p(\R^n)$, we have to show
$\mathscr F^{-1} m_\alpha \mathscr Ff\in L^p(\R^n)$, where  $m_\alpha (\xi) := \frac{\xi^\alpha}{1+|\xi|^2}$. So we
have to prove that
\[ f\mapsto \mathscr F^{-1} m_\alpha(\xi) \mathscr Ff\]
defines a bounded linear operator on $L^p(\R^n)$. This is in fact the case, as we will see from the   classical version
of Mikhlin's theorem,  Theorem~\ref{3.22} below.
\end{example}

In contrast to the above example, we will also need vector-valued versions of Mikhlin's theorem. For this, we need some
preparation,  starting with the vector-valued Fourier transform. Let $X$ be a Banach space. Then the Schwartz space
$\mathscr S(\R^n;X)$ is defined as the space of all infinitely smooth functions $\phi\colon \R^n\to X$ for which
\[ p_N (\phi) := \sup_{x\in\R^n} \max_{|\alpha|\le N} (1+|x|)^N \|\partial^\alpha \phi(x)\|_X < \infty\]
for all $N\in\N$. With the family of seminorms $\{p_N\colon N\in\N\}$, the Schwartz space becomes a Fr\'echet space.
The space of all $X$-valued tempered distributions is defined by
\[ \mathscr S'(\R^n;X) := L(\mathscr S(\R^n), X).\]
On $\mathscr S'(\R^n;X)$, we consider the family of seminorms
\[ \pi_\phi\colon \mathscr S'(\R^n;X)\to [0,\infty),\; u
\mapsto \|u(\phi)\|_X\quad (\phi\in\mathscr S(\R^n)).\] Then the family $\{\pi_\phi:\phi\in\mathscr S(\R^n)\}$ defines
a locally convex topology on $\mathscr S'(\R^n;X)$. Note that in the scalar case $X=\C$, this is the
weak-$\ast$-topology. One can see as in the scalar case that the Fourier transform, defined for $\phi\in\mathscr
S(\R^n;X)$ by
\[ (\mathscr F \phi)(\xi) := (2\pi)^{-n/2} \int _{\R^n} e^{-ix\cdot\xi} \phi(x)d x\quad
 (\xi\in\R^n,\; \phi\in\mathscr S(\R^n;X)),\]
can be extended by duality to an isomorphism $\mathscr F\colon \mathscr S'(\R^n;X) \to \mathscr S'(\R^n;X)$.

\begin{definition}\label{3.20}
Let $X,Y$ be Banach spaces, $1\le p<\infty$, and let $m\colon \R^n\to L(X,Y)$ be a bounded and strongly measurable
function. Because of $\mathscr F^{-1}\in L(L^1(\R^n;X), L^\infty(\R^n;Y))$, the function $m$ induces a map
$T_m\colon\mathscr S(\R^n;X)\to L^\infty(\R^n;Y)$ by
\[ T_m f := \mathscr F^{-1} m \mathscr F f\quad (f\in\mathscr
S(\R^n;X)).\]
The function  $m$ is called a Fourier multiplier (more precisely, an $L^p$-Fourier multiplier) if
\[ \| T_mf\|_{L^p(\R^n;Y)} \le C \| f\|_{L^p(\R^n;X)}\quad
(f\in\mathscr S(\R^n;X)).\] As $\mathscr S(\R^n;X)$ is dense in $L^p(\R^n;X)$ for $p\in [1,\infty)$, this implies that
 $T_m$ has a unique extension to a bounded linear operator $T_m\in
L(L^p(\R^n;X),L^p(\R^n;Y))$. In this case,  $m$ is called the symbol of the operator  $T_m$, and we write $\op [m]:=
 \mathscr F m\mathscr F^{-1}:= T_m$
  and $\symb[T_m] := m$.
\end{definition}

We start with the scalar case $X=Y=\C$.
\begin{remark}
  \label{3.21}
In the Hilbert space case  $p=2$, one can apply Plancherel's theorem.
Therefore, we have
  $\op[m]\in L(L^2(\R^n )$ if and only if the multiplication operator
  $g\mapsto mg$ is a bounded operator in $L^2(\R^n)$. This is equivalent to the condition
   $m\in L^\infty(\R^n)$.

   In fact, if $m\in L^\infty(\R^n)$, then  $\|mg\|_{L^2(\R^n )} \le \|m\|_{L^\infty(\R^n )}
  \|g\|_{L^2(\R^n )}$. On the other hand, if  $m\not\in L^\infty(\R^n )$,
  then there exists a sequence  $(A_k)_{k\in\N}$ of measurable subsets of $\R^n$ such that
  $0<\lambda(A_k)<\infty$ and $|m(x) | \ge k$ for $x\in A_k$. For the characteristic function $g_k
  :=\chi_{A_k}$ we obtain $g_k\in L^2(\R^n)$ and
  \[ \|mg_k\|_{L^2(\R^n)}^2 = \int
  |m(\xi)g_k(\xi)|^2d\xi \ge   k^2\lambda(A_k)=  k^2
  \|g_k\|_{L^2(\R^n)}^2.\]
  Therefore, $\op[m]$ cannot be a bounded operator in  $L^2(\R^n)$.
\end{remark}

The following classical theorem gives a sufficient condition for a function to be a (scalar) Fourier multiplier and has
many applications  in the theory of partial differential equations. In the following, $[\frac n2]$ denotes the largest
integer not greater than $\frac n2$. We state this result in two variants.

\begin{theorem}[Mikhlin's multiplier theorem]
  \label{3.22} Let  $1<p<\infty$ and $m\colon \R^n\setminus\{0\}\to\C$.
  If one of the two conditions
  \begin{enumerate}[\upshape (i)]
  \item $m\in C^{[\frac n2]+1}(\R^n\setminus\{0\})$ and
  \[ |\xi|^{|\beta|} |\partial^\beta m(\xi)| \le C_M \quad
  (\xi\in\R^n\setminus\{0\},\, |\beta|\le \textstyle{[\frac n2]
  +1}),\]
  \item $m\in C^n(\R^n\setminus\{0\})$ and
  \[ \big| \xi^\beta  \partial^\beta m(\xi)\big| \le C_M \quad
  (\xi\in\R^n\setminus\{0\},\, \beta\in \{0,1\}^n)\]
  \end{enumerate}
  holds with a constant $C_M>0$, then $m$ is an
  $L^p$-Fourier multiplier with
   \[\|\op[m]\|_{L(L^p(\R^n))}\le c(n,p)
  C_M,\] with a constant $c(n,p)$  depending only on $n$ and $p$.
\end{theorem}

A proof of this theorem (which is also called Mikhlin-H\"ormander theorem) can be found, e.g., in \cite{Grafakos14},
Section~6.2.3. Condition (i) is sometimes called the Mikhlin condition, whereas condition (ii) is called the Lizorkin
condition. For the $L^p$-continuity of singular integral operators, we also refer to  \cite{Stein93}, Section~6.5.

For the following result, note that a function $m\colon \R^n\setminus\{0\}\to\C$ is called (positively) homogeneous
with respect to  $\xi$ of degree $d\in\R$ if
\[ m(\rho\xi) = \rho^d m(\xi)\quad (\xi\in\R^n\setminus\{0\},\, \rho>0).\]

\begin{lemma}
  \label{3.24}
  Let $m\in C^{[\frac n2]+1}(\R^n\setminus\{0\})$ be homogeneous of degree $0$. Then $m$
satisfies the Mikhlin condition.
\end{lemma}

\begin{proof}
If a function  $m\in C^k(\R^n\setminus\{0\})$ is homogeneous of degree $d$, then its derivative $\partial^\beta m(\xi)$
is homogeneous  of degree $d-|\beta|$ for all $|\beta|\le k$. This follows from the identities $\partial^\beta[
m(\rho\xi)] = \rho^{|\beta|} (\partial^\beta m)(\rho\xi)$ and $\partial^\beta [\rho^d m(\xi)] = \rho^d (\partial^\beta
m)(\xi)$.

Now let $m\in C^{[\frac n2]+1}(\R^n\setminus\{0\})$ be homogeneous of degree 0, and let
 $|\beta|\le [\frac n2]+1$. Then   $m_\beta(\xi) :=
|\xi|^{|\beta|} \partial_\beta m(\xi)$ is homogeneous of degree 0 and continuous. Therefore,
\[ |m_\beta (\xi)| = \Big|m_\beta\Big(\frac \xi{|\xi|}\Big)\Big| \le
\max_{|\eta|=1} |m_\beta(\eta)| <\infty\quad (\xi\in\R^n\setminus\{0\}).\] \vspace*{-2em}

\end{proof}

As a first application of Mikhlin's theorem, we can now answer the question from Example~\ref{3.19}.

\begin{corollary}
  \label{3.25} Let $1<p<\infty$. Then  $\{u\in L^p(\R^n): \Delta
  u\in L^p(\R^n)\} = W^2_p(\R^n)$.
\end{corollary}

\begin{proof}
  As we have seen in Example~\ref{3.19}, we have to show that the function  $m_\alpha(\xi )
  := \frac{\xi^\alpha}{1+|\xi|^2}$ satisfies the Mikhlin condition for all $|\alpha|\le 2$. For this,
  we write $m_\alpha(\xi) = \tilde m_\alpha(\xi,1)$  where the function $\tilde m_\alpha\colon \R^{n+1}
  \setminus\{0\}\to\C$ is defined by
  \[\tilde m_\alpha(\xi,\mu) := \frac{\xi^\alpha\mu^{2-|\alpha|}}{\mu^2+|\xi|^2}.\]
  As the function $\tilde m_\alpha$ is smooth and homogeneous of degree 0, it satisfies the Mikhlin condition by
  Lemma~\ref{3.24}.
   Setting $\mu=1$, we see that also $m_\alpha$ satisfies the Mikhlin condition.
\end{proof}

As mentioned above, we also need vector-valued variants of Mikhlin's theorem. The following results assume some
geometric conditions  on the Banach space $X$. For a detailed discussion of these properties, see, e.g.,
\cite{Hytonen-vanNeerven-Veraar-Weis16}, Chapter~4.

\begin{definition}
  \label{3.26}
  a) A Banach space $X$ is called a UMD space or a space of class $\HT$ if the symbol $m(\xi):= -i\sgn(\xi)\id_X$ yields
  a bounded  operator $\op[m] \in L(L^p(\R;X))$. The operator $\op[m]$ is called the Hilbert transform.

  b) A Banach space $X$ is said to have property $(\alpha)$ if there exists a constant $C>0$ such that for all $N\in\N$,
   all i.i.d. symmetric $\{-1,1\}$-valued random variables $\epsilon_1,\dots,\epsilon_N$ on $\Omega$ and
   $ \epsilon_1',\dots, \epsilon_N'$ on  $\Omega'$, all $\alpha_{ij}\in\C$ with $|\alpha_{ij}|\le 1$, and all $x_{ij}\in X$ we
   have
  \[ \Big\| \sum_{i,j=1}^N \alpha_{ij} \epsilon_i\epsilon_j'x_{ij}\Big\|_{L^2(\Omega\times\Omega';X)}\le
  C \Big\| \sum_{i,j=1}^N  \epsilon_i\epsilon_j'x_{ij}\Big\|_{L^2(\Omega\times\Omega';X)}.\]
\end{definition}

\begin{remark}
  \label{3.27}
  a) Every UMD space is reflexive. In particular, $L^1(G)$ and $L^\infty(G)$ are no UMD spaces. However, $L^1(G)$ has
  property $(\alpha)$.

  b) Every Hilbert space is a UMD space with property $(\alpha)$. If $E$ is a UMD space with property $(\alpha)$
  and if $(S,\sigma,\mu)$ is a $\sigma$-finite measure space, then also $L^p(S;E)$ is a UMD space with property
  $(\alpha)$ for all $p\in (1,\infty)$.

  c) More generally, if $G\subset \R^n$ is a domain,  $E$ is a UMD space with property $(\alpha)$ and $p,q\in(1,\infty)$,
  then the vector-valued Besov space $B_{pq}^s(G;E)$ and the vector-valued Triebel-Lizorkin space $F_{pq}^s(G;E)$
  are again UMD spaces with property $(\alpha)$. In particular, this holds in the scalar case $E=\C$.
\end{remark}

The following result is the vector-valued analog of Mikhlin's theorem and was central in the development of the theory
and application of maximal $L^p$-regularity.

\begin{theorem}
  \label{3.29} Let $X$ and $Y$ be UMD Banach spaces, and let
   $1<p<\infty$. Assume  $m\in C^n(\R^n\setminus\{0\}; L(X,Y))$ with
  \[ \RR\Big(\big\{ |\xi|^{|\alpha|} \partial^\alpha m(\xi):
  \xi\in\R^n\setminus\{0\}, \; \alpha\in\{0,1\}^n\big\}\Big)
  =:\kappa <\infty.\]
  Then $m$ is a vector-valued Fourier multiplier with
  \[ \| \op[m]\|_{L(L^p(\R^n;X),L^p(\R^n;Y))} \le C \kappa,\]
  where the constant  $C$ depends only on  $n,p,X$, and $Y$.
\end{theorem}

The proof of Theorem~\ref{3.29} uses Paley-Littlewood decompositions, see \cite{Kunstmann-Weis04}, Theorem~4.6, or
\cite{Hytonen-vanNeerven-Veraar-Weis16}, Theorem 5.3.18.

In the last result, we had one symbol $m$ and the related operator $\op[m]$. The following theorem shows that for a
family of symbols  satisfying uniform Mikhlin type estimates, also the related operator family is $\mathcal R$-bounded.

\begin{theorem}
  \label{3.30} Let  $X$ and $Y$ be UMD Banach spaces with property  $(\alpha)$. Let $\cal T\subset L(X,Y)$
  be $\RR$-bounded. Consider the set
  \[ M := \Big\{ m\in C^n(\R^n\setminus\{0\}; L(X,Y)): \xi^\alpha
  D^\alpha m(\xi)\in\cal T \quad (\xi\in\R^n\setminus\{0\},\,
  \alpha\in\{0,1\}^n)\Big\}.\]
  Then $\{\op[m]:m\in M\}\subset L(L^p(\R^n;X),L^p(\R^n;Y))$
  is $\RR$-bounded with $\RR_p(\{\op[m]:m\in M\})\le C \RR_p(\cal T)$,
  where the constant $C$ depends only on  $p,m,X$, and $Y$.
\end{theorem}

For a proof of this result, we refer to \cite{Girardi-Weis03}, Theorem~3.2. Theorem~\ref{3.30} is also the basis for an
iteration process:  $\RR$-bounded symbol families yield $\RR$-bounded operator families. For an application to
pseudodifferential operators with $\RR$-bounded symbols, we also refer to \cite{Denk-Krainer07}.

Note that Theorem~\ref{3.30} also gives a strong result in the scalar case $X=\C$. As $\C$ is a Hilbert space,
boundedness in $\C$ equals $\mathcal R$-boundedness. Therefore, boundedness of a family of scalar symbols implies
$\RR$-boundedness of the corresponding operator family. The same holds if $X$ is a general Hilbert space. We give a
simple but useful example.

\begin{corollary}
  \label{3.31} Let $\{m_\lambda:\lambda\in\Lambda\}$ be a family of matrix valued functions
   $m_\lambda\in C^n(\R^n\setminus\{0\};\C^{N\times
  N})$ with
  \[ |\xi^\alpha D^\alpha m_\lambda(\xi)|_{\C^{N\times N}} \le
  C_0\quad (\xi\in\R^n\setminus\{0\},\, \alpha\in\{0,1\}^n,\,
  \lambda\in\Lambda).\]
  Then  $\{ \op[m_\lambda]:\lambda\in\Lambda\}\subset
  L(L^p(\R^n;\C^N))$ is $\RR$-bounded with $\RR$-bound  $C\cdot C_0$,
  where  $C$ only depends on  $p$ and $N$.
\end{corollary}

\begin{proof}
  As a Hilbert space, $X=\C^N$ is a UMD space with property  $(\alpha)$. By assumption, we know that
  \[ \big\{ \xi^\alpha D_\xi^\alpha m_\lambda(\xi):
  \xi\in\R^n\setminus\{0\},\,
  \alpha\in\{0,1\}^n,\,\lambda\in\Lambda\big\}\subset L(X)\]
  is norm bounded and consequently, as  $X$ is a Hilbert space, also
  $\RR$-bounded. Choosing $\cal T := \{
  A\in\C^{N\times N}: |A|\le C_0\}$ in Theorem~\ref{3.30}, we obtain the $\RR$-boundedness of  $\{\op[m_\lambda]:
  \lambda\in\Lambda\}\subset L(L^p(\R^n;\C^N))$.
\end{proof}

\subsection{$\RR$-sectorial operators}

Now we come back to the question of maximal $L^p$-regularity. As we have seen in Theorem~\ref{2.11}, maximal regularity
holds if and  only if the operator-valued symbol  $m(\lambda) := \lambda(\lambda-A)^{-1}$ for $\lambda\in i\R$ is a
bounded operator in $L^p(\R;X)$. So we can apply the one-dimensional case of Theorem~\ref{3.29}. We start with a notion
from operator theory.

In the following, let
    \[\Sigma_{\phi}   :=   \Bigl\{ z \in \C \setminus \{0\}:  |\arg (z)| < \phi
    \Bigr\}
\]
for $\phi\in (0,\pi]$. We denote the spectrum and the resolvent set of an operator $A$ by  $\sigma(A)$ and $\rho(A)$,
respectively.

\begin{definition}
  \label{3.32}
  Let $A\colon D(A)\to X$ be a linear and densely defined operator. Then  $A$ is called sectorial if there exists
  an angle $\phi>0$ such
  that
  $\rho(A)\supset \Sigma_\phi$ and
  \[ \sup_{\lambda\in\Sigma_\phi} \|\lambda
  (\lambda-A)^{-1}\|_{L(X)} <\infty.\]
  If this is the case, we call
  \[\phi_A := \sup\{ \phi: \; \rho(A)\supset \Sigma_\phi,\;
  \sup_{\lambda\in\Sigma_\phi} \|\lambda
  (\lambda-A)^{-1}\|_{L(X)} <\infty\}\]
  the spectral angle of  $A$.
\end{definition}

The following theorem is an important result from the theory of semigroups of operators (see, e.g.,
\cite{Engel-Nagel00} , Theorem~II.4.6).

\begin{theorem} \label{3.33}
Let  $A\colon  D(A) \mapsto X$ be linear and densely defined. Then the following statements are equivalent:
\begin{itemize}
\item [\rm(i)] $A$ generates a bounded holomorphic
$C_{0}$-semigroup on $X$ with angle  $\vartheta \in (0,
\frac\pi 2]$.
\item [\rm(ii)]  $A$ is sectorial with spectral angle $\phi_A\ge
\vartheta+\frac \pi 2$.
\end{itemize}
\end{theorem}

It turns out that a similar condition characterizes operators with maximal $L^p$-regularity. For the following result,
cf. \cite{Denk-Hieber-Pruess03}, Theorem 4.4, \cite{Weis01}, Theorem~4.2,  and  \cite{Kunstmann-Weis04}, Theorem~1.11.

\begin{theorem}[Theorem of Weis]\label{3.34}
Let  $X$ be a UMD Banach space,
$1<p<\infty$, and  $A$ be a sectorial operator with spectral angle
$\phi_A>\frac\pi 2$. Then $A\in \mreg((0,\infty);X)$ if the family
\[ \big\{\lambda
  (\lambda-A)^{-1}: \lambda\in\Sigma_\phi \}\subset L(X) \]
is $\RR$-bounded for some  $\phi>\frac\pi 2$.
\end{theorem}

With respect to the last theorem, one defines $\RR$-sectorial operators:

\begin{definition}
  \label{3.35}
 Let $A\colon D(A)\to X$ be a linear and densely defined operator.
  Then  $A$ is called $\RR$-sectorial if there exists an angle   $\phi>0$ with
  $\rho(A)\supset \Sigma_\phi$ and
  \[ \RR\big\{\lambda
  (\lambda-A)^{-1}: \lambda\in\Sigma_\phi \} <\infty.\]
  The $\RR$-angle of $A$ is defined as the supremum of all angles for which the above $\RR$-bound is finite.
\end{definition}

By Theorem~\ref{3.34}, a sectorial operator has maximal regularity if  it is $\RR$-sectorial with $\RR$-angle larger
than  $\frac\pi 2$. In fact, one has the following equivalences.

\begin{theorem}
  \label{3.36} Let $A$ be the generator of a bounded holomorphic $C_0$-semigroup. Then the following statements are equivalent:
\begin{enumerate}[\upshape (i)]
  \item There exists a $\delta>0$ such that $A$ is $\RR$-sectorial with $\RR$-angle $\phi_{\RR}= \frac\pi2+\delta$.
  \item There exists an $n\in\N$  such that $\{t^n (it-A)^{-n}:
  t\in\R\setminus\{0\}\}$ is $\RR$-bounded.
  \item There exists a  $\delta>0$ such that the family  $\{T_z:z\in\Sigma_\delta\}$
  is $\RR$-bounded.
  \item The family  $\{T_t, tAT_t: t>0\}$ is $\RR$-bounded.
\end{enumerate}
\end{theorem}

\begin{proof}
We only give a sketch of proof, for the full version see \cite{Kunstmann-Weis04}, Theorem~1.11.

(i)$\Longrightarrow$(ii) is trivial.

(ii)$\Longrightarrow$(i). We write
\[ (it-A)^{-n+1} = (n-1)i\int_t^\infty (is-A)^{-n} ds \]
and obtain
\[ (it)^{n-1} (it-A)^{-n+1} = \int_0^\infty h_t(s) \big[ (is)^n
(is-A)^{-n}\big]ds\] for the function $h_t(s) := (n-1) t^{n-1}
s^{-n} \chi_{[t,\infty)}$. We have  $\int_0^\infty h_t(s)ds=1$, and
 Corollary~\ref{3.15} yields (ii) for $n-1$ instead of
 $n$. Iteratively, we see that  (ii) holds for  $n=1$. Now we use
Corollary~\ref{3.16} to show the  $\RR$-boundedness of
$\{\lambda(\lambda-A)^{-1}:\lambda\in\Sigma_{\pi/2}\}$. By considering power series expansion, one can show that
 $\lambda(\lambda-A)^{-1}$ is in fact $\RR$-bounded on  some larger sector.

(iii)$\Longrightarrow$(i). This follows from
Corollary~\ref{3.15}, too, with help of the representation
\[ (\lambda-A)^{-1} = \int_0^\infty e^{-\lambda t} T_t dt.\]

(i)$\Longrightarrow$(iii) follows similarly by
\[ T_z = \frac1{2\pi i}\int_{\Gamma_t} e^{\lambda z}
(\lambda-A)^{-1} d\lambda.\]

(iii)$\Longleftrightarrow$(iv) can be shown using Corollary~\ref{3.16}.
\end{proof}

\section{$L^p$-Sobolev spaces}

In the definition of maximal regularity, the vector-valued Sobolev space $W_p^1(J;X)$ appears. In many cases,
$X=L^p(G)$ for some  domain $G\subset \R^n$, and it would be desirable to obtain a more explicit description of the
space $\gamma_t\E$ of time traces in this situation. Note that Lemma~\ref{2.4a} tells us that this is connected with
real interpolation.  A similar question arises if the operator $A$ is a differential operator in some domain
$G\subset\R^n$. In this case, the domain $D(A)$ is described by boundary operators, and the spaces for the boundary
traces will be non-integer Sobolev spaces. For $p\not=2$, there
 are  different scales of non-integer Sobolev spaces: Besov spaces, Triebel-Lizorkin spaces, and Bessel potential spaces.
A modern definition of these scales is based on dyadic
 decomposition and on the Fourier transform. A classical reference for this is the book by Triebel  (\cite{Triebel95},
 Section~2.3),
 where the scalar case is discussed. For a modern presentation, including the vector-valued situation, we mention the
monograph by Amann (\cite{Amann19}, Chapter~VII). Note that in the vector-valued situation, the related integrals are
Bochner integrals, and we
 refer to \cite{Hytonen-vanNeerven-Veraar-Weis16}, Section~1, and \cite{Arendt-Batty-Hieber-Neubrander11}, Section 1.1,
for an introduction to vector-valued integration.

\begin{definition}
  \label{4.1} A sequence $(\phi_k)_{k\in\N_0}$ of $C^\infty$-functions $(\phi_k)_{k\in\N_0}$ is called a dyadic
  decomposition if
  \begin{enumerate}
    [(i)]
    \item  $\phi_k\ge 0$, $\supp\phi_0\subset B(0,2)$ and $\supp\phi_k\subset\{ \xi\in\R^n: 2^{k-1} <|\xi|<2^{k+1}\}$
        for all $k\in\N$,
    \item $\sum_{k\in\N_0} \phi_k(\xi) =1 $ for all $\xi\in\R^n$,
    \item for all $\alpha\in\N_0^n$ there exists a $c_\alpha>0$ with
    \[ |\xi|^{|\alpha|} \big| \partial^\alpha \phi_k(\xi)\big| \le c_\alpha\quad (\xi\in\R^n,\, k\in\N_0).\]
  \end{enumerate}
\end{definition}

It is easy to define a dyadic decomposition by scaling a fixed function $\phi_1$ (see \cite{Triebel95}, Section~2.3.1).
By the theorem of Paley-Wiener, for every $u\in\mathscr S'(\R^n)$ the distribution $\op[\phi_k]u$ is a regular
distribution and even a smooth function. Therefore, $(\op[\phi_k]u)(x)$ is well-defined.
In the following, let $X$ be a
Banach space.

\begin{definition}
  \label{4.2}
  a) For  $s\in\R$, $p,q\in[1,\infty)$, the Besov space $B_{pq}^s(\R^n;X)$ is defined by $B_{pq}^s(\R^n;X):=
  \{u\in\mathscr S'(\R^n;X): \|u\|_{B_{pq}^s(\R^n;X)}<\infty\}$, where
  \[     \|u\|_{B_{pq}^s(\R^n;X)}  = \Big[ \sum_{k\in\N_0} 2^{skq} \Big( \int_{\R^n} \|(\op[\phi_k]u)(x)\|_X^p dx
  \Big)^{q/p}\Big]^{1/q}.\]

   b) For $s\in\R$ and $p,q\in [1,\infty)$ the Triebel-Lizorkin space $F_{pq}^s(\R^n;X)$ is defined by
  $F_{pq}^s(\R^n;X):=\{ u\in\mathscr S'(\R^n;X): \|u\|_{F_{pq}^s(\R^n;X)}<\infty\}$, where
 \[ \|u\|_{F_{pq}^s(\R^n;X)}  = \Big[ \int_{\R^n} \Big( \sum_{k\in\N_0} 2^{skq} \|(\op[\phi_k]u)(x)\|_X^q\Big)^{p/q}
 dx\Big]^{1/p}.\]

 c) If $p=\infty$ or $q=\infty$, the above definitions hold with the standard modification.
\end{definition}

By an application of Fubini's theorem, we immediately see that for $p=q$ the definitions of Besov spaces and
Triebel-Lizorkin spaces coincide, but in general we have two different scales of Sobolev space type. For the third
scale, the Bessel potential spaces, we consider the function $\langle \,\cdot\,\rangle\colon \R^n\to\R,\, \xi\mapsto
\langle\xi\rangle := (1+|\xi|^2)^{1/2}$. For the following definition, we refer to
\cite{Hytonen-vanNeerven-Veraar-Weis16}, Definition 5.6.2.

\begin{definition}
  Let $s\in\R$ and $p\in [1,\infty]$. Then the Bessel potential space $H_p^s(\R^n;X)$ is defined as the space of all
  $u\in \mathscr S'(\R^n;X)$ for which $\op[\langle\,\cdot\,\rangle^s]u\in L^p(\R^n;X)$. The corresponding norm is defined
  as
  \[ \|u\|_{H^s_p(\R^n;X)} := \|\op[\langle\,\cdot\,\rangle^s]u\|_{L^p(\R^n;X)}.\]
\end{definition}

\begin{remark}\label{4.3}
  a) Many classical Sobolev spaces can be found as special cases of the above definition.
  \begin{itemize}
    \item Let $X$ be a UMD space, $k\in\N$, and $p\in (1,\infty)$, and let  $W_p^k(\R^n;X)$ denote the classical
        Sobolev space,
\[ W_p^k(\R^n;X) := \big\{ u\in L^p(\R^n;X): \,\forall\, |\alpha|\le k:\, \partial^\alpha u\in L^p(\R^n;X)\big\}.\]
         Then $W_p^k(\R^n;X) = H_p^k(\R^n;X)$ with equivalent norms (\cite{Hytonen-vanNeerven-Veraar-Weis16},
        Theorem~5.6.11).
    \item Let $p\in (1,\infty)$ and $s\in\R$. Then the equality  $H_p^s(\R^n;X) = F_{p2}^s(\R^n;X)$ holds if and only
        if $X$ is isomorphic to a Hilbert space (\cite{Han-Meyer96}, Theorem~1.2).
    \item Let $p\in [1,\infty)$ and $s\in (0,\infty)\setminus \N$. Then the Sobolev-Slobodeckii space $W_p^s(\R^n;X)$
        is given as $W_p^s(\R^n;X) = B_{pp}^s(\R^n;X)$ (\cite{Amann19}, Remark~3.6.4).
     \item Let $s\in (0,\infty)\setminus\N$. Then the classical H\"older space is given as $C^s(\R^n;X) =
         B_{\infty,\infty}^s(\R^n;X)$ (\cite{Amann19}, Remark~3.6.4).
  \end{itemize}

  b) Let $G\subset\R^n$ be a domain. Then the space $B_{pq}^s(G;X)$ is defined by restriction, i.e.
  \[B_{pq}^s(G;X) := \{ u\in \mathscr D'(G;X): \,\exists\, \tilde u\in B_{pq}^s(\R^n;X): u = \tilde u|_G\}\]
  with canonical norm
  \[ \|u\|_{B_{pq}^s(G;X)} := \inf\big\{ \|\tilde u\|_{B_{pq}^s(\R^n;X)}: u = \tilde u|_G \big\}.\]
  Note here that the restriction of a distribution is defined as $\tilde u|_G := \tilde u|_{\mathscr D(G)}$. In the
  same way, the other scales are defined on domains.
\end{remark}

The following result can be shown with the theory of interpolation spaces and is the basis for the description of the
 trace spaces. We refer to \cite{Amann19}, Theorem~2.7.4, for a proof (with $G=\R^n$, the case of a domain can be
 handled by a retraction-coretraction argument if the domain is smooth enough).

\begin{theorem}
  \label{4.4}
 Let $G\subset \R^n$ be a sufficiently smooth domain, and let $p,q\in(1,\infty)$,  $k\in\N$, and $s\in (0,k)$. Then
  \[ B_{pq}^{s}(G;X) = (L^p(G;X), W_p^k(G;X))_{s/k,q}.\]
\end{theorem}

From this theorem and the description of the trace spaces as real interpolation space, one can easily obtain $\gamma_0
W_p^k(G;X) = B_{pp}^{k-1/p}(G;X)$, where $\gamma_0 u := u|_{\partial G}$ stands for the trace on the boundary of the
domain. This typical loss of derivatives of order $1/p$ leads to non-integer Sobolev spaces for inhomogeneous boundary
data. For parabolic equations, we also have to consider time and boundary traces of the solution space:

\begin{corollary}
  \label{4.5}
  Let $G\subset\R^n$ be a sufficiently smooth domain, $J=(0,T)$ with $T\in (0,\infty]$,  $k\in\N$,
  and let $\X = W_p^1(J;L^p(G)\cap L^p(J;W_p^k(G))$
   (the typical parabolic solution space).

  a) For the time trace $\gamma_t\colon u\mapsto u|_{t=0}$, we obtain the trace space
  \[ \gamma_t \X = B_{pp}^{k-k/p}(G).\]

  b) For the boundary trace $\gamma_0\colon u\mapsto u|_{\partial G}$, we obtain the trace space
 \[ \gamma_0\X = B_{pp}^{1-1/(kp)}(J;L^p(\partial G))\cap L^p(J; B_{pp}^{k-1/p}(\partial G)).\]
\end{corollary}

\begin{proof}
  We only give the main ideas for a proof and refer to \cite{Denk-Hieber-Pruess07}, Section~3, for a complete version.

  a) By Lemma~\ref{2.4a} a), we have $\gamma_t\X = (L^p(G), W_p^k(G))_{1-1/p,p}$ which equals $B_{pp}^{k-k/p}(G)$
   due to Theorem~\ref{4.4} with $p=q$.

  b) Locally, we can choose a coordinate system such that the inner normal vector is the $x_n$-variable. Then we have
   to take the trace with respect to $x_n$ instead of $t$ which gives by Lemma~\ref{2.4a} a real interpolation space
   again. Computing the real interpolation space of the intersection then gives a Besov space both with respect to
    time and with respect to the other space variables.
\end{proof}

\begin{remark}
  \label{4.6}
  In the above corollary, we have considered functions which are $L^p$ in time and $L^p$ in space. If one considers
  functions which are $L^p$ in time and $L^q$ in space with $p\not=q$, a result similar to Corollary~\ref{4.5} holds,
   but now also Triebel-Lizorkin spaces appear. More precisely, for $\X := W_p^1(J;L^q(G))\cap L^p(J;W_q^k(G))$ we
   obtain (see \cite{Denk-Hieber-Pruess07}, Section~6, and \cite{Meyries-Veraar14}, Section~4)
  \begin{align*}
    \gamma_t \X & = B_{qp}^{k-k/p}(G),\\
    \gamma_0\X & = F_{pq}^{1-1/(kq)}(J;L^p(\partial G))\cap L^q(J; B_{qq}^{k-1/q}(\partial G)).
  \end{align*}
\end{remark}

\section{Parabolic PDE systems in the whole space}

As a first application of the previous results, we now consider parabolic systems of partial differential equations in
the whole space $\R^n$. In the following, let $1<p<\infty$ and $\C_+ := \{ z\in\C: \Re z >0\} = \Sigma_{\pi/2}$. We
assume that we have a linear differential operator $A= A(x,D)$ of the form
\[ A(x,D) = \sum_{|\alpha|\le 2m} a_\alpha(x) D^\alpha\]
with $m\in\N$ and matrix-valued coefficients $a_\alpha\colon\R^n\to \C^{N\times N}$. Recall that $D := -i\partial$. The
definition of parabolicity below is based on
 the concept of parameter-ellipticity which was developed by Agmon \cite{Agmon62} and Agranovich-Vishik
 \cite{Agranovich-Vishik64}.

For the formal differential operator  $A= A(x,D)$, we define its symbol
\[ a(x,\xi) := \sum_{|\alpha|\le 2m} a_\alpha(x)\xi^\alpha\]
and the principal symbol
\[ a_0(x,\xi):=\sum_{|\alpha|=2m} a_\alpha(x)\xi^\alpha.\]
Both symbols map $\R^n\times \R^n$ into
$\C^{N\times N}$. The $L^p$-realization
 $A_p$ of $A(x,D)$ is defined as the unbounded linear operator $A_p\colon
L^p(\R^n;\C^N)\supset D(A_p)\to L^p(\R^n;\C^N)$ with
\[ D(A_p) := W_p^{2m}(\R^n;\C^N),\; A_p u := A(x,D) u \quad (u\in
W_p^{2m}(\R^n;\C^N)).\]

\begin{definition}
  \label{5.1} The operator $A(x,D)$ is called parameter-elliptic with angle $\phi\in (0,\pi]$ if
  \begin{equation}\label{5-1}
 \big| \det(a_0(x,\xi)-\lambda)\big| \ge C_P \big( |\xi|^{2m} +
 |\lambda|\big)^N\quad \big( x\in\R^n,\, (\xi,\lambda)\in
 (\R^n\times \bar\Sigma_\phi)\setminus\{0\}\big).
 \end{equation}
 If this holds for $\phi = \frac\pi 2$ (i.e., $\Sigma_\phi = \C_+$), then $\partial_t - A$ is called parabolic.
\end{definition}

\begin{remark}
  \label{5.2}
  a) For every fixed  $x\in\R^n$, the map
  $(\xi,\lambda)\mapsto p(x,\xi,\lambda) := \det(a_0(x,\xi)-\lambda)$
  is quasi-homogeneous in the sense that
  \[ p(x,r\xi,r^{2m}\lambda) = r^{2mN} p(x,\xi,\lambda)\quad \big( r>0,\,
  (\xi,\lambda)\in(\R^n\times \bar\Sigma_\phi)\setminus\{0\}\big).\]
  Therefore, it is sufficient to consider the compact set $\{(\xi,\lambda):
  |\xi|^{2m} + |\lambda|=1\}$. The operator $A(x,D)$ is parameter-elliptic if and only if
  \[ \inf\big\{ |\det (a_0(x,\xi)-\lambda)|: x\in\R^n,\,
  (\xi,\lambda) \in\R^n\times\bar \Sigma_\phi\,\text{ with }|\xi|^{2m} +
  |\lambda| = 1\big\} >0.\]

  b) If  $a_\alpha\in L^\infty(\R^n)$ for all $|\alpha|<2m$, then the lower-order terms of the symbol can be estimated
  uniformly in $x$. Therefore, $A(x,D)$ is parameter-elliptic if and only if
  there exist constants   $C,R>0$ with
  \[ |\det (a(x,\xi)-\lambda)|\ge C(|\xi|^{2m}+|\lambda|)^N\quad
  (x\in\R^n,\, \lambda\in\bar\Sigma_\phi,\, |\xi|\ge R).\]
  This is one possible definition of parameter-elliticity and parabolicity for pseudodifferential operators.
  We remark that  the principal
   symbol of a pseudodifferential operator is defined only for so-called classical symbols.
\end{remark}

\begin{remark}
  \label{5.3} If $\partial_t - A(x,D)$ is parabolic in the sense of parameter-ellipticity in the closed sector
  $\bar\C_+$, then  $A(x,D)$ is  also parameter-elliptic in some larger sector  $\bar\Sigma_\theta$
  with $\theta>\frac\pi 2$. In fact, it is easily seen that the set of all angles of rays with respect to $\lambda$,
   in which condition \eqref{5-1}
   holds, is open.
\end{remark}

Following a standard approach in elliptic theory, we first consider the so-called model problem and then use
perturbation results for variable coefficients. The remainder of this section is based on \cite{Denk-Hieber-Pruess03},
Sections 5 and 6, and \cite{Kunstmann-Weis04}, Sections 6 and 7.

\begin{theorem}\label{5.4}
Let $A(D) = \sum_{|\alpha|=2m} a_\alpha D^\alpha$ with constant coefficients $a_\alpha\in\C^{N\times N}\quad
(|\alpha|=2m)$ and without lower-order terms. If $\partial_t  - A(D)$ is parabolic with parabolicity constant $C_P$ in
\eqref{5-1}, then $\rho(A_p)\supset \bar\C_+\setminus\{0\}$, and the set
\[ \big\{ \lambda (\lambda-A_p)^{-1}: \lambda\in
\bar\C_+\setminus\{0\} \big\}\] is $\RR$-bounded. Here, the
$\RR$-bound only depends on  $p, n,m, N, C_P$ and  \[M
:=\sum_{|\alpha|=2m} |a_\alpha|_{\C^{N\times N}}.\] In particular,
 $A_p$ is $\RR$-sectorial with $\RR$-angle larger than $\frac\pi 2$,
and $A_p$ has maximal  $L^q$-regularity for all $q\in (1,\infty)$.
\end{theorem}

\begin{proof}
  Note that because of \eqref{5-1}, for  $\lambda\in\bar\C_+\setminus\{0\}$
  and $\xi\in\R^n$ the symbol $(\lambda-a_0(\xi))^{-1}$ is well defined.
  We show that the family $\{m_\lambda:
  \lambda\in\bar\C_+\setminus\{0\}\}$ with $m_\lambda(\xi) := \lambda
  (\lambda-a_0(\xi))^{-1}$ satisfies the assumptions of Corollary~\ref{3.31}.

  For any $r>0$ we have $r^{2m}\lambda-a_0(r\xi) =
  r^{2m}(\lambda-a_0(\xi))$. Therefore, the map   $(\xi,\lambda)\mapsto
  \frac1\lambda(\lambda-a_0(\xi))$ is quasi-homogeneous in $(\xi,\lambda)$
  of degree  0, and the same holds for its inverse
  $(\xi,\lambda)\mapsto \lambda(\lambda-a_0(\xi))^{-1}$. By Lemma~\ref{3.24}, $m_\lambda$ satisfies the Mikhlin condition
  uniformly with respect to $\lambda$. Now we can apply Corollary~\ref{3.31} to obtain the
  $\RR$-boundedness of $\{\op[m_\lambda]:\lambda\in\bar
  \C_+\setminus\{0\}\}\subset L(L^p(\R^n))$. Because of  $\frac 1\lambda
  \op[m_\lambda](\lambda-A_p) = \id_{W_p^{2m}(\R^n)}$ and $\frac
  1\lambda (\lambda-A_p) \op[m_\lambda]= \id_{L^p(\R^n)}$, we see that
  $\op[m_\lambda] = \lambda (\lambda-A_p)^{-1}$. By Corollary~\ref{3.31},
   $A_p$ is $\RR$-sectorial with angle larger than    $\frac\pi 2$, and
  Theorem~\ref{3.34} implies that  $A_p$ has maximal $L^q$-regularity for
  all $q\in (1,\infty)$. To show the statement on the  $\RR$-bound, we have to quantify the Mikhlin constant.

 For this, we write
  \[ (\lambda-a_0(\xi))^{-1} = \frac1{\det(\lambda-a_0(\xi))}
  b(\xi,\lambda)\]
  with the adjunct matrix $b(\xi,\lambda)$. The coefficients of
  $b(\xi,\lambda)$ are determinants of  $(N-1)\times
  (N-1)$-matrices which are constructed by omitting one row and one column of the matrix
 $\lambda-a_0(\xi)$. Therefore, we obtain
  \[ \|b(\xi,\lambda)\|_{\C^{N\times N}} \le C(m,n,M,N) (|\xi|^{2m} +
  |\lambda|)^{N-1}.\]
  Due to \eqref{5-1}, we get
  \[ \|\lambda(\lambda-a_0(\xi))^{-1}\|_{\C^{N\times N}} \le C(m,n,M,N,C_p)
  \frac{|\lambda|}{|\xi|^{2m}+|\lambda|} \le C(m,n,M,N,C_p).\]
  For the derivatives, we note that
  \begin{align*}
    \Big\| \xi_k\frac\partial{\partial \xi_k} a_0(\xi)\Big\|_{\C^{N\times N}} & = \Big\|
    \xi_k\frac\partial{\partial \xi_k} \sum_{|\alpha|=2m} a_\alpha
    \xi^\alpha \Big\|_{\C^{N\times N}}  \le \sum_{|\alpha|=2m} \|a_\alpha\|_{\C^{N\times N}} \Big| \xi_k
    \frac\partial{\partial \xi_k} \xi^\alpha \Big|\\
    & \le 2mM |\xi|^{2m}.
  \end{align*}
  Iteratively, we obtain $\|\xi^\alpha \partial_\xi^\alpha a_0(\xi) \|_{\C^{N\times N}} \le C(m,n,M,N)
  |\xi|^{2m}$ for all $\alpha\in\{0,1\}^n$. In the same way, the derivative of  $b(\xi,\lambda)$ can be estimated. This yields
  \[ \|\xi^\alpha \partial_\xi^\alpha b(\xi,\lambda)\|_{\C^{N\times N}} \le
  C(m,n,M,N)(|\xi|^{2m}+|\lambda|)^{N-1}.\]
  With the product rule (Leibniz rule) we have for the inverse matrix the inequality
  \[ \|\xi^\alpha D_\xi^\alpha (\lambda-a_0(\xi))^{-1} \|_{\C^{N\times N}} \le
  C(m,n,M,N,C_P) (|\xi|^{2m} + |\lambda|)^{-1},\]
 and therefore $\|\xi^\alpha D_\xi^\alpha m_\lambda(\xi)\|_{\C^{N\times N}} \le C(m,n,M,N,C_P)$
  for all $\alpha\in\{0,1\}^n, \lambda\in\bar\C_+\setminus\{0\}$
  and all $\xi\in\R^n$. By Corollary~\ref{3.31}, the $\RR$-bound of
  $\{\lambda(\lambda-A_p)^{-1}:\lambda\in\bar\C_+\setminus\{0\}\}$
  only depends on $m,n,M,N,C_P$, and $p$.
\end{proof}

To generalize the above result to operators with variable coefficients, we need perturbation results for
$\RR$-boundedness. For this, we  define for an $\RR$-sectorial operator $A$ with $\RR$-angle $\phi_{\RR}(A)$ and for
$\theta\in (0,\phi_{\RR}(A))$:
\begin{align*}
  M_\theta(A) &:= \sup\big(\{ \|\lambda (\lambda-A)^{-1}\|: \lambda
  \in\Sigma_\theta\}\big),\\
  \tilde M_\theta(A) &:= \sup\big(\{ \|A (\lambda-A)^{-1}\|: \lambda
  \in\Sigma_\theta\}\big),\\
  R_\theta(A) &:= \RR\big(\{ \lambda (\lambda-A)^{-1}: \lambda
  \in\Sigma_\theta\}\big),\\
  \tilde R_\theta(A) &:= \RR\big(\{ A (\lambda-A)^{-1}: \lambda
  \in\Sigma_\theta\}\big).
\end{align*}
Note that  $\tilde M_\theta(A)$ is finite because of
$A(\lambda-A)^{-1} = \lambda(\lambda-A)^{-1} - 1$, and the same holds for
$\tilde R_\theta(A)$.

\begin{theorem}
  \label{5.5} Let $X$ be a Banach space and  $A$ be an $\RR$-sectorial
  operator in $X$  with angle $\phi_{\RR}(A)>0$. Further, let   $\theta\in
  (0,\phi_{\RR}(A))$, and let  $B$ be a linear operator  in $X$ with $D(B)\supset
  D(A)$ and
  \begin{equation}\label{5-2}
  \| Bx\| \le a\|Ax\|\quad (x\in D(A)).
  \end{equation}
  If $a<\frac1{\tilde R_\theta(A)}$, then $A+B$ is
  $\RR$-sectorial, too, with angle larger or equal to  $\theta$ and
  \[ R_\theta(A+B) \le \frac{R_\theta(A)}{1-a\tilde
  R_\theta(A)}\,.\]
\end{theorem}

\begin{proof}
  For $\lambda\in\bar{\Sigma}_\theta\setminus\{0\}$ one obtains
  \[ \|B(\lambda-A)^{-1} x\| \le a\| A(\lambda-A)^{-1} x\| \le a
  \tilde M_\theta(A) \|x\|\quad (x\in X).\]
Because of $a<\frac1{\tilde R_\theta(A)}$, the operator   $1+B(\lambda-A)^{-1}$ is invertible, and we get
\begin{align*}
 (\lambda-(A+B))^{-1} & = (\lambda-A)^{-1} \big[
1+B(\lambda-A)^{-1}\big]^{-1} \\
&= (\lambda-A)^{-1}
\sum_{n=0}^\infty (-B(\lambda-A)^{-1})^n.
\end{align*}
In particular,  $\rho(A+B)\supset\Sigma_\theta$. By definition of  $\RR$-boundedness and due to the assumption, we get
\[ \RR(\{ B (\lambda-A)^{-1}: \lambda\in \Sigma_\theta\}) \le a \RR ( \{ A(\lambda-A)^{-1}: \lambda\in
\Sigma_\theta\})= a\tilde R_\theta(A).\] Inserting this into the above series yields
\[ R_\theta(A+B) \le \frac{R_\theta(A)}{1-a\tilde R_\theta(A)}\,.\]
This shows that also $A+B$ is $\RR$-sectorial with $\RR$-angle  $\ge \theta$.
\end{proof}

The second perturbation results deals with the case where we have an additional term $\|x\|$ on the right-hand side of
\eqref{5-2}. However, now the $\RR$-sectoriality of the operator holds only with an additional shift in the operator.

\begin{theorem}
  \label{5.6} Let $A$ be $\RR$-sectorial with angle $\phi_{\RR}(A)>0$,
  and let $\theta\in
  (0,\phi_{\RR}(A))$. Let $B$ be a linear operator satisfying
  $D(B)\supset D(A)$ and
  \[\|Bx\|\le a\|Ax\| +b\|x\|\quad (x\in D(A))\]
  with constants  $b\ge 0 $ and $0\le a <\big[ \tilde
  M_\theta(A)\tilde R_\theta(A)\big]^{-1}$. Then $A+B-\mu$
  is $\RR$-sectorial for
  \[ \mu > \frac{bM_\theta(A)\tilde R_\theta(A)}{1-a\tilde
  M_\theta(A)\tilde R_\theta(A)}.\]
  For the $\RR$-angle, we have  $\phi_{\RR}(A+B-\mu)\ge \theta$.
\end{theorem}

\begin{proof}
  For $\mu>0$, the following inequalities hold
  \begin{align*}
    \|B(A-\mu)^{-1} x \| & \le a\|A(A-\mu)^{-1}x\| +
    b\|(A-\mu)^{-1}x\|\\
    & \le\big( a\tilde M_\theta(A) + \tfrac b\mu M_\theta(A)\big)
    \|x\|\quad (x\in X).
  \end{align*}
  Therefore, $B$ satisfies the assumption of Theorem~\ref{5.5} with $A$ being replaced by $A-\mu$.
   In Theorem~\ref{5.5}, the condition  for the constants is given by
    $c(\mu)\tilde R_\theta(A)<1$, where
  $c:= a\tilde M_\theta(A) + \frac b\mu M_\theta(A)$. Because of $a\tilde
  M_\theta(A)<1$, this is the case if
  \[ \mu > \frac{bM_\theta(A)\tilde R_\theta(A)}{1-a\tilde
  M_\theta(A)\tilde R_\theta(A)}.\]

\vspace*{-2em}
\end{proof}

The above perturbation results allow us to treat small perturbations in the principal part of the differential operator.
\begin{lemma}
  \label{5.7} Let $A(x,D) =\sum_{|\alpha|=2m} a_\alpha D^\alpha$ with $a_\alpha\in\C^{N\times N}$ for $|\alpha|=2m$,
   and assume $\partial_t - A(x,D)$ to be parabolic with
  constant  $C_P$. Then there exists some
  $\theta>\frac\pi 2$ such that $A(x,D)$ is parameter-elliptic in
  $\bar\Sigma_\theta$, and there exist $\epsilon>0$ and
  $K>0$ such that for all operators $B(x,D)=\sum_{|\alpha|=2m}
  b_\alpha(x)D^\alpha$ with $b_\alpha\in L^\infty(\R^n;\C^{N\times
  N})$ and
  \[ \sum_{|\alpha|=2m} \|b_\alpha\|_\infty< \epsilon\]
  the inequality
  \[ \RR\Big(\big\{ \lambda(\lambda-(A_p+B_p))^{-1}:
  \lambda\in\bar\Sigma_\theta\setminus\{0\}\big\}\Big)\le K\]
  holds. Here, $\epsilon$ and $K$ only depend on  $n,p,m,N,C_P$.
\end{lemma}

\begin{proof}
  Let $\epsilon>0$ and $f\in W_p^{2m}(\R^n;\C^N)$. Then the inequality
  \[ \|Bf\|_{L^p(\R^n;\C^N)} \le \sum_{|\alpha|=2m}
  \|b_\alpha\|_\infty \|D^\alpha f\|_{L^p(\R^n;\C^N)}\le \epsilon
  \max_{|\alpha|=2m}\|D^\alpha f\|_{L^p(\R^n;\C^N)},\]
   holds if $B$ satisfies the above condition. We write
  \[ D^\alpha f = (\mathscr F^{-1} m_\alpha \mathscr F) A(D)f \]
  with
  \[ m_\alpha(\xi) := \xi^\alpha \Big(\sum_{|\beta|=2m} a_\beta
  \xi^\beta\Big)^{-1}.\]
  Then $m_\alpha \in C^\infty(\R^n\setminus\{0\};\C^{N\times
  N})$, and $m_\alpha$ is homogeneous of degree $0$ and therefore satisfies the Mikhlin condition.
  Consequently, there exists some $C_1>0$ such that we have
  \[ \| \op[m_\alpha]\|_{L(L^p(\R^n;\C^N))} \le C_1\quad
  (|\alpha|=2m).\]
  Choose  $\epsilon < \big[ C_1 (\tilde R_\theta(A)+1)\big]^{-1}$.
  Then
  \[ \|Bf\|_{L^p(\R^n;\C^N)} \le \epsilon C_1
  \|Af\|_{L^p(\R^n;\C^N)} \le a \|Af\|_{L^p(\R^n;\C^N)}\]
  with $a=\frac1{\tilde R_\theta(A)+1}$. By Theorem~\ref{5.5}, the operator
  $A_p+B_p$ is $\RR$-sectorial with angle    $\ge \theta$, and
  \[ R_\theta(A+B) \le \frac{R_\theta(A)}{1-a\tilde R_\theta(A)}
  =:K.\]

  \vspace*{-2em}
\end{proof}

In the next step, we consider an operator $A$ whose coefficients in the principal part are bounded and uniformly
continuous. We can reduce  this situation to the small perturbation from the last lemma by introducing an infinite
partition of unity. This is done in the following lemma.

\begin{lemma}
  \label{5.8}
  For every $r>0$ there exists  $\phi\in\mathscr D(\R^n)$ with $0\le \phi\le
  1$, $\supp\phi\subset(-r,r)^n$ and
  \[ \sum_{\ell\in r\Z^n} \phi_\ell^2(x) =1 \quad (x\in\R^n).\]
  Here, $\phi_\ell(x) := \phi(x-\ell)$.
\end{lemma}

\begin{proof}
  a) We first consider the case $r=1$ and $n=1$. Choose some
  $\phi_1\in\mathscr D(\R)$ with $\phi_1>0$ in $(-\frac34,\frac34)$,
  $\supp\phi_1=[-\frac34,\frac34]$, and $\phi_1(x) = \phi_1(-x)$ for all
  $x\in \R$. We set
  \[\phi(x) :=\begin{cases}
    \sqrt{\frac{\phi_1^2(x)}{\phi_1^2(x)+\phi_1^2(1-x)}} & \text{ if } x\in [0,\frac34],\\
    0, & \text{ if } x\in (\frac34,\infty),
  \end{cases}\]
  and $\phi(x) := \phi(-x)$ for $x< 0$.
  Then $\supp\phi\subset(-1,1)$, and for $x\in [0,1]$ we obtain
  \begin{align*}
    \sum_{\ell\in\Z} \phi_\ell^2(x) &= \phi^2(x) +\phi^2(x-1) =
    \phi^2(x)+\phi^2(1-x)\\
    &  = \frac{\phi_1^2(x)}{\phi_1^2(x)+\phi_1^2(1-x)} +
    \frac{\phi_1^2(1-x)}{\phi_1^2(1-x)+\phi_1^2(x)} = 1.
  \end{align*}
 As $\sum_{\ell\in\Z}\phi_\ell^2$ is periodic with period $1$, we have
  $\sum_{\ell\in\Z}\phi_\ell^2=1$ in $\R$.

  b) In the general case, define $\phi^{(n)}(x) :=\prod_{j=1}^n
  \phi(\frac{x_j}r)$ with $\phi$ from part a). Then
  \begin{align*}
    \sum_{\ell\in r\Z^n} (\phi^{(n)})^2(x-\ell) & = \sum_{\ell\in r\Z^n}
    \prod_{j=1}^n \phi^2\Big(\frac{x_j-\ell_j}r\Big) = \sum_{\ell\in\Z^n}
    \prod_{j=1}^n\phi^2(y_j-\ell_j)\\
    & = \prod_{j=1}^n \sum_{\ell_j\in\Z} \phi^2(y_j-\ell_j) = 1
  \end{align*}
  for $y:= \frac xr$.
\end{proof}

We now come to the main result of this section. Here, $\buc(\R^n)$ stands for the space of all bounded and uniformly
continuous functions.

\begin{theorem}
  \label{5.9}
  Let $A(x,D) = \sum_{|\alpha|\le 2m} a_\alpha(x) D^\alpha$ with
  \begin{align*}
    a_\alpha & \in \buc(\R^n;\C^{N\times N})\quad (|\alpha|=2m),\\
    a_\alpha & \in L^\infty(\R^n;\C^{N\times N}) \quad
    (|\alpha|<2m).
  \end{align*}
  Let $1<p<\infty$. If $\partial_t - A(x,D)$ is parabolic, then there exist
  $\theta>\frac\pi 2$ and $\mu >0$ such that $A_p-\mu$ is
  $\RR$-sectorial with angle $\theta $. In particular,
  $A_p-\mu$ has maximal $L^q$-regularity for all $1<q<\infty$.
\end{theorem}

\begin{proof}
  As $A(x,D)$ is parameter-elliptic in $\bar\C_+$ by assumption, there exists a $\theta>\frac\pi 2$ such that $A(x,D)$
   is still  parameter-elliptic in  $\bar\Sigma_\theta$ (Remark~\ref{5.3}). The proof of the theorem uses localization
   and is done   in several steps. We first explain the ideas.

  \begin{enumerate}
    [(1)]
\item We fix the coefficients of $A$ at some point $\ell\in\Gamma$, where the grid
  $\Gamma\subset\R^n$ is chosen fine enough such that in each cube the localized operator  $A^\ell$ is a small
  perturbation of the model problem with frozen coefficient. Here, we apply Lemma~\ref{5.7}
to see that  $A^\ell$ is still $\RR$-sectorial.

\item We consider the sequence $A:= (A^\ell)_{\ell\in\Gamma}$ of all localized operators and show that this defines
 an $\RR$-sectorial operator in some suitably chosen sequence space  $X_0$.

\item  The $L^p$-realization $A_p$ and the operator $A$ have the same properties up to lower-order perturbations.
More precisely, we have   $JA_p=AJ$ and $A_pP=PA$ modulo lower order operators, where $J$ and $P$ are the localization
and the patching operator, respectively.

\item With the help of the interpolation inequality for Sobolev spaces, the lower-order operators can be seen as a
small perturbation, and therefore the $\RR$-sectoriality of $A$ implies the $\RR$-sectoriality of $A_p$.
  \end{enumerate}

In detail, these steps can be done in the following way.

  \textbf{(1)} Choose $\epsilon = \epsilon(n,p,m,N,C_p)$ as in Lemma~\ref{5.7} for the operator $\sum_{|\alpha|=2m}
  a_\alpha(\ell)D^\alpha$ with $\ell\in\Z^n$. As $a_\alpha\in \buc(\R^n;\C^{N\times N})$,
  there exists a $\delta>0$ with
  \[ \sum_{|\alpha|=2m} |a_\alpha(x)-a_\alpha(y)| <\epsilon\quad
  (|x-y|\le \delta).\]
  Now choose $r\in (0,\delta)$ and $\phi\in\mathscr D(\R^n)$ as in Lemma~\ref{5.8}.
We write $Q:=(-r,r)^n$ and $Q_\ell := Q+\ell $ for $\ell\in
  r\Z^n =:\Gamma$. Choose $\psi\in\mathscr D(\R^n)$ with
  $\supp\psi\subset Q$, $0\le \psi\le 1$, $\psi=1$ on $\supp\phi$,
  and set $\psi_\ell(x) := \psi(x-\ell) \;(\ell\in\Z)$. Define the coefficients
  \[
    a_\alpha^\ell(x) := \begin{cases}
      a_\alpha(x),& x\in Q_\ell,\\
      a_\alpha(\ell),& x\not\in Q_\ell
    \end{cases}\quad (\ell\in\Gamma,\,|\alpha|=2m)
\]
and the operator $A^\ell(x,D) := \sum_{|\alpha|=2m} a_\alpha^\ell(x) D^\alpha$. For the principal part, we obtain
$A_0(x,D) = A^\ell(x,D)$ $(x\in Q_\ell)$ and therefore $A_0(x,D)u = A^\ell(x,D)u$ for all $u\in W_p^{2m}(\R^n;\C^N)$
with $\supp u\subset Q_\ell$.

\textbf{(2)} Define $X_k := \ell_p(\Gamma;W_p^k(\R^n;\C^N))$ for $k\in\N_0$ and the operator $A\colon X_0\supset D(A)\to X_0$
by $D(A) := X_{2m} $ and
\[ A(u_\ell)_{\ell\in\Gamma} := (A^\ell u_\ell)_{\ell\in\Gamma}.\]

By Lemma~\ref{5.7}, the operator $A^\ell$ is $\RR$-sectorial with $R_\theta(A^\ell)\le K$, where $K$ does not depend on
$\ell$. We show that the same holds for
 $A$. For this, let $T_j = \lambda_j
(A-\lambda_j)^{-1}$ with $\lambda_j\in\Sigma_\theta$ and $x_j = (f_\ell^{(j)})_{\ell\in\Gamma}\in X_0$ for
$j=1,\dots,J$. Then we obtain
\begin{align*}
  \Big\| & \sum_{j=1}^J r_j T_j x_j \Big\|_{L^p([0,1];X_0)} = \\
  & = \Big( \int_0^1 \Big\| \sum_{j=1}^J r_j(t) T_j x_j\Big\|_{X_0}^p
  dt\Big)^{1/p} \\
  & = \Big( \int_0^1 \sum_{\ell\in\Gamma} \Big\| \sum_{j=1}^J r_j(t)
  \lambda_j (A^\ell -\lambda_j)^{-1} f_\ell^{(j)}
  \Big\|_{L^p(\R^n;\C^N)}^p dt\Big)^{1/p}\\
  & = \Big(\sum_{\ell\in\Gamma} \int_0^1 \Big\| \sum_{j=1}^J r_j(t)
  \lambda_j (A^\ell-\lambda_j)^{-1}
  f_\ell^{(j)}\Big\|_{L^p(\R^n;\C^N)}^p dt\Big)^{1/p}\\
  & = \Big( \sum_{\ell\in\Gamma} \Big\| \sum_{j=1}^J r_j \lambda_j
  (A^\ell -\lambda_j)^{-1} f_\ell^{(j)}
  \Big\|^p_{L^p([0,1];L^p(\R^n;\C^N))}\Big)^{1/p} \\
  & \le \Big(\sum_{\ell\in\Gamma} \big[ R_\theta(A^\ell)\big]^p \Big\|
  \sum_{j=1}^J r_j f_\ell^{(j)}
  \Big\|_{L^p([0,1];L^p(\R^n;\C^N))}^p\Big)^{1/p} \\
  &\le K \Big\| \sum_{j=1}^J r_j x_j\Big\|_{L^p([0,1];X_0)},
\end{align*}
i.e. $R_\theta(A)\le K$.

Now we consider the localization operator $J\colon L^p(\R^n;\C^N)\to
X_0,\; f\mapsto (\phi_\ell f)_\ell$. As we have
\[ \sum_{\ell\in\Gamma} \|\phi_\ell f\|_{L^p(\R^n;\C^N)}^p \le
\sum_{\ell\in\Gamma} \|\chi_{Q_\ell} f\|_{L^p(\R^n;\C^N)}^p =
2^N\|f\|_{L^p(\R^n;\C^N)}^p,\] the operator  $J$ is continuous. In the same way, one sees that
$J\in L(W_p^{2m}(\R^n;\C^N), X_{2m})$.

Analogously, the patching operator $P$ is defined by
\[P\colon X_0\to L^p(\R^n;\C^N),\, (f_\ell)_{\ell\in\Gamma}\mapsto
\sum_{\ell\in\Gamma} \phi_\ell f_\ell.\] Note here that the sum is locally finite.
We obtain  $P\in L(X_0,L^p(\R^n;\C^N))$ and
$PJ=\id_{L^p(\R^n;\C^N)}$ because of $PJf =
\sum_{\ell\in\Gamma}\phi_\ell^2 f= f$.

\textbf{(3)} Now let $A_p$ be the  $L^p(\R^n;\C^N)$-realization of  $A(x,D)$ and $A_{p,0}$ the
$L^p(\R^n;\C^N)$-realization of $A_0(x,D)$. Then for $u\in W_p^{2m}(\R^n;\C^N)$ and $\ell\in\Gamma$ the following
equality holds:
\begin{align*}
  \phi_\ell A_p u &= A_p(\phi_\ell u) + (\phi_\ell A_p - A_p\phi_\ell) u
  \\
  & = A^\ell(\phi_\ell u) + (A_p - A_{p,0}) \psi_\ell \phi_\ell u +
  \sum_{k:Q_k\cap Q_\ell\not=\emptyset} (\phi_\ell
  A_p-A_p\phi_\ell)\phi_k^2 u.
\end{align*}
Thus, $JA_p = AJ + BJ$ with
\[ B\big( (u_\ell)_{\ell\in\Gamma}\big) := \Big(
(A_p-A_{p,0})\psi_\ell u_\ell + \sum_{k:Q_k\cap Q_\ell\not=\emptyset}
(\phi_\ell A_p-A_p\phi_\ell)\phi_k u_k\Big)_{\ell\in\Gamma}.\]
Writing $B((u_\ell)_\ell) = ( \sum_{k\in\Gamma} B_{k\ell}
u_\ell)_{k\in\Gamma}$, we see that  $B_{k\ell}$ is a differential operator of order not greater than $\le 2m-1$, and
the number of elements in each row of the infinite matrix  $(B_{k\ell})_{k,\ell}$ is bounded. As
$a_\alpha\in L^\infty(\R^n;\C^N)$, this yields $B\in L(X_{2m-1}, X_0)$.

Analogously, we obtain for $(u_\ell)_{\ell\in\Gamma}\in X_{2m}$ the equality
\begin{align*}
  (A_p P-PA)& (u_\ell)_{\ell\in\Gamma} =
  A_p\Big(\sum_{\ell\in\Gamma} \phi_\ell u_\ell\Big) -
  \sum_{\ell\in\Gamma} \phi_\ell A^\ell u_\ell \\
  & = \sum_{\ell\in\Gamma} \phi_\ell (A_p-A_{p,0}) u_\ell +
  \sum_{k\in\Gamma} (A_p\phi_k - \phi_k A_p) u_k \\
  & = \sum_{\ell\in\Gamma} \phi_\ell (A_p-A_{p,0}) u_\ell +
  \sum_{k\in\Gamma} \sum_{\ell: Q_k\cap Q_\ell\not=\emptyset}
  \phi_\ell^2 (A_p\phi_k -\phi_k A_p) u_k\\
  & = \sum_{\ell\in\Gamma} \phi_ \ell\Big[ (A_p-A_{p,0}) u_\ell +
  \sum_{k:Q_k\cap Q_\ell\not=\emptyset} \phi_\ell (A_p\phi_k -\phi_k
  A_p) u_k\Big]\\
  &= PD(u_\ell)_{\ell\in\Gamma}
\end{align*}
with
\[ D(u_\ell)_{\ell\in\Gamma} := \Big( (A_p-A_{p,0}) u_\ell +
\sum_{k:Q_k\cap Q_\ell\not=\emptyset} (A_p\phi_k - \phi_k A_p)
u_k\Big)_{\ell\in\Gamma}.\]
In the same way as before, we see that $D\in L(X_{2m-1},X_0)$.

\textbf{(4)} We apply the interpolation inequality for Sobolev spaces and obtain for every
  $\epsilon>0$ the inequality
\begin{align*}
  \|B(u_\ell)_{\ell\in\Gamma}\|_{X_0} & + \|
  D(u_\ell)_{\ell\in\Gamma}\|_{X_0} \le C
  \|(u_\ell)_{\ell\in\Gamma}\|_{X_{2m-1}} \\
  & \le \epsilon \|(u_\ell)_{\ell\in\Gamma}\|_{X_{2m}} + C_\epsilon
  \| (u_\ell)_{\ell\in\Gamma}\|_{X_0} \quad (u\in X_{2m}).
\end{align*}
Due to Theorem~\ref{5.6}, there exists a $\mu>0$ such that $A+B-\mu$ and
$A+D-\mu$ are both $\RR$-sectorial with angle $\ge\theta$.

Let $u\in W_p^{2m}(\R^n;\C^N)$ and $f := (\lambda+\mu-A_p) u \in
L^p(\R^n;\C^N)$. Then
\[ Jf = J(\lambda+\mu-A_p) u = (\lambda+\mu - (A+B)) Ju,\]
and therefore
\[ u = PJu = P(\lambda+\mu-(A+B))^{-1} Jf.\]
In particular,  $\lambda+\mu-A_p$ is injective.

On the other hand, for  $f\in L^p(\R^n;\C^N)$ we get
\begin{align*}
  f & = PJf = P(\lambda+\mu-(A+D)) (\lambda+\mu -(A+D))^{-1} Jf\\
  & = (\lambda+\mu-A_p) P (\lambda+\mu-(A+D))^{-1} Jf \in
  R(\lambda+\mu-A_p),
\end{align*}
i.e., $\lambda+\mu-A_p$ is surjective, too. Therefore, $\lambda+\mu\in
\rho(A_p)$ and
\[ (\lambda+\mu-A_p)^{-1} = P (\lambda+\mu-(A+D))^{-1} J.\]
Because of $P\in L(X_0,L^p(\R^n;\C^N))$, $J\in L(L^p(\R^n;\C^N), X_{2m})$, and $R_\theta(A+D-\mu)<\infty$, it follows
that  $R_\theta(A_p-\mu)<\infty$, and  $A_p-\mu$ is $\RR$-sectorial with angle greater or equal to $ \theta$.
\end{proof}

\section{Parabolic boundary value problems}

In the last section, we considered parabolic systems in the whole space. Now we want to show that similar results also
hold for boundary value problems in sufficiently smooth domains. In addition to the parameter-ellipticity of the
operator $A$, we now have to impose a condition on the boundary operators called Shapiro-Lopatinskii condition. For a
reference for this condition, we mention,e.g., \cite{Wloka87}, \S 11.

\subsection{The Shapiro-Lopatinksii condition}

In the following, let $p\in (1,\infty)$, and let $G\subset\R^n$ be a bounded domain. We consider a linear partial
differential operator  $A=A(x,D)$ of the form
\[ A(x,D) = \sum_{|\alpha|\le 2m} a_\alpha(x) D^\alpha\]
with $m\in \N$, $a_\alpha\colon\bar G\to\C$ and boundary operators
$B_1,\dots, B_m$ of the form
\[ B_j(x',D) = \sum_{|\beta|\le m_j} b_{j\beta}(x') \gamma_0 D^\beta\]
with $m_j<2m$, $b_{j\beta}\colon \partial G\to \C$. Here,
$\gamma_0$ stands for the boundary trace  $u\mapsto u|_{\partial G}$, which is
a bounded linear map
\[ \gamma_0\colon W_p^k(\Omega)\to W_p^{k-1/p}(\partial \Omega), \]
 $k=1,\dots,2m$ if $G$ is a $C^{2m}$-domain. Note here that $W_p^{k-1/p}(\partial\Omega) = B_{pp}^{k-1/p}(\partial
 \Omega)$ is the Sobolev-Slobodeckii space (see Section~4).

The  $L^p$-realization $A_{B,p}$ of the boundary value problem $(A,B)=(A,B_1,\dots,B_m)$ is defined by \[D(A_{B,p})
:=\{ u\in W_p^{2m}(G): B_1(x,D) u = \dots = B_m(x,D) u =0\}\] and $A_{B,p} u := A(x,D) u\quad (u\in D(A_{B,p}))$. We
will assume the following smoothness:
\begin{enumerate}[(i)]
  \item The domain $\Omega$ is bounded and of class $C^{2m}$.
  \item For the coefficients $a_\alpha$ of $A(x,D)$ we have
  \begin{align*}
    a_\alpha& \in C(\bar G)\quad (|\alpha|=2m),\\
    a_\alpha & \in L^\infty(G)\quad (|\alpha|<2m).
  \end{align*}
  \item For the coefficients  $b_{j\beta}$ of $B_j(x',D)$ we have
  \[ b_{j\beta}\in C^{2m-m_j}(\partial G)\quad (|\beta|\le m_j,\,
  j=1,\dots,m). \]
\end{enumerate}
By trace results on Sobolev spaces, we immediately see the following continuity:

\begin{lemma}
  \label{6.1} The operator
  \[ (A,B)\colon W_p^{2m}(G) \to L^p(G) \times \prod_{j=1}^m
  W_p^{2m-m_j-1/p}(\partial G)\]
  is continuous.
\end{lemma}

As usual, we define the principal symbols $a_0(x,\xi) := \sum_{|\alpha|=2m} a_\alpha(x)
\xi^\alpha$ and $b_{j0}(x',\xi) := \sum_{|\beta|=m_j} b_{j\beta}(x')
\xi^\beta$.

\begin{definition}
  \label{6.2} The boundary value problem  $(A,B)$ is called parameter-elliptic
   in the sector $\overline{\Sigma}_\phi$ if:
   \begin{enumerate}
     [(a)]
\item We have $a_0(x,\xi)-\lambda \not=0$ for all $x\in \bar G$ and all $(\xi,\lambda)\in(\R^n\times
    \overline{\Sigma}_\phi)\setminus\{0\}$.

\item The following Shapiro-Lopatinskii condition is satisfied: for all $x'\in\partial G$ and all $(\xi',\lambda)\in
    (\R^{n-1}\times\overline{\Sigma}_\phi)\setminus\{0\}$ the
  ordinary differential equation
  \begin{equation}\label{6-1}
  \begin{aligned}
    (a_0(x',\xi',D_n)-\lambda) v(x_n) & = 0 \quad (x_n>0),\\
    b_{j0}(x',\xi',D_n) v(x_n)\big|_{x_n=0} & = 0\quad
    (j=1,\dots,m),\\
    v(x_n) & \to 0 \quad (x_n\to\infty)
  \end{aligned}
  \end{equation}
  has only the trivial solution. Here, the boundary value problem is written in coordinates corresponding to $x'$.
 These coordinates arise from the original ones by translation and rotation in such a way that the $x_n$-direction
 in the new coordinates is the direction of the inner normal at the point $x'$.
\end{enumerate}

If this holds for the sector $\bar{\Sigma}_{\pi/2} =\{ \lambda\in\C:\Re\lambda\ge 0\}$, the instationary problem
$(\partial_t - A, B)$ is called parabolic.
\end{definition}

Note that (a) implies inequality \eqref{5-1} from Definition~\ref{5.1}, as $\bar G$ is compact and $a_0$ is continuous
in $x$ and homogeneous in $\xi$.

\begin{definition}
  \label{6.3}
  Assume that in the situation of Definition~\ref{6.2}, (a) holds. Then  $A(x,D)-\lambda$ is called proper
    parameter-elliptic  if for all $(x',\xi',\lambda)\in \partial G\times (\R^{n-1}\setminus\{0\})\times\bar  \C_+$,
     the polynomial $a_0(x',\xi',\cdot)-\lambda$ has exactly $m$ roots (including multiplicities)  $\tau_j =
  \tau_j(x',\xi',\lambda)$, $j=1,\dots,m$  with
  positive imaginary part. In this case, define
  \[ a_+(\tau) := a_+(x',\xi',\lambda,\tau) := \prod_{j=1}^m
  (\tau-\tau_j(x',\xi',\lambda)) \in \C[\tau].\]
  We consider the equivalence class $\bar b_{j0} = \bar
  b_{j0}(x',\xi',\lambda,\cdot)\in \C[\tau]/(a_+)$ of $b_{j0}$
  modulo $a_+$, and write  $\bar b_{j0}$ with respect to the canonical basis
    $\bar 1, \bar\tau,\dots,\bar{\tau^{m-1}}\in \C[\tau]/(a_+)$, i.e.
  \[ \begin{pmatrix}
    \bar b_{10}\\ \vdots \\ \bar b_{m0}
  \end{pmatrix} = L\begin{pmatrix} \bar 1\\ \vdots\\
  \bar{\tau^{m-1}}\end{pmatrix} \quad \text{ with }
  L=L(x',\xi',\lambda)\in\C^{m\times m}.\]
  Then $L$ is called the Lopatinskii matrix of $(A,B)$ at the point $x$.
\end{definition}

\begin{lemma}
  \label{6.4}
  Let $A$ be properly parameter-elliptic in $\bar G$. Then the
  Shapiro-Lopatinskii holds if and only if
  \[ \det L(x',\xi',\lambda) \not=0\quad (x'\in \partial G,\,
  (\xi',\lambda)\in (\R^{n-1}\times \bar \C_+)\setminus\{0\}).\]
\end{lemma}

\begin{proof}
  Let $v_j\;(j=1,\dots,m)$ be the solution of
  \begin{align*}
    \big( a_+(x',\xi,D_n)-\lambda\big) v(x_n) & = 0 \quad (x_n> 0),\\
    D_n^{k-1} v(x_n)\big|_{x_n=0} & =\delta_{kj}\quad (k=1,\dots,m).
  \end{align*}
  Then $\{v_1,\dots,v_m\}$ is a basis of the space  $\mathscr
  M_+$ of all stable solutions of the ordinary differential equation
   $\big( a_+(x',\xi,D_n)-\lambda\big) v(x_n)  =
  0$. Therefore, for all $v\in \mathscr M_+$ we have the representation
  $v=\sum_{j=1}^m\lambda_j v_j$ and
  \begin{align*}
    \begin{pmatrix}
      b_{10}(D_n) \\ \vdots \\ b_{m0}(D_n)
    \end{pmatrix} v(x_n)\big|_{x_n=0} & =
        \begin{pmatrix}
      b_{10}(D_n) \\ \vdots \\ b_{m0}(D_n)
    \end{pmatrix} \big(
    v_1(x_n),\dots,v_m(x_n)\big)\big|_{x_n=0} \begin{pmatrix}
      \lambda_1\\ \vdots \\ \lambda_m
    \end{pmatrix}\\
    & =         \begin{pmatrix}
      \bar b_{10}(D_n) \\ \vdots \\ \bar b_{m0}(D_n)
    \end{pmatrix}  \big(
    v_1(x_n),\dots,v_m(x_n)\big)\big|_{x_n=0} \begin{pmatrix}
      \lambda_1\\ \vdots \\ \lambda_m
    \end{pmatrix}\\
&= L \begin{pmatrix}
      D_n^0 \\ \vdots \\ D_n^{m-1}
    \end{pmatrix} \big(
    v_1(x_n),\dots,v_m(x_n)\big)\big|_{x_n=0} \begin{pmatrix}
      \lambda_1\\ \vdots \\ \lambda_m
    \end{pmatrix}\\
    & = L \begin{pmatrix}
      \lambda_1\\ \vdots \\ \lambda_m
    \end{pmatrix}.
  \end{align*}
  Note that $b_{k0}(D_n)v_j(x_n)\big|_{x_n=0} = \bar
  b_{k0}(D_n)v_j(x_n)\big|_{x_n=0}$ holds because $a_+(D_n)v_j(x_n)=0$.
  Therefore, \eqref{6-1} has only the trivial solution if and only if
  $\det L\not=0$.
\end{proof}

\begin{remark}
  \label{6.5}
  a) The condition of Lemma~\ref{6.4} can be formulated in the following way: The boundary conditions are linearly
  independent modulo
  $a_+$, i.e., $\bar b_{10},\dots,\bar b_{m0}$ are linearly independent in
    $\C[\tau]/(a_+)$.

  b) The boundary conditions $B_1,\dots,B_m$ are called completely elliptic if for every proper parameter-elliptic
  $A$ the boundary value problem $(A,B)$ is parameter-elliptic. This is the case for
 \begin{enumerate}[(i)]
    \item $B_j(x',D) = \gamma_0 (\frac{\partial}{\partial x_n})^{j-1}
    \; (j=1,\dots,m)$ (general Dirichlet boundary conditions),
    \item $B_j(x',D) = \gamma_0 (\frac{\partial}{\partial x_n})^{m+j-1}
    \; (j=1,\dots,m)$ (general Neumann boundary conditions).
  \end{enumerate}
  More general, this holds for all boundary conditions of the form
  \[ B_j(x',D) = \gamma_0 \Big(\frac{\partial}{\partial x_n}\Big)^{s+j-1}
    + \text{ lower order terms }\; (j=1,\dots,m),\]
    where $s\in\{0,\dots,m\}$ is fixed. To see this, we have to show that
    $\{\bar{\tau^{s+j-1}}:j=1,\dots,m\}$
    is linearly independent in  $\C[\tau]/(a_+)$. If this is not the case, there exist
      $c_j\in\C$ and $p\in\C[\tau]$ with
    \[ \sum_{j=1}^m c_j \tau^{s+j-1} = p(\tau)a_+(\tau).\]
   Because of $a_+(0)\not=0$, it follows that $\tau^s$ is a divisor of $p(\tau)$. Therefore,   $\sum_{j=1}^m
    c_j\tau^{j-1} = \tilde p(\tau) a_+(\tau)$ with some polynomial $\tilde p$, in contradiction to  $\deg a_+=m$.

    c) If the domain and the coefficients of   $(A,B)$ are infinitely smooth, then for every fixed
     $\lambda\in\bar\C_+$, the coefficients of
    $L(x',\xi',\lambda)$ are symbols of pseudodifferential operators on the closed
     $(n-1)$-dimensional manifold $\partial G$.
\end{remark}

\subsection{The main result on parameter-elliptic boundary value problems}

Under the condition of parameter-ellipticity, one can construct the solution operators for boundary value problems. We
follow the exposition in \cite{Agranovich-Denk-Faierman97}, Section~2, and \cite{Denk-Hieber-Pruess03}, Sections~6 and
7. We start with a remark on ordinary differential equations.

\begin{theorem}
  \label{6.6}
  Let $(A,B)$ be parameter-elliptic in some sector $\bar\Sigma_\phi$, and let $(x',\xi',\lambda)\in\partial G\times
 ((\R^{n-1}\times\bar\Sigma_\phi)\setminus\{0\})$. Choose a closed curve
   $\gamma = \gamma(x',\xi',\lambda)$ in $\{z\in \C:\Im z>0\}$,
  enclosing all roots $\tau_1,\dots,\tau_m$ of $a_+$.
We define $p_\ell$ by
  \[ a_+(x',\xi',\lambda,\tau) = \sum_{\ell=0}^m
  p_\ell(x',\xi',\lambda)\tau^{m-\ell},\]
 and set  $N_k(\tau) := N_k(x',\xi',\lambda,\tau) :=
  \sum_{\ell=0}^{m-k} p_\ell(x',\xi',\lambda)\tau^{m-k-\ell}$ and
  \[ (M_1(\tau),\dots,M_m(\tau)) := \big(
  N_1(\tau),\dots,N_m(\tau)\big) L^{-1}.\]
  Let $w_k(x_n) = w_k(x',\xi',\lambda,x_n)\;(x_n>0)$ be defined by
  \[ w_k(x_n) := \frac1{2\pi i}\int_\gamma
  \frac{M_k(\tau)}{a_+(\tau)} e^{ix_n\tau} d\tau\quad
  (k=1,\dots,m).\]
  Then  $\{w_1,\dots,w_m\}$ is a basis of the stable solution space of
  $a_0(D_n)w = 0$, $w(x_n)\to 0\quad (x_n\to\infty)$ and satisfies the initial conditions
  \[ b_{j0}(x',\xi',\lambda,D_n)w_k(x_n)\big|_{x_n=0} =
  \delta_{jk}\quad (j,k=1,\dots,m).\]
\end{theorem}

\begin{proof}
  (i) We first show that
  \[ \frac{1}{2\pi i}\int_\gamma
  \frac{N_k(\tau)\tau^{j-1}}{a_+(\tau)} d\tau = \delta_{kj}\quad
  (j,k=1,\dots,m).\]
  For this, we replace $\gamma$ by a large ball $\{\tau\in\C: |\tau|=R\}$.
  For $j<k$
  we have $\deg\big( N_k(\tau)\tau^{j-1}\big) = m-k+j-1\le m-2$. Therefore,
  the integrand is of order $O(R^{-2})$ for $R\to\infty$ which shows that the integral vanishes.

  For $j=k$, the integrand equals
  $\frac{\tau^{m-1}+O(\tau^{m-2})}{(\tau-\tau_1)\cdot\ldots\cdot
  (\tau-\tau_m)}$. By the residue's theorem, the integral has the value $1$.

  For $j>k$ we consider
  \begin{align*}
    Q(\tau) & := -a_+(\tau)\tau^{j-k-1} + N_k(\tau)\tau^{j-1} \\
    & =
  -\sum_{\ell=0}^m p_\ell \tau^{m-\ell+j-k-1} +
  \sum_{\ell=0}^{m-k}p_\ell\tau^{m-\ell+j-k-1}.
  \end{align*}
  We obtain $\deg Q = j-2\le m-2$, and therefore
  \[ \int_{B(0,R)} \frac{N_k(\tau)\tau^{j-1}}{a_+(\tau)}\;d\tau =
  \int_{B(0,R)} \frac{a_+(\tau) \tau^{j-k-1} + Q(\tau)}{a_+(\tau)}
  \; d\tau = \int_{B(0,R)} \frac{Q(\tau)}{a_+(\tau)}\; d\tau =0.\]

  (ii) We have modulo $a_+$, i.e., as equality in  $\R[\tau]/(a_+)$:
  \begin{align*}
  \begin{pmatrix}
  \bar b_{10}(\tau) \\ \vdots \\ \bar b_{m0}(\tau)
  \end{pmatrix}
  (\bar M_1(\tau), & \dots,\bar M_m(\tau)) \\
  & =   \begin{pmatrix}
  \bar b_{10}(\tau) \\ \vdots \\ \bar b_{m0}(\tau)
  \end{pmatrix}\big( \bar N_1(\tau),\dots, \bar N_m(\tau)\big) L^{-1}\\
  & = L  \begin{pmatrix}
  \bar 1\\ \vdots \\ \bar{\tau^{m-1}}
  \end{pmatrix}
  \big( \bar N_1(\tau),\dots, \bar N_m(\tau)\big) L^{-1}.
  \end{align*}
  Therefore,
  \begin{align*}
  \Big( \frac1{2\pi i}\int_\gamma \frac{b_{j0}(\tau) M_k(\tau)}{a_+(\tau)} \; d\tau \Big)_{j,k=1,\dots,m} & = L\cdot
  \Big( \frac1{2\pi i}\int_\gamma \frac{\tau^{j-1} N_k(\tau)}{a_+(\tau)} \; d\tau \Big)_{j,k=1,\dots,m}\cdot L^{-1}\\
  & = L\cdot I_m\cdot L^{-1} = I_m.
  \end{align*}
  This yields
  \[ \frac1{2\pi i} \int_\gamma \frac{b_{j0}(\tau) M_k(\tau)}{a_+(\tau)}\; d\tau= \delta_{jk}\quad(j,k=1,\dots,m).\]

  (iii) Define $w_k$ as in the theorem. Because of $\gamma\subset\{z\in\C: \Im z>0\}$, we see that $w_k(x_n)\to 0$ for
  $x_n\to\infty$. Further,
  \[ a_0(D_n) w(x_n) = \frac1{2\pi i}\int_\gamma \frac{M_k(\tau)}{a_+(\tau)}\, a(\tau) e^{ix_n\tau} d\tau = 0,\]
  as the integrand is holomorphic. Finally,
  \[ b_{j0}(D_n)w(x_n)\big|_{x_n=0} = \frac1{2\pi i}\int_\gamma \frac{b_{j0}(\tau)M_k(\tau)}{a_+(\tau)}\; e^{ix_n\tau}
  \big|_{x_n=0} d\tau = \delta_{jk}\quad (j,k=1,\dots,m),\]
  which finishes the proof.
\end{proof}

\begin{remark}
\label{6.7}
a) With the above notation, the following expressions are quasi-homogeneous  in $(\xi',\lambda,\tau)$, more precisely,
positively homogeneous in  $(\xi',\lambda^{1/2m},\tau)$:
\begin{itemize}
\item $a_+(x',\xi',\lambda,\tau)$ of degree $m$,
\item $\tau_j(\xi',\lambda)$ of degree 1,
\item $p_\ell (x',\xi',\lambda)$ of degree $\ell$,
\item $N_k(x',\xi',\lambda,\tau)$ of degree $m-k$,
\item $b_{j0}(x',\xi',\tau)$ of degree $m_j\;(j=1,\dots,m)$,
\item $L_{ij}(x',\xi',\lambda)$ of degree $m_i-j+1$,
\item $M_k(x',\xi',\lambda,\tau)$ of degree $m-m_k-1$,
\item $\gamma(x',\xi',\lambda)$ of degree 1,
\item $\frac{M_k(\tau)}{a_+(\tau)}$ of degree $-m_k-1$.
\end{itemize}

b) In the following, let
\[\langle\xi'\rangle_\lambda := |\xi'|+|\lambda|^{1/2m}.\]
By a), the length of
 $\gamma(x',\xi',\lambda)$ can be estimated by $ C\langle\xi'\rangle_\lambda$. For  $\tau\in\gamma$, one gets
 \begin{align*}
 \Im \tau & \ge C \langle\xi'\rangle_\lambda,\\
 |\tau-\tau_j(x',\xi',\lambda)| & \ge C \langle\xi'\rangle_\lambda,\\
 |e^{i\tau x_n}| & \le \exp(-C\langle\xi'\rangle_\lambda \;x_n).
 \end{align*}
 For $\gamma'\in\N_0^{n-1}$ and $\alpha_n\in\N_0$, we obtain
 \[ \big| D_n^{\alpha_n} D_{\xi'}^{\gamma'} w_k(x',\xi',\lambda,x_n)\big| \le C \langle\xi'
 \rangle_\lambda^{-m_k+\alpha_n-|\gamma'|} e^{-C\langle\xi'\rangle\, x_n}.\]
 In the smooth situation, these estimates show that  $w_k$ is the symbol of a Poisson operator. Such operators belog to
 the pseudodifferential calculus of boundary value problems which is also known as the Boutet de Monvel calculus (see,
 e.g., \cite{Grubb96}).
\end{remark}

To show maximal regularity for parabolic boundary value problems, we again start with the model problem related to
 $(A,B)$ acting in   $\R^n_+:= \{x\in\R^n: x_n>0\}$  with boundary $\partial\R^n_+=\R^{n-1}$.
For this, we fix $x_0'\in\partial G$ and choose the coordinate system corresponding to  $x_0'$. We obtain the boundary
value problem
\begin{equation}
\label{6-2}
\begin{aligned}
(A_0(D)-\lambda) u & = f\quad\text{ in }\R^n_+,\\
B_{j0}(D) u & = 0\quad (j=1,\dots,m)\text{ on } \R^{n-1}.
\end{aligned}
\end{equation}
Here we have set
\begin{align*}
A_0(D) & := \sum_{|\alpha|=2m} a_\alpha(x_0')D^\alpha,\\
B_{j0}(D) & := \sum_{|\beta|=m_j} b_{j\beta}(x_0')\gamma_0 D^\beta.
\end{align*}

In the following result, we construct the solution operators for the model problem.

\begin{theorem}
\label{6.8} Let the boundary value problem $(A,B)$ be parameter-elliptic in the sector $\bar{\Sigma}_\phi$, and let
$x_0'\in\partial\Omega$ be fixed. Then the model problem   \eqref{6-2} has for every  $f\in L^p(\R^n_+)$ and
$\lambda\in\Sigma_\phi\setminus\{0\}$ a unique solution  $u\in W_p^{2m}(\R^n_+)$. This solution is given by
\begin{align*}
u & = R_+ R(\lambda) E_0 f - \sum_{j=1}^m T_j(\lambda) \Lambda_{2m-m_j}(\lambda) \tilde B_{j0}(D) R_+ R(\lambda) E_0 f \\
& - \sum_{j=1}^m \tilde T_j(\lambda)\Lambda_{2m-m_j-1}(\lambda) \partial_n \tilde B_{j0}(D) R_+R(\lambda) E_0 f.
\end{align*}
Here, the operators are defined in the following way:

a) $E_0\colon L^p(\R^n_+)\to L^p(\R^n),\, f\mapsto E_0f$ with
\[ E_0f := \begin{cases}
f, & \text{ for } x_n>0,\\ 0, & \text{ for } x_n\le 0
\end{cases}\]
(trivial extension by $0$).

b) $R(\lambda) := (A_p-\lambda)^{-1}\in L(L^p(\R^n))$, where $A_p$ is the $L^p(\R^n)$-
realization of  $A_0(D)$.

c) $R_+\colon L^p(\R^n)\to L^p(\R^n_+),\, u\mapsto u|_{\R^n_+}$, the restriction to $\R^n_+$.

d) $\tilde B_{j0}(D) := \sum_{|\beta|=m_j} b_{j0}(x_0')D^\beta$, the boundary operators without taking the trace
 $\gamma_0$ on the boundary.

e) $\Lambda_s(\lambda):= (\mathscr F')^{-1} (\lambda+|\xi'|^{2m})^{s/2m} \mathscr F'\in L(W_p^s(\R^n_+), L^p(\R^n_+))$
for $s\in\N_0$, where $\mathscr F'$ denotes the Fourier transform in the tangential variables $x'=(x_1,\dots,x_{n-1})$.

f) $T_j(\lambda)$ is given by
\[ (T_j(\lambda)\phi)(x',x_n) := \int_0^\infty (\mathscr F')^{-1} (\partial_n w_j)(x_0',\xi',\lambda,x_n+y_n)
\mathscr F'(\Lambda_{-2m+m_j}(\lambda)\phi)(\xi',y_n) d y_n\]
for $\phi\in L^p(\R^n_+)$.

g) $\tilde T_j(\lambda)$ is given by
\[ (\tilde T_j(\lambda)\phi)(x',x_n) := \int_0^\infty (\mathscr F')^{-1} w_j(x_0',\xi',\lambda,x_n+y_n)
\mathscr F'(\Lambda_{-2m+m_j+1}(\lambda)\phi)(\xi',y_n) d y_n\]
for $\phi\in L^p(\R^n_+)$.

The functions $w_j(x_0',\xi',\lambda,x_n)$ are defined in Theorem~\ref{6.6}.
\end{theorem}

\begin{proof}
Here we  only show the solution formula for $u$, as the property  $u\in W_p^{2m}(\R^n_+)$ will be included in the
 proof of the $\RR$-boundedness of the solution operators below.

Let $u_1\in W_p^{2m}(\R^n_+)$ be the unique solution of
\[ (A_0(D)-\lambda) u_1 = E_0f \quad\text{ in } \R^n,\]
which exists due to Theorem~\ref{5.4}. So we have
 $u_1 = R(\lambda)E_0f$. For $u$, we choose the ansatz    $u=u_1+u_2$. Then $u$ is a solution of  \eqref{6-2} if and
 only if $u_2$ is a solution of the boundary value problem
\begin{align*}
(A_0(D)-\lambda) u_2 & = 0\quad\text{ in }\R^n_+,\\
B_{j0}(D) u_2 & = g_j\quad (j=1,\dots,m) \text{ on }\R^{n-1}
\end{align*}
with
\[ g_j := -B_{j0}(D)R_+ u_1.\]
Taking partial Fourier transform $\mathscr F'$ with respect to  $x'$, we obtain
\begin{equation}\label{6-3}
\begin{aligned}
(a_0(x_0',\xi',D_n)-\lambda) v(x_n) & = 0 \quad (x_n>0),\\
b_{j0}(x_0',\xi',D_n) v(x_n)\big|_{x_n=0} & = h_j(\xi')\quad (j=1,\dots,m).
\end{aligned}
\end{equation}
Here, $v(x_n) := v(\xi',x_n) := (\mathscr F'u_2(\cdot,x_n))(\xi')$ and $h_j(\xi') := (\mathscr F' g_j)(\xi')$.
By Theorem~\ref{6.6}, the unique solution of \eqref{6-3} is given by
\[ v(\xi',x_n) = \sum_{j=1}^m w_j(x_0',\xi',\lambda,x_n) h_j(\xi').\]
Note that  $g_j$ is first defined only on the boundary $\R^{n-1}$. By
\[ \tilde g_j:= \sum_{|\beta|=m_j} b_{j\beta}(x_0') D^\beta u_1 = \tilde B_{j0}(D) u_1\]
we define an extension  $g_j$ to $\R^n_+$. Then $\tilde h_j := \mathscr F'\tilde g_j(\cdot,x_n)$ is an extension of $h_j$.

For $j=1,\dots,m$ we write (this is sometimes called the ``Volevich trick'')
\begin{align*}
w_j(x_0',&\xi',\lambda,x_n) h_j(\xi') \\
& = -\int_0^\infty \partial_n\big[ w_j(x_0',\xi',\lambda,x_n+y_n)\tilde h_j(\xi',y_n)\big] dy_n \\
& = -\int_0^\infty (\partial_n w_j)(x_0',\xi',\lambda,x_n+y_n)\tilde h_j(\xi',y_n)dy_n\\
& \quad - \int_0^\infty w_j(x_0',\xi',\lambda,x_n+y_n)(\partial_n \tilde h_j)(\xi',y_n) dy_n.
\end{align*}
For $\lambda\in\C_+\setminus\{0\}$ it holds that $\Lambda_{-s}(\lambda)\Lambda_s(\lambda) = \id_{L^p(\R^n)}$
for all $s\in\R$. Therefore, we can write  $\tilde g_j = \Lambda_{-2m+m_j}(\lambda)\Lambda_{2m-m_j}(\lambda)\tilde g_j$
and  $\partial_n \tilde g_j = \Lambda_{-2m+m_j+1}(\lambda)\Lambda_{2m-m_j+1}(\lambda)\partial_n \tilde g_j$,
respectively. This yields
\begin{align*}
u_2(x',x_n) & = \big( (\mathscr F')^{-1} v(\cdot,x_n)\big) (x') \\
& =\sum_{j=1}^m \big( T_j(\lambda)\Lambda_{2m-m_j}(\lambda)\tilde g_j + \tilde T_j(\lambda)\Lambda_{2m-m_j+1}(\lambda)
\partial_n\tilde g_j\big).
\end{align*}
Inserting  $\tilde g_j = \tilde B_{j0}(D)R_+ u_1$ and $u=u_1+u_2$ into this formula, the solution formula of the
 theorem follows. As both the whole space problem as well as \eqref{6-3} is uniquely solvable and as the Fourier
 transform is a bijection in  $\mathscr S'(\R^{n-1})$, we obtain unique solvability with the unique solution $u=u_1+u_2$.
\end{proof}

\begin{lemma}
\label{6.9}
The one-sided Hilbert transform
\[ (Hf)(x) := \int_0^\infty \frac{f(y)}{x+y}\; dy \]
defines a bounded linear operator $H\in L(L^p(\R_+))$.
\end{lemma}

\begin{proof}
For $\epsilon\in(0,1]$, let $m_\epsilon:= \sign(\xi)e^{-\epsilon\xi}\;(\xi\in\R)$. Then $|m_\epsilon(\xi)|\le 1$ and
$|\xi|\cdot |m_\epsilon'(\xi)| = \epsilon|\xi|e^{-\epsilon|\xi|}\le 1$, where we used the inequality $te^{-t}<1\;(t>0)$.
By Mikhlin's theorem,  $\|\mathscr F_1^{-1}m_\epsilon \mathscr F_1\|_{L(L^p(\R_+))}\le C$ with a constant $C>0$
independent of  $\epsilon$. Here $\mathscr F_1$ stands for the one-dimensional Fourier transform.

For $f\in\mathscr S(\R)$ we get
\begin{align*}
(\mathscr F_1^{-1} & m_\epsilon \mathscr Ff)(x) = \frac1{\sqrt{2\pi}} \int_{-\infty}^\infty e^{ix\xi} \sign(\xi)
e^{-\epsilon|\xi|}(\mathscr F_1)(\xi)d\xi\\
& = \frac1{\sqrt{2\pi}}\int_0^\infty\Big[ e^{ix\xi-\epsilon\xi}\mathscr F_1f(\xi) - e^{-ix\xi-\epsilon\xi}
\mathscr F_1f(-\xi)\Big] d\xi\\
& = \frac1{2\pi}\int_0^\infty \int_{-\infty}^\infty \big( e^{ix\xi-\epsilon\xi-iy\xi} - e^{-ix\xi-\epsilon\xi+iy\xi}
\big) f(y)dy\,d\xi\\ & = \frac1{2\pi}\int_{-\infty}^\infty \Big( \frac{e^{i(x-y)\xi-\epsilon\xi}}{i(x-y) -\epsilon}
\Big|_{\xi=0}^\infty - \frac{e^{-i(x-y)\xi-\epsilon\xi}}{-i(x-y)-\epsilon}\Big|_{\xi=0}^\infty\Big) f(y)dy\\
& = \frac1{2\pi}\int_{-\infty}^\infty \Big(-\frac1{i(x-y)-\epsilon} + \frac1{-i(x-y)-\epsilon}\Big) f(y)dy\\
& = \frac{i}{\pi} \int_{-\infty}^\infty \frac{x-y}{(x-y)^2+\epsilon^2}\; f(y)\,dy.
\end{align*}
Define for $\epsilon\in(0,1]$
\[ (H_\epsilon f)(x) := \int_0^\infty \frac{x+y}{(x+y)^2+\epsilon^2} \; f(y)\,dy\quad (f\in L^p(\R_+)).\]
Then  $H_\epsilon f(x) = (-\frac{\pi}{i}) (\mathscr F_1^{-1}m_\epsilon \mathscr F_1E_0f)(-x)$ for $x\ge 0$, where
$E_0\colon L^p(\R_+)\to L^p(\R)$ again stands for the trivial extension. We obtain
\[
\| H_\epsilon f\|_{L^p(\R_+)}  \le \pi \| \mathscr F_1^{-1} m_\epsilon \mathscr FE_0f\|_{L^p(\R)} \le C
\|E_0f\|_{L^p(\R)} \le C \|f\|_{L^p(\R_+)}.\]
The sequence $H_{1/n}(|f|)$ is monotonously increasing and converges pointwise to  $H(|f|)$. By monotone convergence,
we see that
\begin{align*}
 \|Hf\|_{L^p(\R_+)} & \le \| H(|f|)\|_{L^p(\R_+)} = \lim_{n\to\infty} \|H_{1/n}(|f|)\|_{L^p(\R_+)} \\
 &\le C \|\,|f|\,\|_{L^p(\R_+)} = C \|f\|_{L^p(\R_+)}.
 \end{align*}
Therefore, $H\in L(L^p(\R_+))$.
\end{proof}

The following result shows that the solution operators are indeed $\RR$-bounded.

\begin{theorem}
\label{6.10} Let $\delta>0$ be fixed. In the situation of Theorem~\ref{6.8}, the following operator families  in
$L(L^p(\R^n_+))$ are
 $\RR$-bounded:

a) $\{\Lambda_{2m-m_j}(\lambda)\tilde B_{j0}(D)R_+R(\lambda)E_0: j=1,\dots,m, \, \lambda\in\bar\C_+,\,|\lambda|\ge \delta\}$,

b) $\{\Lambda_{2m-m_j-1}(\lambda)\partial_n \tilde B_{j0}(D) R_+ R(\lambda) E_0: j=1,\dots,m, \, \lambda\in\bar\C_+,\,|
\lambda|\ge \delta\}$,

c) $\{\lambda T_j(\lambda): j=1,\dots,m, \, \lambda\in\bar\C_+,\,|\lambda|\ge \delta\}$,

d) $\{\lambda\tilde T_j(\lambda): j=1,\dots,m, \, \lambda\in\bar\C_+,\,|\lambda|\ge \delta\}$.
\end{theorem}

\begin{proof}
a) We have $\Lambda_{2m-m_j}(\lambda)\tilde B_{j0}(D)R_+ = R_+\Lambda_{2m-m_j}(\lambda) \tilde B_{j0}(D)$.
As the operators $R_+\in L(L^p(\R^n), L^p(\R^n_+))$ and $E_0\in L(L^p(\R^n_+),L^p(\R^n))$ are bounded,
it suffices to show the $\RR$-boundedness of
\[ \{\Lambda_{2m-m_j}(\lambda) \tilde B_{j0}(\lambda) R(\lambda): j=1,\dots,m, \, \lambda\in\bar\C_+,\,
|\lambda|\ge \delta\}.\]
The corresponding family of symbols (with respect to the Fourier transform in $\R^n$) is given by
\[ m(\xi,\lambda) := (\lambda+|\xi'|^{2m})^{\frac{2m-m_j}{2m}} b_{j0}(x_0',\xi) \big(a_0(x_0',\xi)-\lambda\big)^{-1}.\]
As $m(\xi,\lambda)$ is quasi-homogeneous of degree 0 in $(\xi,\lambda)$ and bounded on $|\lambda|+|\xi|^{2m}=1$,
it follows that
\[ |D^\alpha m(\xi,\lambda)| \le C |\xi|^{-|\alpha|}\quad (\xi\in\R^n\setminus \{0\},\, \lambda\in\bar\C_+,\,
 |\lambda|\ge\delta).\]
By Corollary~\ref{3.31}, the operator family in a) is  $\RR$-bounded.

b) can be shown analogously.

c) For $\phi\in L^p(\R^n_+)$, we write
\[ \lambda T_j(\lambda)\phi = \int_0^\infty k_\lambda(x_n,y_n)\psi(y_n) dy_n\]
with $\psi\in L^p(\R_+;L^p(\R^{n-1})$, $\psi(y_n) := \phi(\cdot,y_n)$, and the operator valued integral kernel
\begin{align*}
k_\lambda(x_n,y_n) & := (\mathscr F')^{-1} \tilde m(\xi',\lambda, x_n+y_n) \mathscr F'\\
& := (\mathscr F')^{-1} \lambda \partial_n w_j(x_0',\xi',\lambda,x_n+y_n)(\lambda+|\xi'|^{2m})^{-\frac{2m-m_j}{2m}}
\mathscr F'.
\end{align*}
By Remark~\ref{6.7} b), the inequalities
\begin{align*}
| D_{\xi'}^{\gamma'} \tilde  m(\xi',\lambda,x_n+y_n) & \le C (|\xi'|+|\lambda|^{1/2m} \exp\big(-C(|\xi'|+|\lambda|^{1/2m})
(x_n+y_n)\big) |\xi'|^{-|\gamma'|}\\
& \le \frac C{x_n+y_n}\; |\xi'|^{-|\gamma'|}
\end{align*}
hold, where in the last step we again used the elementary estimate  $te^{-t}<1\;(t>0)$. Again by Corollary~\ref{3.31},
it follows that  $k_\lambda(x_n,y_n)\in L(L^p(\R^{n-1}))$ with
\[ \RR\big\{ k_\lambda(x_n,y_n): \lambda\in\bar\C_+,\, |\lambda|\ge \delta\big\} \le \frac C{x_n+y_n}\,.\]
The scalar integral operator with kernel $k_0(x_n,y_n):= \frac 1{x_n+y_n}$, given by
\[ (K_0g)(x_n) := \int_0^\infty \frac{g(y_n)}{x_n+y_n}\;dy_n\quad(g\in L^p(\R_+))\]
is the one-sided Hilbert transform in  $L^p(\R_+)$ and, due to Lemma~\ref{6.9}, a bounded linear operator $K_0\in
L(L^p(\R_+))$. By Theorem~\ref{3.18} we get
\[ \RR\big\{ \lambda T_j(\lambda): \lambda\in\bar\C_+,\,|\lambda|\ge \delta\big\} \le C\|K_0\|_{L(L^p(\R_+))} <\infty.\]

d) follows in the same way as c).
\end{proof}

Now maximal regularity for the model problem is an immediate consequence of the previous results.

\begin{theorem}
\label{6.11}
Let the boundary value problem $(\partial_t - A,B)$ be parabolic, and let  $x_0'\in\partial G$. Choose the coordinate system
 corresponding to  $x_0'$, and consider the   $L^p$-realization $A_B^{(0)}$ of the model problem $(A_0(x_0',D), B(x_0',D))$.
  Then  $\rho(A_B^{(0)})\supset \bar\C_+\setminus\{0\}$, and for every $\delta>0$ the operator family
\[ \big\{ \lambda(\lambda-A_B^{(0)})^{-1}: \lambda\in\bar\C_+,\, |\lambda|\ge \delta\big\}\subset L(L^p(\R^n_+))\]
is $\RR$-bounded. In particular, $A_B^{(0)}-\delta$ has for every  $\delta>0$ maximal $L^q$-regularity for all
 $1<q<\infty$ (and generates a bounded holomorphic $C_0$-semigroup).
\end{theorem}

\begin{proof}
Replacing in the proof of Theorem~\ref{6.10} the operators $\lambda T_j(\lambda)$ by $D^\alpha T_j(\lambda)$
(and analogously for $\tilde T_j(\lambda)$) with $|\alpha|=2m$, we see that the solution operators in fact define
 a solution  $u\in W_p^{2m}(\R^n_+)$. Therefore, the solution coincides with the resolvent.
Now the  $\RR$-boundedness follows directly from the resolvent description in Theorem~\ref{6.8} and the statements
 on  $\RR$-boundedness from Theorem~\ref{6.10}.
\end{proof}

To deal with variable coefficients, we first study small perturbations in the principal part.

\begin{theorem}\label{6.12}
   Let  $A^0(x,D) = \sum_{|\alpha|=2m} a^0_\alpha D^\alpha$ and $B^0_j(x,D) = \sum_{|\beta|=m_j}
  b^0_{j\beta}\gamma_0 D^\beta$ with $a^0_\alpha\in \C$ and  $b^0_{j\beta}\in \C$. Assume the boundary value
   problem $(\partial_t - A,B)$ to be parabolic in the domain  $\R^n_+$. Then there exists an $\epsilon>0$
   such that the following statement holds: Let  $A(x,D) = A^0(x,D)+\tilde A(x,D)$ and  $B(x,D)=B^0(x,D)+\tilde B(x,D)$ with
  \begin{align*}
   \tilde A(x,D) & = \sum_{|\alpha|=2m} \tilde a_\alpha(x)D^\alpha,\\
   \tilde B_j(x,D) & = \sum_{|\beta|=m_j} \tilde b_{j\beta}(x) D^\beta\;(j=1,\dots,m).
  \end{align*}
  Here, $\tilde a_\alpha\in L^\infty(\R^n_+)$ and $\tilde b_{j\beta}\in \buc^{2m-m_j}(\R^{n-1})$.
Assume further that
  \begin{align*}
    \sum_{|\alpha|=2m} \|\tilde a_\alpha\|_{L^\infty(\R^n_+)} & \le \epsilon,\\
    \sum_{|\beta|=m_j} \|\tilde b_{j\beta}\|_{L^\infty(\R^{n-1})} & \le\epsilon\; (j=1,\dots,m).
  \end{align*}
  Let $A_{B,p}$ be the  $L^p$-realization of the boundary value problem $(A(x,D), B(x,D))$. Then there exists a
   $\mu>0$ such that the operator family
  \[ \big\{ \lambda (A_{B,p}-\lambda)^{-1}: \lambda\in\bar\C_+,\;|\lambda|\ge \mu\big\}\subset L(L^p(\R^n_+))\]
  is $\RR$-bounded. Here,  $\epsilon$ and the  $\mathcal R$-bound only depend on $(A^0(x,D),B^0(x,D))$, and $\mu$
   additionally depends on the norms  $\|b_{j\beta}\|_{\buc^{2m-m_j}}(\R^{n-1})$ for  $|\beta|=m_j$, $j=1,\dots,m$.
\end{theorem}

\begin{proof}
We indicate the main steps of the proof, for a more elaborated version, see  \cite{Denk-Hieber-Pruess03},
Subsection~7.3.

Without loss of generality, we may assume that the coefficients of $\tilde B(x,D)$ are defined on all of $\R^n_+$.
 We
write the boundary value problem
  \begin{align*}
    (A(x,D)-\lambda) u & = f\quad \text{ in }\R^n_+,\\
    B_j(x,D) u & = 0 \quad (j=1,\dots,m)\; \text{ on } \R^{n-1}
  \end{align*}
 in the form
  \begin{align*}
    (A^0(x,D)-\lambda) u & = f-\tilde A(x,D) u \quad \text{ in } \R^n_+,\\
    B^0_j(x,D) u & = - \tilde B_j(x,D) u \quad (j=1,\dots,m)\;\text{ on } \R^{n-1}.
  \end{align*}
  Let $(A_{B,p}^{0}-\lambda)^{-1}$ be the resolvent of the  $L^p$-realization of $(A^0(x,D), B^0(x,D))$,
  which exists due to Theorem~\ref{6.11}. Applying the solution operators from Theorem~\ref{6.8}, we obtain
  \begin{align*}
    u & = (A_{B,p}^{0}-\lambda)^{-1} f - (A_{B,p}^{0}-\lambda)^{-1}\tilde A(x,D) u \\
    & - \sum_{j=1}^m T_j(\lambda) \Lambda_{2m-m_j}(\lambda) \tilde B_j(x,D) u  - \sum_{j=1}^m \tilde T_j(\lambda)
    \Lambda_{2m-m_j-1}(\lambda)
    \partial_n \tilde B_j(x,D) u \\
    & =: (A_{B,p}^{0}-\lambda)^{-1} f - S(\lambda) u.
  \end{align*}
We estimate the norm of $S(\lambda)u$. For the term  $(A_{B,p}^{0}-\lambda)^{-1}\tilde A(x,D) u$, we use
\[ \|(A_{B,p}^{0}-\lambda)^{-1}\|_{L(L^p(\R^n_+), W_p^{2m}(\R^n_+))} \le C_1\] and obtain
  \[
    \| (A_{B,p}^{0}-\lambda)^{-1}\tilde A(x,D)u\|_{W_p^{2m}(\R^n_+)} \le C_1 \|\tilde A(x,D)u\|_{L^p(\R^n_+)}
     \le C_1 \epsilon \|u\|_{W_p^{2m}(\R^n_+)}.
  \]
For the other terms, we use the fact that the operator families
\[ \{ \lambda^{(2m-|\alpha|)/2m} D^\alpha T(\lambda): |\alpha|\le 2m,\, \lambda\in\bar\C_+,\,
|\lambda|\ge \lambda_0\}\subset L(L^p(\R^n_+))\]
are $\mathcal R$-bounded and therefore bounded, which can be seen as in the proof of Theorem~\ref{6.10}. This yields
\begin{align*}
  \|T_j(\lambda) \Lambda_{2m-m_j}(\lambda) \tilde B_j(x,D) u \|_{W_p^{2m}(\R^n_+)} & \le C \| \Lambda_{2m-m_j}(\lambda)
  \tilde B_j(x,D) u \|_{L^p(\R^n_+)} \\
  & \le C \| \tilde B_j(x,D) u\|_{W_p^{2m-m_j}(\R^n_+)} .
  \end{align*}
 The terms of the form $\tilde b_{j\beta} D^\beta u$ can be estimated, using the Leibniz rule, by
  \begin{align*}
  \| \tilde b_{j\beta} D^\beta u \|_{W_p^{2m-m_j}(\R^n_+)} & \le C  \sum_{|\gamma|\le 2m-m_j} \sum_{\delta +\delta'
  =\gamma}\| (D^\delta\, \tilde b_{j\beta})( D^{\delta'+\beta} u)\|_{L^p(\R^n_+)}\\
  & \le C_2 \epsilon \|u\|_{W_p^{2m}(\R^n_+)} + C_3 \|u\|_{W_p^{2m-1}(\R^n_+)}.
\end{align*}
Here, the constant $C_3$ depends on the norm $\|b_{j\beta}\|_{\buc^{2m-m_j}(\R^{n-1})}$. With the interpolation
inequality, we see that for some constants $C_1,C_2$ we have
\[ \| S(\lambda)u \|_{W_p^{2m}(\R^n_+)} + |\lambda| \|S(\lambda)u\|_{L^p(\R^n_+)}
\le C_1 \epsilon \|u\|_{W_p^{2m}(\R^n_+)} + C_2 \|u\|_{L^p(\R^n_+)}.\]
Now we endow  $W_p^{2m}(\R^n_+)$ with the parameter-dependent norm  $\norm u\norm := \|u\|_{W_p^{2m}(\R^n_+)} +
|\lambda| \|u\|_{L^p(\R^n_+)}$. Note that for every fixed  $\lambda$, this norm is equivalent to the standard norm.
For         $|\lambda|\ge  2C_2$ and $C_1\epsilon \le \frac 12$, it follows that
\[ \norm S(\lambda) u\norm \le \tfrac 12 \norm u\norm.\]
Therefore,  $(1+S(\lambda))\in L(W_p^{2m}(\R^n_+))$ is invertible (with respect to the new norm, and therefore also
with respect to the standard norm). Thus, we have seen that the above boundary value problem is uniquely solvable and
that the resolvent   $(A_{B,p}-\lambda)^{-1}$ exists for all $\lambda\in\bar\C_+$ with $|\lambda|\ge 2C_2$.

To obtain an estimate on the $\mathcal R$-bounds, we can argue similarly. Starting from the identity
\[ (A_{B,p}-\lambda)^{-1} = (A_{B,p}^0-\lambda)^{-1} - S(\lambda) (A_{B,p}-\lambda)^{-1}, \]
one can show for sufficiently large $\mu>0$
  \begin{align*}
   \RR\big\{ \tilde A(x,D) & (A_{B,p}-\lambda)^{-1}:\lambda\in\bar\C_+,\,|\lambda|\ge\mu\big\} \\
   & \le \sum_{|\alpha|=2m} \|\tilde a_\alpha\|_{L^\infty(\R^n_+)} \RR \big\{ D^\alpha (A_{B,p}-\lambda)^{-1}:
   \lambda\in\bar\C_+,\,|\lambda|\ge\mu\big\}\\
   & \le C\epsilon  \RR \big\{ D^\alpha (A_{B,p}-\lambda)^{-1}:\lambda\in\bar\C_+,\,|\lambda|\ge\mu\big\}.
  \end{align*}
  Similarly, the other terms in  $S(\lambda)(A_{B,p}-\lambda)^{-1} $ can be estimated. Consider the operator
  family
  \[ \mathcal T := \big\{ \lambda^{2m-|\alpha|)/(2m)}D^\alpha (A_{B,p}-\lambda)^{-1}: |\alpha|\le 2m,\, \lambda\in
  \bar\C_+,\,|\lambda|\ge\mu\big\}.\]
  The above calculations show that for every finite subset $\mathcal T_0$ of $\mathcal T$, we get
  the inequality
  \[ \RR(\mathcal T_0)\le R_1 + (C_1\epsilon+ C_2(\mu)) \RR(\mathcal T_0).\]
  Here,
  \[ C_1 := \RR \big\{ \lambda^{2m-|\alpha|)/(2m)}D^\alpha (A_{B,p}^{0}-\lambda)^{-1}:|\alpha|\le 2m,\,\lambda\in
  \bar\C_+,\,|\lambda|\ge\mu\big\}<\infty\]
  and $C_2(\mu)\to 0$ for $\mu\to\infty$. Choosing $\epsilon$ small enough  and $\mu$ large enough, we have
  $C_1\epsilon+C_2(\mu)<\frac 12$, and therefore
  $\RR(\mathcal T_0)< 2 R_1<\infty$. As this holds for every finite subset $\mathcal T_0$ of $\mathcal T$, with $R_1$
  being independent of $\mathcal T_0$, we get the same estimate for $\mathcal T$, i.e.,  $\RR(\mathcal T)\le 2R_1$.
\end{proof}

The last result deals with small perturbations of the top-order coefficients. As before, lower-order terms of the
operators can be handled by the interpolation inequality. For a proof of maximal regularity in the situation of a
bounded domain and under the above smoothness assumptions, the method of localization can be used.
We mention some main ideas in the following remark.

\begin{remark}[Localization]
\label{6.13} Let $(\partial_t - A,B)$ be a parabolic boundary value problem in the bounded domain $G$, and assume
 the smoothness assumptions from the beginning of this section to hold. To prove $\RR$-sectoriality of the
 $L^p$-realization of $(A,B)$, one can use the following steps:

a) For every fixed $x_0\in\partial G$, by definition of a  $C^{2m}$-domain, there exists a neighbourhood
 $U(x_0)\subset\R^n$ and a $C^{2m}$-diffeomorphism $\Phi_{x_0}\colon U(x_0)\to V(x_0) := \Phi_{x_0}(U(x_0))\subset\R^n$ with
\[ \Phi_{x_0}(U(x_0)\cap G) = V(x_0)\cap \R^n_+.\]
We denote by $(\tilde A,\tilde B)$ the transformed boundary value problem in the domain $V(x_0)$. The coefficients
$\tilde a_\alpha$ of $\tilde A$ are defined in $V(x_0)\cap\bar{\R^n_+}$ and satisfy the same
 smoothness assumptions as  $a_\alpha$. In the same way, this holds for the transformed coefficients $\tilde b_{j\beta}$
 of $\tilde B_j$. Moreover, it is possible to show that the transformed problem is parabolic in  $V(x_0)\cap \R^n_+$.

The coefficients  $\tilde a_\alpha$ and $\tilde b_{j\beta}$ can be extended to the half space
 $\bar\R^n_+$ and  $\R^{n-1}$, respectively, in such a way that both the smoothness and the parabolicity is preserved.
  For  $\tilde a_\alpha$, we can choose an appropriate continuous extension. For the coefficients on the boundary
   $\tilde b_{j\beta}$, we have to preserve higher smoothness. For this, one can, e.g., define
\[ \tilde b_{j\beta} (y) := \tilde b_{j\beta}\Big( y_0 + \chi\Big(\frac{y-y_0}{r}\Big)(y-y_0)\Big)\quad (y\in\R^{n-1}),\]
where $\chi\in C^\infty(\R^{n-1})$ satisfies  $\chi(x)=1$ for $|x|\le 1$ and $\chi(x) = 0$ for $|x|\ge 2$.
Here, $y_0 := \Phi_{x_0}(x_0)$, and $r>0$ is chosen sufficiently small.

For an eventually even smaller  $r=r(x_0)$, the following inequalities hold true for a given $\epsilon>0$:
\begin{align*}
\sum_{|\alpha|=2m} \|\tilde a_\alpha(\,\cdot\,) -\tilde a_\alpha(y_0)\|_{L^\infty(\R^n_+)} & <\epsilon,\\
\sum_{|\beta|=m_j} \|\tilde b_{j\beta}(\,\cdot\,)-\tilde b_{j\beta}(y_0)\|_{L^\infty(\R^{n-1})} & <\epsilon \quad (j=1,\dots,m).
\end{align*}
Therefore, the localized boundary value problems satisfy the conditions of Theorem~\ref{6.12}.

For fixed $\epsilon>0$, this construction   yields an open cover  of the form
\[ \partial G \subset \bigcup_{x_0\in\partial G} \Phi_{x_0}^{-1}(B(y_0,r(x_0))).\]
By compactness of  $\partial G$, there exists a finite subcover $\partial G\subset \bigcup_{k=1}^N U_k$, where we have
set $U_k := \Phi_{x_k}^{-1}(B(y_k,r(x_k)))$.

b) In the same way, in the interior of the domain, we obtain for every $x_0\in G$ a small neighbourhood
$U(x_0)\subset\R^n$ and an extension $\tilde a_\alpha$ of $ a_\alpha|_{U(x_0)}$ such that
\[ \sum_{|\alpha|=2m} \|\tilde a_\alpha(\,\cdot\,) - \tilde
a_\alpha(x_0)\|_{L^\infty(\R^n)} <\epsilon\] holds.  In this way, we obtain an open cover
\[G\setminus\bigcup_{k=1}^N U_k \subset \bigcup_{x_0\in G} B(x_0,r(x_0)).\]
Note that no boundary operator and no diffeomorphism is involved.  As $G\setminus\bigcup_{k=1}^N U_k$ is compact, there
exists a finite subcover
\[ G\setminus\bigcup_{k=1}^N U_k \subset \bigcup_{k=N+1}^M U_k \] with  $U_k= B(x_k,r(x_k)))$.
 Altogether, this yields a finite open cover
 $\bar G \subset \bigcup_{k=1}^{M} U_k$.

c) With this construction, one obtains finitely many operators $(\tilde A^{(k)}, \tilde B^{(k)})$ for $k=1,\dots,N$
and $\tilde A^{(k)}$ for $k=N+1,\dots,M$, which satisfy the assumptions of Theorem~\ref{6.12} and Lemma~\ref{5.7},
respectively. Now we can use the resolvents of the $L^p$-realization of these operators to show $\mathcal R$-sectoriality
 of $A_{B,p}-\mu$ for large $\mu$. This can be done  similarly as in the proof of Theorem~\ref{5.9}, using a partition
 of unity and estimating the commutators with help of the interpolation inequality.
\end{remark}

With the above techniques, it is possible to show the following main theorem on parabolic boundary value problems:

\begin{theorem}
  \label{6.14}
  Assume the boundary value problem $(\partial_t - A,B)$ to be parabolic and to satisfy the smoothness assumptions
  above. Let  $1<p<\infty$. Then there exist $\theta>\frac\pi 2$ and $\mu>0$ such that
  $\rho(A_{B,p}-\mu)\supset \bar\C_+$ and the operator $A_{B,p}-\mu$ is $\RR$-sectorial with angle $\theta$.
  In particular, $A_{B,p}-\mu$ has maximal $L^q$-regularity for all $q\in (1,\infty)$.
\end{theorem}

\section{Quasilinear parabolic evolution equations}

We have seen in the previous sections that, under appropriate parabolicity and smoothness assumptions, the
$L^p$-realization of linear boundary value problems have maximal regularity. This is the basis for the analysis of
nonlinear problems, which will be described in the present section.

\subsection{Well-posedness for quasilinear parabolic evolution equations}

We consider nonlinear evolution equations which can be written in the abstract form
\begin{equation}
  \label{4-1}
  \begin{aligned}
  \partial_t u(t) - A(t,u(t))u(t) & = F(t,u(t))\quad    \text{ in } (0,T_0),\\
  u(0) & = u_0.
  \end{aligned}
\end{equation}
Here, $T_0\in(0,\infty)$. We fix the following situation: Let $p\in (1,\infty)$, and let $X_1\subset X_0$ be Banach
spaces with
 $X_1$ being dense in $X_0$. With $T\in (0,T_0]$, the   spaces  for the right-hand side and the solution are
\[ \F := \F_T := L^p( (0,T);X_0)\quad\text{and}\quad  \E := \E_T := H_p^1((0,T);X_0)\cap L^p((0,T);X_1),\]
respectively. The time trace space, and therefore the space for the initial value $u_0$, is  given by $\gamma_t \E =
(X_0,X_1)_{1-1/p,p}$ (cf. Lemma~\ref{2.4a}). We again set $\leftidx{_{0}}{\E}{} := \{ u\in\E: \gamma_t u =0 \}$. Here
and in the following, we consider the operator $A$ as a map $A\colon (0,T_0)\times \gamma_t\E \to L(X_1,X_0)$. For each
$t\in (0,T_0)$ and $v\in\gamma_t\E$, the operator $A(t,v)\in L(X_1,X_0)$ is identified with the unbounded operator
$A(t,v)$ acting in $X_0$ with domain $X_1$, and $A(t,v)\in\mreg (X_0)$ has to be understood in this sense.

\begin{example}
  \label{4.1a}
We recall the example of the graphical mean curvature flow (Example~\ref{2.1}), which has the form
  \begin{equation}
    \label{4-2}
    \begin{aligned}
    \partial_t u - \Big(\Delta u -\sum_{i,j=1}^n\frac{\partial_i u \partial_j u}{1+|\nabla u|^2}\,\partial_i \partial_j
     u\Big)  & = 0 \quad \text{in } (0,T_0),\\
    u(0) & = u_0.
    \end{aligned}
  \end{equation}
This quasilinear equation can be written in the form \eqref{4-1}, where \[ A(t,u(t)) =
   \Delta -\sum_{i,j=1}^n\frac{\partial_i u(t) \partial_j u(t)}{1+|\nabla u(t)|^2}\,\partial_i \partial_j \]
   and $F=0$. Here we have $X_0=L^p(\R^n)$, $X_1=W_p^2(\R^n)$, and $\gamma_t\E = B_{pp}^{2-2/p}(\R^n)=W_p^{2-2/p}(\R^n)$.
\end{example}

For the nonlinearities $A$ and $F$ in \eqref{4-1}, we assume:
\begin{enumerate}
  [(\mbox{A}1)]
  \item We have $A\in C([0,T_0]\times \gamma_t\E, L(X_1,X_0))$, and for all $R>0$ there exists a  Lipschitz constant
      $L(R)>0$ with
      \[ \|A(t,w)v - A(t,\bar w)v\|_{X_0} \le L(R) \|w-\bar w\|_{\gamma_t\E} \|v\|_{X_1}\]
      for all $t\in [0,T_0]$, $v\in X_1$ and all $w,\bar w\in \gamma_t\E$ with $\|w\|_{\gamma_t\E}\le R$ and $\|\bar
      w\|_{\gamma_t\E}\le R$.
  \item For the mapping $F\colon[0,T_0]\times \gamma_t\E\to X_0$ we assume:
  \begin{enumerate}[(i)]
  \item $F(\cdot, w)$ is measurable for every $w\in\gamma_t\E$,
  \item $F(t,\cdot)\in C(\gamma_t\E,X_0)$ for almost all  $t\in [0,T_0]$,
  \item
  $f(\cdot):= F(\cdot,0)\in L^p((0,T_0);X_0)$,
  \item for every $R>0$, there exists a $\phi_R\in L^p((0,T_0))$ with
  \[ \|F(t,w)-F(t,\bar w)\|_{X_0} \le \phi_R(t) \|w-\bar w\|_{\gamma_t\E}\]
  for almost all $t\in [0,T_0]$ and all $w,\bar w\in\gamma_t\E$ with $\|w\|_{\gamma_t\E}\le R$, $\|\bar
  w\|_{\gamma_t\E}\le R$.
  \end{enumerate}
\end{enumerate}
Apart from standard conditions on measurability and continuity, the above conditions essentially mean that the
functions $A(t,\cdot)v$ and $F(t,\cdot)$ are locally Lipschitz, i.e., they are Lipschitz on bounded subsets of
 $\gamma_t\E$. The following result is based on \cite{Pruess02}, Section~3 (see also \cite{Clement-Li93}).

\begin{theorem}
  \label{4.8}
  Assume (A1) and (A2) as well as $A_0 := A(0,u_0)\in \mreg(X_0)$. Then there exists a $T\in (0,T_0]$ such that
  \eqref{4-1} has a unique solution  $u\in \E_T$ in the interval  $(0,T)$.
\end{theorem}

\begin{proof}
\textbf{(i)}
We use the maximal regularity of  $A_0 := A(0,u_0)$ in the time interval $(0,T)$ with $T\le T_0$ to obtain
estimates for the solutions of the linearized equation. For this, we first consider the equation with initial value 0,
  \begin{equation}
    \label{4-4}
    \begin{aligned}
      \partial_t w(t) - A_0 w(t) & = g(t)\quad (t\in (0,T)),\\
      w(0) & = 0.
    \end{aligned}
  \end{equation}

As $A_0\in \mreg(X_0)$, for every $g\in \F$ there exists a unique solution $w\in\E$, and we obtain the estimate
  \[ \|w\|_{\E} \le C_0 \|g\|_{\F}\]
with a constant $C_0>0$ which does not depend on $T$ or $w$
  (Lemma~\ref{4.6}). By Lemma~\ref{4.3} b), there exists a constant $C_1$ (again independent of $T>0$ and $w$) with
  \[ \|w\|_{C([0,T],\gamma_t\E)} \le C_1\|w\|_{\E}. \]
Note here that  $w(0)=0$ holds.

In the following, we consider the reference solution   $u^*\in\E$ which is defined as the unique solution of
  \begin{equation}
    \label{4-5}
    \begin{aligned}
      \partial_t w(t) - A_0 w(t) & = f(t)\quad (t\in (0,T)),\\
      w(0) & = u_0.
    \end{aligned}
  \end{equation}
Here,  $f:= F(\cdot,0)\in\F$ due to condition (A2) (iii).

  \textbf{(ii)} For $r\in (0,1]$ set
  \[ B_r := \{ v\in\E: v-u^*\in \leftidx{_0}{\E}{}, \, \|v-u^*\|_{\E}\le r\}. \]
  For each $v\in B_r$, define $\Phi(v):= u$ as the unique solution of
  \begin{equation}
    \label{4-6}
    \begin{aligned}
      \partial_t u(t) - A_0 u(t) & = F(t,v(t)) - \big( A(0,u_0)- A(t,v(t))\big) v(t)\quad (t\in (0,T)),\\
      u(0) & = u_0.
    \end{aligned}
  \end{equation}
  We will show that  $\Phi(B_r)\subset B_r$ holds and that  $\Phi$ is a contraction in $B_r$, given that both  $T$ and
   $r$ are sufficiently small.

  \textbf{(iii)} In this step, we show that $\Phi(B_r)\subset B_r$ holds for sufficiently small $T$ and $r$. For this,
   we write
  \begin{equation}
    \label{4-7}
    \|\Phi(v)-u^*\|_{\E}  = \|u-u^*\|_{\E} \le C_0\Big(  \|F(\cdot,v)-f(\cdot)\|_{\F}
  + \|(A(0,u_0)-A(\cdot,v))v\|_{\F}\Big).
  \end{equation}
  Let $m_T := \sup_{t\in[0,T]}\|A(0,u_0)-A(t,u_0)\|_{L(X_1,X_0)}$. By condition (A1) with fixed $R:= C_1 +
  \|u^*\|_{L^\infty((0,T);\gamma_t\E)}$, we obtain
  \begin{align*}
    \| A(0,u_0)v & - A(\cdot,v)v\|_{\F}  = \|A(0,u_0)v-A(\cdot,v)v\|_{L^p((0,T);X_0)}\\
    & \le \|A(0,u_0)-A(\cdot,v)\|_{L^\infty((0,T); L(X_1,X_0))} \|v\|_{L^p((0,T);X_1)}\\
    & \le \Big( \|A(0,u_0)-A(\cdot,u_0)\|_{L^\infty((0,T); L(X_1,X_0))} \\
    & \qquad + \| A(\cdot,u_0)-A(\cdot,v(\cdot))\|_{L^\infty((0,T);L(X_1,X_0))}\Big) \|v\|_{\E} \\ & \le \Big(
    m_T + L(R) \|v-u_0\|_{L^\infty((0,T);\gamma_t\E)} \Big) \|v\|_{\E}\\
    & \le \Big( m_T + L(R) C_1 \|v-u_0\|_{\E}\Big) \|v\|_{\E}.
  \end{align*}
  For   $r\le 1$, we can estimate
  \[
    \|v-u_0\|_{\E}  \le \|v-u^*\|_{\E} + \|u^*-u_0\|_{\E} \le  r + \|u^*-u_0\|_{\E}
  \]
  and
  \[ \|v\|_{\E} \le \|v-u^*\|_{\E} + \|u^*\|_{\E} \le r+ \|u^*\|_{\E}.\]
  Therefore, we obtain
  \[ \|A(0,u_0)v-A(\cdot,v)v\|_{\F} \le \Big( m_T + L(R)C_1(r+\|u^*-u_0\|_{\E} )\Big) (r+\|u^*\|_{\E}).\]
  In a similar way, using (A2), we see that
  \begin{align*}
    \|F(\cdot,v)-f\|_{\F} & \le \| F(\cdot,v)-F(\cdot,u^*)\|_{\F} + \| F(\cdot, u^*)-F(\cdot,0)\|_{\F} \\
    & \le \|\phi_R\|_{L^p((0,T))} \Big( \|v-u^*\|_{L^\infty((0,T); \gamma_t\E)} + \|u^*\|_{L^\infty((0,T); \gamma_t\E)}
    \Big)\\
    & \le \|\phi_R\|_{L^p((0,T))} \Big(C_1 \|v-u^*\|_{\E} + \|u^*\|_{L^\infty((0,T); \gamma_t\E)}\Big)\\
    & \le \|\phi_R\|_{L^p((0,T))} C_1 \big( r + \|u^*\|_{L^\infty((0,T); \gamma_t\E)}\big).
  \end{align*}
  Inserting this into \eqref{4-7}, we get
  \begin{equation}\label{4-8}
  \begin{aligned}
   \|\Phi(v)-u^*\|_{\E} &  \le C_0 \Big[ \|\phi_R\|_{L^p((0,T))} \big(  C_1 r+\|u^*\|_{L^\infty((0,T); \gamma_t\E)}
   \big) \\
   & \quad +
   \big(m_T + L(R)C_1(r+\|u^*-u_0\|_{\E})\big)(r+\|u^*\|_{\E})\Big]\\
   & \le C_0 (C_1 + \|u^*\|_{L^\infty((0,T);\gamma_t\E)})\|\phi_R\|_{L^p((0,T))} \\
   & \quad + C_0 (r+\|u^*\|_{\E})\big( m_T + L(R)C_1r+L(R)C_1\|u^*-u_0\|_{\E}\big).
   \end{aligned}
   \end{equation}
   In the limit $T\to 0$, we obtain the following convergences:
   \begin{itemize}
     \item $m_T\to 0$, as $A(\cdot, u_0)$ is continuous,
     \item $\|\phi_R\|_{L^p((0,T))}\to 0$, as $\phi_R\in L^p((0,T_0))$,
     \item $\|u^*-u_0\|_{\E_T}\to 0$, as $u^*-u_0\in \E_{T_0}$,
     \item $\|u^*\|_{\E_T} \to 0$, as $u^*\in \E_{T_0}$.
   \end{itemize}
   First, choose  $r>0$ small enough such that
     \[ C_0 L(R)C_1r<\tfrac 18\]
    holds.  Then, choose  $T>0$ small enough such that the following inequalities hold:
     \begin{align*}
      \|u^*\|_{\E} & < r\\
      C_0 (C_1 + \|u^*\|_{L^\infty((0,T);\gamma_t\E)})\|\phi_R\|_{L^p((0,T))}& < \tfrac r2,\\
       C_0( m_T + L(R)C_1\|u^*-u_0\|_{\E}\big) & < \tfrac 18.
     \end{align*}
   Inserting this into  \eqref{4-8},  we obtain
   \[ \|\Phi(v)-u^*\|_{\E} \le \tfrac r2 + (r+r)(\tfrac 18+\tfrac18) =r,\]
   which shows that $\Phi(B_r)\subset B_r$.

   \textbf{(iv)} In the same way as in (iii), one sees that for sufficiently small $r>0$ and $T>0$ the inequality
   \[ \|\Phi(v)-\Phi(\bar v)\|_{\E} \le \tfrac 12 \|v-\bar v\|_{\E}\]
   holds for all $v,\bar v\in B_r$. Therefore, $\Phi\colon B_r\to B_r$ is a contraction, and with the Banach fixed point
   theorem (contraction mapping principle), there exists a unique fixed point   $u$ of $\Phi$. By definition of $\Phi$,
   its fixed points are exactly the solutions of the nonlinear equation  \eqref{4-1}, which finishes the proof.
\end{proof}

\begin{theorem}
  \label{4.9}
  Assume (A1) and (A2) to hold, and assume $A(t,v)\in \mreg(X_0)$ for all $t\in [0,T_0)$ with $T_0\in (0,\infty]$.
  Then for every $u_0\in\gamma_t\E$ there exists a unique maximal solution of \eqref{4-1} with maximal
   existence interval $[0,T^+(u_0))\subset [0,T_0)$. If $T^+(u_0)<T_0$ (i.e., if there is no global solution),
    then  $T^+(u_0)$ is characterized by each of the following conditions.
  \begin{enumerate}
    [(i)]
    \item $\lim_{t\nearrow T^+(u_0)} u(t)$ does not exist in $\gamma_t\E$,
    \item $\int_0^{T^+(u_0)} \big( \|u(t)\|_{X_1}^p + \|\partial_t u(t)\|_{X_0}^p\big) dt = \infty$.
  \end{enumerate}
\end{theorem}

\begin{proof}
  Assume  $u\in \E_T$ to be a local solution on the interval $(0,T)$. Then $u\in C([0,T];\gamma_t\E)$. Therefore,
  we can apply Theorem~\ref{4.8} in the interval $(T,T_0)$ with initial condition $u_1=u(T)\in\gamma_t\E$, and obtain an
  extension of $u$ to some interval $(0,T')$ with $T'>T$. Continuing in this way, we obtain a unique maximal solution
  which exists in some time interval $[0,T^+(u_0))$.

  If $\lim_{t\nearrow T^+(u_0)} u(t)\in \gamma_t\E$ exists, this can be taken as initial value at time $T^+(u_0)$. By
  the above arguments, we see that $u$ can be extended to a small time interval  $(T^+(u_0), T^+(u_0)+\epsilon)$,
  which is a contradiction to the maximality of  $T^+(u_0)$. Therefore,  $T^+(u_0)$ is characterized by condition (i).

  For each $T<T^+(u_0)$ we have, by definition of a solution,  $\int_0^T (\|u(t)\|_{X_1}^p +
  \|\partial_t u(t)\|_{X_0}^p)dt<\infty$. If this also holds for  $T=T^+(u_0)$, then $u\in \E_{T^+(u_0)}(X_1,X_0)
  \subset C([0,T^+(u_0)]; \gamma_t\E)$. Therefore,  $\lim_{t\nearrow T^+(u_0)} u(t)$ exists in $\gamma_t\E$ in
  contradiction to (i).
\end{proof}

As an application of the above theorems, we obtain a result on lower-order perturbation (the map $B$ in the
 following lemma) for linear non-autonomous problems.

\begin{lemma}
  \label{4.10}
  Let $A\in C([0,T], L(X_1,X_0))$ with $A(t)\in \mreg (X_0)\;(t\in [0,T])$, and let $B\in L^p((0,T); L(\gamma_t\E,X_0))$.
   Then the initial value problem
  \begin{align*}
    \partial_t u(t) - A(t) u(t) & = B(t) u(t) + f(t)\quad (t\in [0,T]),\\
    u(0) = u_0
  \end{align*}
  has for each $f\in\F_T$ and each $u_0\in\gamma_t\E$ a unique solution $u\in\E_T$.
\end{lemma}

\begin{proof}
  We set $A(t,u(t)) = A(t)$ and $F(t,u(t)) = B(t)u(t)+f(t)$. Obviously, the conditions  (A1) and (A2) are satisfied
  with $\phi_R(t) := \|B(t)\|_{L(\gamma_t\E, X_0)}$. The proof of Theorem~\ref{4.8} shows that the length of the
  existence interval only depends on  $u_0$ and the constants $L(R)$, $C_0$, $C_1$ and $\gamma_T$. Because of
  $A\in C([0,T], L(X_1,X_0))$ and the continuity of $A\mapsto \|(\partial_t+A)^{-1}\|_{L(\F,\E)} = C_0(A)$, all these
  constants can be chosen globally in the time interval $[0,T]$. Therefore, we have global existence of the solution.
\end{proof}

\subsection{Higher regularity}

We consider the same situation as in the last subsection and study the autonomous quasilinear differential equation
\begin{equation}
  \label{4-9}
  \begin{aligned}
  \partial_t u(t) - A(u(t)) u(t) &= F(u(t))\quad (t\in (0,T)),\\
  u(0) & = u_0.
  \end{aligned}
\end{equation}
Here, $T\in (0,\infty)$,  $u_0\in\gamma_t\E(X_1,X_0)$, $A\colon \gamma_t\E\to L(X_1,X_0)$ and $F\colon
\gamma_t\E\to\F$.

It is well known that parabolic equations are smoothing, and the solution is even -- in many applications --
real analytic. We start with a definition.

\begin{definition}
  \label{4.11}
  Let $X,Y$ be Banach spaces, $U\subset X$ open, and $T\colon U\to Y$ be a function. Then  $T$ is called real analytic
   if  for all $u_0\in U$ there exists an $r>0$  with $B(u_0,r)\subset U$ and
  \[ T(u) = \sum_{k=0}^\infty \frac{D^kT(u_0)}{k!}\, \underbrace{(u-u_0,\dots,u-u_0)}_{k-\text{times}}
  \quad (u\in B(u_0,r)).\]
  Here,  $D^kT(u_0) \in L(X\times\ldots\times X, F)$ denotes the  $k$-th Fr\'{e}chet derivative of  $T$ at  $u_0$.
   In this case, we write $T\in C^\omega(U,Y)$.
\end{definition}

The main step in the proof of smoothing properties for parabolic equations is the implicit function theorem
in Banach spaces.

\begin{theorem}[Implicit function theorem]
  \label{4.12}
  Let $X,Y,Z$ be Banach spaces, $U\subset X\times Y$ be open, and $T\in C^1(U,Z)$. Further, let $(x_0,y_0)\in U$
  with $T(x_0,y_0)=0$ and $D_y T((x_0,y_0))\in L_{\Isom}(Y,Z)$, where $D_y T$ stands for the Fr\'{e}chet derivative
  with respect to the second component.  Then there exist neighbourhoods $U_X$ of $x_0$ and $U_Y$ of $y_0$ with
  $U_X\times U_Y\subset U$ and a unique function $\psi\in C^1(U_X,U_Y)$ such that
  \[ T(x, \psi(x)) = 0 \quad (x\in U_X)\]
  and $\psi(x_0)=y_0$.  Therefore, the equation $T(x,y)=0$ is locally solvable with respect to  $y$. The function
  $\psi$ has the same regularity as $T$, i.e., if $T\in C^k(U,Z)$ for $k\in\N\cup\{\infty,\omega\}$, then also
  $\psi\in C^k(U_X, U_Y)$.
\end{theorem}

With the help of the implicit function theorem, one can prove smoothing properties with respect to the time variable.
As references, we mention \cite{Angenent90}, \cite{Pruess02}, Section~5, and \cite{Pruess-Simonett16}, Section~5.2.

\begin{theorem}
  \label{4.13}
  Let $k\in \N\cup\{\infty,\omega\}$, and let $A\in C^k(\gamma_t\E; L(X_1,X_0))$ and $F\in C^k(\gamma_t\E, X_0)$.
  Assume $u\in \E_T(X_1,X_0)$ to be a solution of \eqref{4-9}, and assume that $A(u(t))\in \mreg(X_0)$ for
  all $t\in [0,T]$. Then
  \[ t\mapsto t^j \partial_t^j u(t)\in W_p^1(J;X_0)\cap L^p(J;X_1) \]
  holds for all $j\in\N_0$ with $j\le k$. In particular,
  \[ u\in W_p^{k+1}((\epsilon,T); X_0)\cap W_p^k ((\epsilon,T); X_1)\]
  for every $\epsilon >0$ as well as
  \[ u\in C^k((0,T);\gamma_t\E)\cap C^{k+1-1/p}((0,T);X_0)\cap C^{k-1/p}((0,T);X_1).\]
  Here,  $C^{k+1-1/p}$ and $C^{k-1/p}$ stand for the H\"older spaces of order $k+1-1/p$ and $k-1/p$, respectively.
  If $k=\infty$, then $u\in C^\infty((0,T);X_1)$, and if $k=\omega$, then $u\in C^\omega((0,T);X_1)$.
\end{theorem}

\begin{proof}
  We fix  $\epsilon\in (0,1)$ and set $T(\epsilon):= \frac{T}{1+\epsilon}$. For $\lambda\in (1-\epsilon,1+\epsilon)$
  we define the function $u_\lambda\colon [0,T(\epsilon)]\to \gamma_t\E$ by $u_\lambda(t):= u(\lambda t)\;
  (t\in [0,T(\epsilon)])$. Then $\partial_t u_\lambda(t) = \lambda (\partial_t u)(\lambda t)$, and therefore
  \begin{align*}
   \partial_t u_\lambda (t) - \lambda A(u_\lambda(t)) u_\lambda(t) & = \lambda F(u_\lambda(t))\quad (t\in
   (0, T(\epsilon))),\\
   u_\lambda(0) & = u_0.
   \end{align*}
   Now consider the function
   \[ H\colon (1-\epsilon,1+\epsilon)\times \E_{T(\epsilon)}\to \F_{T(\epsilon)}\times\gamma_t\E\]
   defined by
   \[ H(\lambda, w)(t) := \begin{pmatrix}
    \partial_t w(t)-\lambda A(w(t)) w(t) - \lambda F(w(t))\\
     w(0)-u_0
     \end{pmatrix} \quad (t\in (0,T(\epsilon)))\]
   for $\lambda\in (1-\epsilon,1+\epsilon)$ and $w\in \E_{T(\epsilon)}$. As $A$ and $F$ are both of class $C^k$,
    the same holds for $H$. Moreover, $H(1,u)=0$ and
   \begin{align*}
    D_\lambda H(\lambda,w) & = \begin{pmatrix} -A(w)w - F(w)\\  0\end{pmatrix},\\
    D_w H(\lambda,w) h & = \begin{pmatrix} \partial_t h - \lambda A(w)h - \lambda A'(w) hw - \lambda F'(w)h\\ h(0)
    \end{pmatrix}
    \end{align*}
    for $h\in \E_{T(\epsilon)}$. Here $A'(u)$ stands for the  Fr\'echet derivative of $A$ at $u$.
    In particular, we obtain for $\lambda=1$ and $w=u$
    \[ D_w H(1,u)h = \begin{pmatrix} \partial_t h + A(u) h + A'(u) hu - F'(u)h\\ h(0)\end{pmatrix}.\]
For  $t\in [0,T(\epsilon)]$ and $v\in \gamma_t\E$, we define $B(t)v:= -A'(u(t))v u(t) - F'(u(t))v$. As $A\in
C^1(\gamma_t\E, L(X_1,X_0))$ and $F\in C^1(\gamma_t\E, X_0)$, we get $B\in L^p((0,T); L(\gamma_t\E,X_0))$. Therefore,
we can apply Lemma~\ref{4.10} (replacing $A(t)$ in this lemma by $A(u(t))$). Note that $t\mapsto A(u(t)) \in
C([0,T(\epsilon)], L(X_1,X_0))$ holds because of $t\mapsto u(t)\in C([0,T(\epsilon)];\gamma_t\E)$. By assumption,
$A(u(t)) \in \mreg(X_0)$ for every $t\in [0,T]$, and we can apply Lemma~\ref{4.10}.  This yields
\[ D_w H(1,u) \in L_{\Isom}(\E_{T(\epsilon)}  , \F_{T(\epsilon)} \times \gamma_t\E ).\]
Now the implicit function theorem, Theorem~\ref{4.12}, tells us that there exists a $\delta>0$ and a $C^k$-function
 $\psi\colon (1-\delta, 1+\delta)\to \E_{T(\epsilon)}$ with $H(\lambda, \psi(\lambda))=0\;(\lambda\in
  (1-\delta,1+\delta))$ and $\psi(1)=u$.

By definition of $H$ and the uniqueness of the solution, we obtain $\psi(\lambda) = u_\lambda$, i.e., $\lambda\mapsto
u_\lambda\in C^k((1-\delta,1+\delta), \E_{T(\epsilon)})$. Because of $\E_{T(\epsilon)} \subset C([0,T(\epsilon)],
\gamma_t\E)$, we obtain $\lambda\mapsto u_\lambda(t)=u(\lambda t)\in C^k((1-\delta,1+\delta),\gamma_t\E)$. But this
means $u\in C^k((0,T(\epsilon)), \gamma_t\E)$.

Now we use $\frac{\partial}{\partial\lambda} u_\lambda(t)|_{\lambda=1} = t\partial_t u(t)\;(t\in (0,T(\epsilon))$. As
$\psi\in C^k((1-\delta,1+\delta), \E_{T(\epsilon)}$, we get $t\mapsto t\partial_t u(t)\in \E_{T(\epsilon)}  $. An
iteration shows that $t\mapsto t^k \partial_t^k u(t) \in \E_{T(\epsilon)}$, and therefore
\[ u\in W_p^{k+1}((\delta, T(\epsilon)); X_0)\cap W_p^k ((\delta,T(\epsilon)); X_1)\]
for every $\delta>0$ and $\epsilon>0$. Now we apply  Sobolev's embedding theorem which tells us that
$W_p^k((\delta,T(\epsilon))\subset C^{k-1/p}([\delta,T(\epsilon)])$. With this we obtain, as  $\epsilon>0$ and
$\delta>0$ can be chosen arbitrary,
\[ u\in C^{k+1-1/p}((0,T);X_0)\cap C^{k-1/p}((0,T);X_1).\]
In the case $k=\infty$, we get $u\in C^\infty((0,T); X_1)$. If $k=\omega$, then the function $\psi$ is real analytic.
The above embeddings are linear and therefore real analytic, too, which yields  $u\in C^\omega((0,T), X_1)$.
\end{proof}

\begin{remark}
  \label{4.13a}
  This method of proof is known as parameter trick or method of Angenent \cite{Angenent90}. Note that the two main ingredients are the
  implicit function theorem in Banach spaces and the fact that $D_w H(1,u)$ is an isomorphism. The latter is exactly
  the maximal regularity of the linearization, and it can also be seen as one of the main ideas of the maximal
  regularity approach to show that the implicit function theorem can be applied to the nonlinear equation.
\end{remark}

As an example, we consider the quasilinear autonomous second order equation in  $\R^n$
\begin{equation}
  \label{4-10}
  \begin{aligned}
    \partial_t u(t,x) - \spur \big( a(u(t,x), \nabla u(t,x)) \nabla^2 u(t,x)\big) & = f(u(t,x),\nabla u(t,x))\\
    & \qquad\qquad ((t,x)\in (0,T)\times\R^n),\\
    u(0,x) & = u_0(x)
  \end{aligned}
\end{equation}

To solve the nonlinear problem, we need the following result from the linear theory, which can be shown by the
methods of Section~5.

\begin{lemma}
  \label{4.14}
  Let  $b\in \buc(\R^n;\sym)$ with $b(x)\ge cI_n\;(x\in\R^n)$ for some constant $c>0$. Define the operator $B$ by $D(B):=
  W_p^2(\R^n)\subset L^p(\R^n)$,
  \[ (Bu)(x) := \spur \big( b(x)\nabla^2 u(x)\big) = \sum_{i,j=1}^n b_{ij}(x) \partial_i \partial_j u(x)
  \quad (x\in\R^n,\, u\in D(B)).\]
  Then $B\in \mreg( L^p(\R^n))$.
\end{lemma}

For the nonlinear equation, we obtain the following result (see \cite{Pruess02}, Theorem~5.1).

\begin{theorem}
  \label{4.15}
  Let  $p\in (n+2,\infty)$ and  $k\in \N\cup\{\infty,\omega\}$. Assume that $a\in C^k(\R^{n+1}, \sym)$ and
  $f\in C^k(\R^{n+1},\R)$ with $f(0)=0$, and assume that for all  $(r,p)\in\R\times\R^n$ the matrix $a(r,p)$
  is positive definite. Then equation \eqref{4-10} has for all  $u_0\in W_p^{2-2/p}(\R^n)$ a unique maximal
  solution $u\in L^p((0,T^+); W_p^2(\R^n_+))\cap W_p^1((0,T^+); L^p(\R^n))$ in the existence interval $J=(0,T^+)$
  with $T^+=T^+(u_0)>0$. Moreover,
  \[ u\in C^k(J; W_p^{2-2/p}(\R^n))\cap C^{k+1-1/p}(J; L^p(\R^n))\cap C^{k-1/p}(J; W_p^2(\R^n)).\]
\end{theorem}

\begin{proof}
  For $X_0:=L^p(\R^n)$ and $X_1:= W_p^2(\R^n)$, the trace space is given by $\gamma_t\E(X_0,X_1) =
  (X_0,X_1)_{1-1/p,p} = W_p^{2-2/p}(\R^n)$. An application of Sobolev's embedding theorem yields
  \[ \gamma_t\E = W_p^{2-2/p}(\R^n) \subset C_0^1(\R^n) := \big\{u\in C^1(\R^n): \lim_{|x|\to\infty}
  |\partial^\alpha u(x)|=0\;(|\alpha|\le 1)\big\}.\]
  Now define the mappings $A\colon \gamma_t\E \to L(X_0,X_1)$ and $F\colon \gamma_t\E\to X_0$ by
  \begin{align*}
    (A(v) w)(x) & := \spur \big( a(v(x),\nabla v(x))\nabla^2 w(x)\big),\\
    (F(v))(x) & := f(v(x), \nabla v(x))
  \end{align*}
  for $x\in\R^n$, $v\in \gamma_t\E$, and $w\in W_p^2(\R^n)$.

  Let $v\in\gamma_t\E$. Because of  $v\in C_0^1(\R^n)$, the set $\{(v(x),\nabla v(x)):x\in\R^n\}\subset \R^{n+1}$
   is bounded. As $a$ is continuous by assumption, we see that
  \[ b_v := a(v(\cdot),\nabla v(\cdot)) \in \buc(\R^n)\]
  and  $b_v(x) \ge c_v I_n\;(x\in\R^n)$ with $c_v>0$. By Lemma~\ref{4.14}, we obtain  $A(v)\in \mreg(X_0)$ for
  all $v\in \gamma_t\E$.

  To show that assumptions (A1) and (A2) are satisfied, we use the fact that $a$ is a $C^1$-function and
  therefore Lipschitz on bounded sets. Therefore, we get for all $v,\bar v\in\gamma_t\E$ and $w\in X_1$
  with $\|v\|_{\gamma_t\E}\le R$, $\|\bar v\|_{\gamma_t\E}\le R$ the inequality
  \begin{align*}
    \| A(v)w - &A(\bar v)w\|_{L^p(\R^n)} = \big\| \spur\big( a(v,\nabla v) w - a(\bar v,
    \nabla\bar v)w\big)\big\|_{L^p(\R^n)} \\
    & \le C \big\| a(v,\nabla v)-a(\bar v,\nabla\bar v)\|_{L^\infty(\R^n;\R^{n\times n})} \| \nabla^2 w\|_{L^p(\R^n;
    \R^{n\times n})}\\
    & \le C L(R) \|v-\bar v\|_{C^1(\R^n)} \|w\|_{X_1} \\
    & \le C L(R) \|v-\bar v\|_{\gamma_t\E} \|w\|_{X_1}.
  \end{align*}
  This shows assumption  (A1) and, in particular, the continuity of  $A\colon \gamma_t\E \to L(X_0,X_1)$. Similiary,
  assumption (A2) can be shown. Here, we have to show the continuity of  $F\colon \gamma_t\E\to X_0$. For this we
  use the fact that  $F$ is a variant of the so-called  Nemyckii operators, i.e.,
  \[ F\colon W_p^{2-2/p}(\R^n)\to L^p(\R^n),\; F(v) := f(v(\cdot),\nabla v(\cdot))\quad (v\in W_p^{2-2/p}(\R^n)).\]
   For this, we also use $f(0)=0$. By known results on the Nemyckii operator, one obtains   $A\in C^k(\gamma_t\E,
   L(X_1,X_0))$ and $F\in C^k(\gamma_t\E, X_0)$. Therefore, all assumptions of Theorem~\ref{4.13} are satisfied,
   and we obtain higher regularity for the solution $u$ as stated in the theorem.
\end{proof}

\nocite{*}

\providecommand{\bysame}{\leavevmode\hbox to3em{\hrulefill}\thinspace} \providecommand{\href}[2]{#2}

\end{document}